\documentclass[11pt,A4paper,microtype]{gtpart}
\usepackage{hyperref}
\usepackage[english]{babel}
\usepackage[utf8]{inputenc}
\usepackage[T1]{fontenc}
\usepackage{amsmath,amsthm,amssymb,amsfonts}
\numberwithin{equation}{section}
\usepackage{enumerate}
\usepackage{faktor}
\usepackage{dsfont}
\usepackage{scalerel,stackengine}
\usepackage{textcomp}
\usepackage{graphicx}
\usepackage{extarrows}
\usepackage{epigraph}
\usepackage{graphicx}
\usepackage{sidecap}
\usepackage{indentfirst}
\usepackage{tikz}
\usepackage{tikz-cd}
\usetikzlibrary{shapes.misc,arrows,matrix,patterns,decorations.markings,positioning}
\tikzset{cross/.style={cross out, draw=black, minimum size=2*(#1-\pgflinewidth), inner sep=0pt, outer sep=0pt},
cross/.default={4.5pt}}
\usepackage{pgfplots}
\usepackage{caption}
\usepackage{float}
\usepackage{color}
\usepackage{epstopdf}
\usepackage{epsfig}
\usepackage{stmaryrd}
\usepackage{color}
\usepackage{geometry}
\usepackage{multicol}
\usepackage{verbatim}

\DeclareMathOperator{\Ker}{Ker }
\DeclareMathOperator{\Imm}{Im }

\DeclareMathOperator{\lk}{\ell k}
\DeclareMathOperator{\writhe}{wr}

\DeclareMathOperator{\Hom}{Hom }
\DeclareMathOperator{\Ann}{Ann }
\DeclareMathOperator{\Id}{Id}
\DeclareMathOperator{\vv}{\emph{\textbf v}}
\renewcommand{\geq}{\geqslant}
\renewcommand{\leq}{\leqslant} 
\renewcommand{\epsilon}{\varepsilon}
\newcommand{\N}{\mathbb{N}}
\newcommand{\F}{\mathbb{F}}
\newcommand{\X}{\mathbb{X}}
\newcommand{\OO}{\mathbb{O}}

\newcommand{\T}{\mathbb T}

\newcommand{\lkhov}{\llbracket}
\newcommand{\rkhov}{\rrbracket}

\newcommand{\cC}{cCFL}
\newcommand{\cH}{c\mathcal{HFL}}

\DeclareFontFamily{U}{mathx}{\hyphenchar\font45}
\DeclareFontShape{U}{mathx}{m}{n}{
      <5> <6> <7> <8> <9> <10>
      <10.95> <12> <14.4> <17.28> <20.74> <24.88>
      mathx10
      }{}
\DeclareSymbolFont{mathx}{U}{mathx}{m}{n}
\DeclareFontSubstitution{U}{mathx}{m}{n}
\DeclareMathAccent{\widecheck}{0}{mathx}{"71}
\DeclareMathAccent{\wideparen}{0}{mathx}{"75}

\newtheorem{teo}{Theorem}[section]
\newtheorem*{teo*}{Theorem}
\newtheorem{lemma}[teo]{Lemma}
\newtheorem{prop}[teo]{Proposition}
\newtheorem*{prop*}{Proposition}

\newtheorem{cor}[teo]{Corollary}

\usepackage{xpatch}
\makeatletter
\xpatchcmd{\@thm}{\thm@headpunct{.}}{\thm@headpunct{}}{}{}
\makeatother

\theoremstyle{definition}
\newtheorem{remark}[teo]{Remark}
\newtheorem{ex}[teo]{Example}

\pgfplotsset{compat=1.13}
\begin{document}
\title{Locally equivalent Floer complexes and unoriented link cobordisms}

%
\author{Alberto Cavallo}
\givenname{Alberto}
\surname{Cavallo}
\address{Max Planck Institute for Mathematics, Bonn 53111, Germany}
\email{cavallo@mpim-bonn.mpg.de}
\urladdr{https://sites.google.com/view/albertocavallomath}

%
%
%
%
%



\arxivreference{arXiv:1911.03659}

\begin{abstract}
 We show that the local equivalence class of the collapsed link Floer complex $\cC^\infty(L)$, together with many $\Upsilon$-type invariants extracted from this group, is a concordance invariant of links. In particular, we define a version of the invariants $\Upsilon_L(t)$ and $\nu^+(L)$ when $L$ is a link and we prove that they give a lower bound for the slice genus $g_4(L)$.
 
 Furthermore, in the last section of the paper we study the homology group $HFL'(L)$ and its behaviour under unoriented cobordisms. We obtain that a normalized version of the $\upsilon$-set, introduced by Ozsv\'ath, Stipsicz and Szab\'o, produces a lower bound for the 4-dimensional smooth crosscap number $\gamma_4(L)$. 
\end{abstract}
\maketitle

\section{Introduction}
In \cite{Hom} Hom introduced an equivalence relation on the knot Floer complex $CFK^\infty(K)$ called stable equivalence. Namely, we say that two knots are \emph{stably equivalent} if and only if their chain complexes become filtered chain homotopy equivalent after adding some acyclic complexes. A very important result in \cite{Hom} is that when $K_1$ is concordant to $K_2$ then the complex $CFK^\infty(K_1)$ is stably equivalent to $K_2$, which made possible to prove that many knot invariants coming from $CFK^\infty(K)$ are indeed concordance invariants, see \cite{Antonio,Allen,HomWu,Livingston} for some examples.

Another relation on knot Floer chain complexes was given by Zemke in \cite{Zemke}: two knots $K_1$ and $K_2$ are called \emph{locally equivalent} if there exist two maps $f:CFK^\infty(K_1)\rightarrow CFK^\infty(K_2)$ and $g:CFK^\infty(K_2)\rightarrow CFK^\infty(K_1)$ which preserve the filtrations (both the Alexander and algebraic filtration) and induce filtered isomorphisms in homology. Even though those two relations appear to be very different from their definition, we can actually show that they coincide. We recall that this theorem was proved in the involutive setting by Hendricks and Hom (\cite{HomHe}).
\begin{teo}
\label{teo:equivalence}
 Consider two knots $K_1$ and $K_2$ in $S^3$. Then $CFK^\infty(K_1)$ is locally equivalent to $CFK^\infty(K_2)$ if and only if such two chain complexes are stably equivalent.
\end{teo}
For the purpose of this paper, the local equivalence relation has the advantage that it can be used in the same way for links. Let us consider the chain complex $\cC^\infty(L)$, defined from $CFL^-(L)$ by collapsing the variables $U_1,...,U_n$ to $U$ and taking the tensor product with $\F[U,U^{-1}]$, where here $\F$ always denotes the field with two elements, see \cite{Book,OSlinks}. We equip $\cC^\infty(L)$ with a filtration $\mathcal F$, obtained from the algebraic filtration and the (collapsed) Alexander filtration; such $\mathcal F$ descends to homology so that we can define the filtered group $\mathcal F\cH^\infty(L)$. 
Based on an intuition of Alfieri in \cite{Antonio}, we consider $\mathcal F$ as indexed by some particular subsets $S$ of the plane, which he calls south-west regions, satisfying the property that if $(\overline x,\overline y)\in S$ then each $(x,y)$ such that $x\leq\overline x$ and $y\leq\overline y$ is in $S$; a more precise definition is given later in Subsection \ref{subsection:Alexander}. 

We recall that two $n$-component links are \emph{concordant} if there is a cobordism between them consisting of $n$ disjoint annuli. Then the local equivalence class of $\cC^\infty(L)$ and the filtered homology group $\cH^\infty(L)$ are a concordance invariant in the following sense.
\begin{teo}
 \label{teo:concordance}
 Suppose that $L_1$ is concordant to $L_2$; then there are chain maps \\ $\begin{tikzcd}
 \cC^\infty(L_1) \ar[r,shift left=.75ex]
 &
  \ar[l,yshift=-1ex,shift right=.75ex]
 &
 \cC^\infty(L_2)
 \end{tikzcd}$, which preserve $\mathcal F$ and induce an $\mathcal F$-filtered isomorphism in homology. In particular, the restrictions of such isomorphisms give identifications
 \[\mathcal F^S\cH^\infty_d(L_1)\cong_\F\mathcal F^S\cH^\infty_d(L_2)\] for every $d\in\mathbb Z$ and $S$ south-west region of $\mathbb Z^2$.
\end{teo}
The strategy of the proof of this result consists in decomposing a concordance into standard pieces and then a careful usage of the maps introduced by Sarkar, see \cite{Sarkar}, on grid diagrams. In fact, starting from Sarkar's work, we can construct maps induced by some specific cobordisms. Some of these maps were already used by the author in \cite{Cavallo}.  
\begin{remark}
 \label{remark:remark}
 In \cite{Zemke2} Zemke, using different techniques, also defined maps induced by (decorated) link cobordisms, which conjecturally coincide with the ones presented in this paper. We can use such maps to give another proof of Theorem \ref{teo:concordance}: namely, according to \cite[Theorems A and C]{Zemke2} every link concordance induces a graded isomorphism in link Floer homology; while the fact that the $\mathcal F$-filtration is preserved follows from \cite[Theorem 1.4]{Zemke3}. This argument is similar to the one in \cite{Zemke}, where Zemke proved a version of Theorem \ref{teo:concordance} for knots.
\end{remark}
Theorem \ref{teo:concordance} allows us to define some numerical concordance invariants for links; including a generalization of Alfieri's $\Upsilon_S$ \cite{Antonio}, the $\nu^+$-invariant of Hom and Wu \cite{HomWu} (see also \cite{Rasmussen}) and the secondary upsilons, defined by Kim and Livingston \cite{Livingston}. 
We briefly describe how to extract some of these invariants. 

Write $\cC^\infty_*(L)$ for the filtered chain homotopy type of the link Floer complex of $L$. Once we fix a filtered basis, we can represent such model complex in the plane $(j,A)$, where $j$ and $A$ represent the minimal algebraic and Alexander filtration level respectively and $*$ the Maslov grading of each generator. We use the fact that $\dim_{\F}\mathcal F^{\{j\leq0\}}\cH^\infty_0(L)=1$, see Theorem \ref{teo:dim}, and then compute how far we can shift the region $S$ while being able to find a generator for such homology class in $\cC^\infty_0(L)$. In this way, given a south-west region $S$, we associate a real number to it that we call $\Upsilon_S(L)$; the complete definition can be found in Subsection \ref{subsection:Alexander}.

In the case of knots $\Upsilon_S(L)$ is a normalization of the invariant of Alfieri in \cite{Antonio}. This choice was done because, say $A_t$ is the region of the plane consisting of the pairs $(j,A)$ with $At+j(2-t)\leq0$, we have that \[\Upsilon_{A_t}(K)=\Upsilon_K(t)\quad\text{ for }\quad t\in[0,2]\] and the latter is the $\Upsilon$-function of Ozsv\'ath, Stipsicz and Szab\'o, see \cite{OSSz}. 

Moreover, since we also have that there is a unique homology class in $\mathcal F^{\{j\leq0\}}\cH^\infty_{1-n}(L)\setminus\mathcal F^{\{j\leq-1\}}\cH^\infty_{1-n}(L)$, the same procedure allows us to define another family of invariants which we call $\Upsilon_S^*(L)$. Clearly, for knots we have that $\Upsilon_S(K)=\Upsilon_S^*(K)$ for every $S$.
The following proposition summarizes some of the main properties of $\Upsilon_S(L)$.
\begin{prop}
 \label{prop:properties}
 Suppose that $L,L_1,L_2$ are links in $S^3$ and $L$ has $n$ components. Then we have that $\Upsilon_S(L)$ and $\Upsilon_S^*(L)$ are concordance invariants and
 \begin{enumerate}
     \item \[\tau(L)=-\Upsilon_L'(0)\quad\text{ and }\quad\tau^*(L)=-(\Upsilon^*_L)'(0)\] where the invariants $\tau(L)$ and $\tau^*(L)$ are defined in \emph{\cite{Cavallo}};
     \item \[\Upsilon_S(L)=\Upsilon_{-S}(L)\quad\text{ and }\quad\Upsilon^*_S(L)=\Upsilon^*_{-S}(L)\quad\text{for any}\quad S\] where $-S$ is the region obtained by reflecting $S$ along the line $\{j-A=0\}$;
     \item \[\Upsilon_S(L)=\Upsilon_S(-L)\quad\text{ and }\quad\Upsilon^*_S(L)=\Upsilon^*_S(-L)\quad\text{for any}\quad S\] where $-L$ is the reverse of $L$;
     \item \[\Upsilon_S(L^*)=-\Upsilon_{\overline{\iota S}}^*(L)\] where $L^*$ is the mirror image of $L$ and $\overline{\iota S}$ is the topological closure of the complement of the region obtained by reflecting $S$ using the central symmetry of $\mathbb R^2$ at the origin;
     \item \[\Upsilon_{L_1\#L_2}(t)=\Upsilon_{L_1}(t)+\Upsilon_{L_2}(t)\quad\text{ and }\quad\Upsilon^*_{L_1\#L_2}(t)=\Upsilon^*_{L_1}(t)+\Upsilon^*_{L_2}(t)\quad\text{ for }t\in[0,2]\] where $L_1\#L_2$ is a connected sum of $L_1$ and $L_2$;
     \item \[\Upsilon_L(t)=\dfrac{1-n+\sigma(L)}{2}\cdot t\quad\text{ and }\quad\Upsilon_L^*(t)=\dfrac{n-1+\sigma(L)}{2}\cdot t\quad\text{ for }t\in[0,1]\] whenever $L$ is quasi-alternating and $\sigma(L)$ is the signature of the link as in \emph{\cite{GL}}.
 \end{enumerate}
\end{prop}
We prove that each $\Upsilon_S(L)$ gives a lower bound for the slice genus $g_4(L)$, which as usual is defined as the minimum genus of a compact, oriented surface $\Sigma$ properly embedded in $D^4$ such that $\partial\Sigma=L$. We recall that, since we can add tubes between surfaces in $D^4$ without increasing the genus, we can suppose that such $\Sigma$ is also connected. Moreover, in Subsection \ref{subsection:slice} we define the notion of distance $h_S(m)$ from the point $(0,m)$ to the centered south-west region $S$, where centered means that $(0,0)\in\partial S$; therefore, we have the following result.
\begin{teo}
 \label{teo:slice_genus}
 If $L$ is a link in $S^3$ with $n$ components then \[-\Upsilon_S(L)\leq h_{\pm S}(g_4(L)+n-1)\quad\text{ and }\quad-\Upsilon^*_S(L)\leq h_{\pm S}(g_4(L))\] for every centered south-west region $S$ of $\mathbb R^2$.
 In particular, for the classic $\Upsilon$-functions one has \[-\Upsilon_L(t)\leq t(g_4(L)+n-1)\quad\text{ and }\quad-\Upsilon^*_L(t)\leq t\cdot g_4(L)\] for every $t\in[0,1]$.
\end{teo}
Now let us consider the south-west regions $V_k$ for $k\geq0$, defined as the subset of the plane consisting of the pairs $(j,A)$ such that $j\leq0$ and $A\leq k$. We can now define the invariant $\nu^+(L)$ as the minimum $k$ such that $-2\cdot V_L(k):=\Upsilon_{V_k}(L)=0$. The author gave an equivalent definition of $\nu^+(L)$ in \cite[Section 4]{Cavallo}; although, in the latter paper the invariant was denoted by $\nu(L)$ and the concordance invariance was not proven.
\begin{prop}
 Suppose that $L,L_1,L_2$ are links in $S^3$ and $L$ has $n$ components. Then we have that $\nu^+(L)$ is a concordance invariant and \[0\leq\tau(L)\leq\nu^+(L)\leq g_4(L)+n-1\quad\text{ and }\quad\nu^+(L_1\#L_2)\leq\nu^+(L_1)+\nu^+(L_2)\:.\]
\end{prop}
In \cite{Unoriented} Ozsv\'ath, Stipsicz and Szab\'o introduced the homology group $HFL'(L)$ that they called unoriented link Floer homology group. From $HFL'(L)$ they define the $\upsilon$-set of $L$ which is a set of $2^{n-1}$ integers and is an isotopy invariant of unoriented links after a suitable normalization.

Moreover, for knots they showed that $\upsilon(K)$, that coincides with $\Upsilon_K(1)$ and is the only element of the $\upsilon$-set in this case, gives a lower bound for the \emph{$4$-dimensional smooth crosscap number} $\gamma_4(K)$, which is the minimum first Betti number of a compact surface $F$ properly embedded in $D^4$ and such that $\partial F=L$. Note that this time $F$ is not necessarily orientable (and always non-oriented). 

Starting from these results, in this paper we consider a slightly different version of $HFL'(L)$ and we prove that it is an unoriented concordance invariant. 
Since it shares many information with the original group and we only use this new version, we denote it in the same way.

We say that a collection of $n$ disjoint annuli $\Sigma$ is an \emph{unoriented concordance} between $L_1$ and $L_2$, which are $n$-component links, if $\Sigma$ is a concordance between $L'_1$ and $L_2'$, obtained by changing the orientation of some components on $L_1$ and $L_2$ respectively.
\begin{teo}
 \label{teo:unoriented_concordance}
 The group $HFL'(L_1)\left\llbracket\frac{\sigma(L_1)}{2}\right\rrbracket$ is $j$-filtered isomorphic to $HFL'(L_2)\left\llbracket\frac{\sigma(L_2)}{2}\right\rrbracket$ whenever $L_1$ is unoriented concordant to $L_2$. 
\end{teo}
From Theorem \ref{teo:unoriented_concordance} we obtain that the wideness of the $\upsilon$-set $|\upsilon_{\max}(L)-\upsilon_{\min}(L)|$ and the numbers $\upsilon_{\max}(L)-\frac{\sigma(L)}{2}$ and $\upsilon_{\min}(L)-\frac{\sigma(L)}{2}$ are unoriented concordance invariants of $L$. Using the same techniques in Subsection \ref{subsection:slice}, we show that such invariants give lower bounds for $\gamma_4^{(k)}(L)$, a version of the $4$-dimensional smooth crosscap number for links. In fact, we say that $\gamma_4^{(k)}(L)$ is defined as the minimum first Betti number of a compact surface, properly embedded in $D^4$, which has $k$ connected components and is bounded by $L$.
\begin{teo}
 \label{teo:wideness}
 Say the $n$-component link $L$ in $S^3$ bounds a compact, unoriented surface $F$ with $k$ connected components and properly embedded in $D^4$. Then we have that \[\left|k-1-\upsilon_{\max}(L)+\upsilon_{\min}(L)\right|\leq\gamma_4^{(k)}(L)\:.\]
\end{teo}
A corollary of this theorem is the following result, which was already proved in a different way by Donald and Owens in \cite{DO}.
\begin{cor}
 \label{cor:DO}
 Every quasi-alternating link $L$ can bound an unoriented, compact surface $F$, properly embedded in $D^4$, only when the Euler characteristic $\chi(F)$ is at most equal to one.
\end{cor}
Theorem \ref{teo:wideness} gives a bound that involves the wideness of $\upsilon(L)$. We give other inequalities in the following theorem.
\begin{teo}
 \label{teo:unoriented_bound}
 Consider an $n$-component link $L$ in $S^3$ which bounds a compact, unoriented surface $F$ with $k$ connected components and properly embedded in $D^4$. Then we have that \[\left|\upsilon_{\max}(L)-\dfrac{\sigma(L)}{2}-\dfrac{k-n}{2}\right|\leq\gamma_4^{(k)}(L)\quad\text{ and }\quad\left|\upsilon_{\min}(L)-\dfrac{\sigma(L)}{2}-\dfrac{2-k-n}{2}\right|\leq\gamma_4^{(k)}(L)\:.\] In particular, when $k=n$ one has \[\left|\upsilon_{\max}(L)-\dfrac{\sigma(L)}{2}\right|\leq\gamma_4^{(n)}(L)\quad\text{ and }\quad\left|\upsilon_{\min}(L)-\dfrac{\sigma(L)}{2}+n-1\right|\leq\gamma_4^{(n)}(L)\] and when $k=1$ one has \[\left|\upsilon_i(L)-\dfrac{\sigma(L)}{2}-\dfrac{1-n}{2}\right|\leq\gamma_4^{(1)}(L)\] for every $\upsilon_i(L)$ in the $\upsilon$-set of $L$.
\end{teo}
We apply this result to the family of links $L_n=T_{2,4}^*\#T_{3,4}^{\#n}$; namely, the connected sum of the mirror of the torus link $T_{2,4}$ and $n$ torus knots $T_{3,4}$. In particular, we show that $\{L_n\}$ for $n\geq0$ is a family of $2$-component links such that $\gamma_4^{(1)}$ is arbitrarily large.
\begin{cor}
 \label{cor:app}
 Given the link $L_n=T_{2,4}^*\#T_{3,4}^{\#n}$; we have that $\gamma_4^{(2)}(L_n)=n+1$ and $\gamma_4^{(1)}(L_n)\geq n$ for every $n\geq0$.
\end{cor}
The paper is organized as follows: in Section \ref{section:two} we summarize the construction of the link Floer complex $\cC^\infty(L)$ and we describe how to define the filtered homology group $\mathcal F^S\cH^\infty(L)$ and the invariant $\Upsilon_S(L)$. Moreover, we prove the equivalence between stable and local equivalence of knot Floer chain complexes stated in Theorem \ref{teo:equivalence}. In Section \ref{section:concordance} we prove the concordance invariance of $\cC^\infty(L)$. In Section \ref{section:four} we define the other $\Upsilon$-type invariants and we prove some of their properties, including Proposition \ref{prop:properties}. We also give the proof of Theorem \ref{teo:slice_genus}, which describes our bound for the slice genus. Finally, in Section \ref{section:five} we introduce the group $HFL'(L)$ and the $\upsilon$-set of $L$, showing that they give unoriented concordance invariants. Moreover, we study their behaviour under unoriented cobordisms and we prove the lower bounds for $\gamma_4^{(k)}(L)$. 

\paragraph*{Acknowledgements:}
The author would like to thank Antonio Alfieri and Andr\'as Stipsicz for their many conversations about the $\Upsilon$-invariant; and Kouki Sato for his observations. The alternative argument to prove Theorem \ref{teo:concordance}, appearing in Remark \ref{remark:remark}, was communicated by Ian Zemke, to whom many thanks are due for his interest and help.
We also thank the referee for his many corrections.

The author has a post-doctoral fellowship at the Max Planck Institute for Mathematics in Bonn.

\section{Link Floer homology}
\label{section:two}
\subsection{Chain complex and homology}
\label{subsection:complex}
Throughout the paper we assume that the reader is familiar with the construction of the link Floer homology chain complexes, both when links are represented with multi-pointed Heegaard diagrams \cite{OS,OSknots,OSlinks} or grids \cite{MOST,Book}.  
We only recall the main features.

Let us consider $\mathcal D=(\Sigma,\alpha,\beta,\textbf w,\textbf z)$ a multi-pointed Heegaard diagram for an oriented $n$-component link $L$ in $S^3$. The chain complex $\cC^\infty(\mathcal D)$ is the free $\F[U,U^{-1}]$-module over the intersection points $\T=\T_\alpha\cap\T_\beta$ in the symmetric power $\text{Sym}(\Sigma,\alpha,\beta)$, see \cite{OSknots,OSlinks}, where $\F $ is the field with two elements and $\textbf w$ and $\textbf z$ are two $n$-tuples of basepoints in $\Sigma$, see \cite{OSlinks}. 
The differential $\partial^-$ is defined by counting pseudo-holomorphic curves on some special (\cite{OSlinks}) domains in $\text{Sym}(\Sigma,\alpha,\beta)$ with Maslov index $\mu$ equal to one (\cite{Lipshitz,OS}); denote the set of such domains with $\pi_2$, then for every intersection point $x$ we can write
\[\partial^-x=\sum_{y\in\T}\sum_{\substack{\phi\in\pi_2(x,y) \\ \mu(\phi)=1}}m(\phi)\cdot U^{n_{\textbf w}(\phi)}y\:,\] where $m(\phi)\in\F$ depends also on the choice of an almost-complex structure on $\text{Sym}(\Sigma,\alpha,\beta)$ and $0\leq n_{\textbf w}(\phi)=n_{w_1}(\phi)+...+n_{w_n}(\phi)$ is the multiplicity of the basepoints $\textbf w$ in $\phi$. Moreover, we say that \[\partial^-(U^{\pm 1}p)=U^{\pm 1}\cdot\partial^-p\] for any $x\in\T$ and $p\in\cC^\infty(\mathcal D)$.

For every $x\in\T$ we can assign an absolute $\mathbb Z$-grading, called \emph{Maslov grading} \cite{OSlinks}, which is denoted by $M(x)$ and can be extended to the whole complex by taking \[M(U^{\pm}p)=M(p)\mp 2\] for any $p$ homogeneous.
We then have that \[\cC^\infty(\mathcal D)=\bigoplus_{d\in\mathbb Z}\cC^\infty_d(\mathcal D)\] as $\F$-vector spaces; moreover, there is a map \[\partial^-_d:\cC^\infty_d(\mathcal D)\longrightarrow\cC^\infty_{d-1}(\mathcal D)\] for any $d\in\mathbb Z$.

The chain complex $\cC^\infty(\mathcal D)$ comes with a natural increasing filtration, usually denoted as the \emph{algebraic filtration} $j$, defined as follows 
\[j^t\cC^\infty(\mathcal D)=U^{-t}\cdot\cC^-(\mathcal D),\]  where $\cC^-(\mathcal D)$ is the free $\F[U]$-module over $\T$. It is easy to check that the differential $\partial^-$ respects $j$.

We define the homology group \[\cH^\infty(L)=\bigoplus_{d\in\mathbb Z}\cH_d^\infty(L)=\bigoplus_{d\in\mathbb Z}\dfrac{\Ker \partial^-_d}{\Imm\partial^-_{d+1}}\:.\]
Since the Maslov grading and the differential only depend on $\textbf w$, we have that such a group, together with the algebraic filtration, is isomorphic to $HF^\infty(S^3,n)\cong\F[U,U^{-1}]^{2^{n-1}}$, where the $n$ denotes the number of basepoints in the Heegaard diagram. The filtration $j$ descends to homology in the following way.
Say $\pi_d:\Ker\partial^-_d\rightarrow\cH^\infty_d(L)$ is the quotient map; then \[j^t\cH^\infty_d(L)=\pi_d\left(\Ker\partial^-_d\cap j^t\cC^\infty(\mathcal D)\right)\:,\] which is an $\F$-subspace of $\cH^\infty_d(L)$. More specifically we have the following theorem.
\begin{teo}
 \label{teo:dim}
 Say the link $L$ has $n$ components. Then we have that
 \[\dfrac{j^{\frac{d+k}{2}}\cH^\infty_d(L)}{j^{\frac{d+k}{2}-1}\cH^\infty_d(L)}\cong_\F\F^{\binom{n-1}{k}}\] whenever $d\equiv k$ \emph{mod} $2$ and $0\leq k\leq n-1$. It is zero otherwise.
\end{teo}
\begin{proof}
 From \cite{OSlinks} we know that $HF^-(S^3,n)$ has $2^{n-1}$ generators such that exactly $\binom{n-1}{k}$ of them have Maslov grading $-k$. Since \[HF_*^\infty(S^3,n)\cong_{\F[U,U^{-1}]}HF^-_*(S^3,n)\otimes_{\F[U]}\F[U,U^{-1}]\:,\] then one has 
 \[\cH^\infty_d(L)\cong_\F\left\{\begin{aligned}
   &\F^{2^{n-2}}\quad\text{if }\quad n\geq 2 \\
   &\F\hspace{0.7cm}\quad\text{if }\quad n=1\quad\text{ and }\quad d\quad\text{is even} \\
   &\{0\}\hspace{0.35cm}\quad\text{if }\quad n=1\quad\text{ and }\quad d\quad\text{is odd}
  \end{aligned}\right.\] and this determines the Maslov gradings.
  
 Now we want to compute the filtration $j$. We note that all the generators of $HF^-(S^3,n)$ have minimal $j$-filtration level zero. Hence, the statement is true for $j^0$; in fact if we substitute $d=-k$ in then we obtain the right distribution of the Maslov gradings. 
 \begin{figure}[t]
 \centering
 \begin{tikzpicture}[scale =.65]
        \coordinate (Origin)   at (0,0);
        \coordinate (XAxisMin) at (-5,0);
        \coordinate (XAxisMax) at (5,0);
        \coordinate (YAxisMin) at (0,-5);
        \coordinate (YAxisMax) at (0,5);
        \draw [thick, black,-latex] (XAxisMin) -- (XAxisMax);
        \draw [thick, black,-latex] (YAxisMin) -- (YAxisMax);
        \draw [thick, black,-latex] (-5,1) -- (5,1);
        \draw [thick, black,-latex] (1,-5) -- (1,5);
        \clip (-4.9,-4.9) rectangle (5.99,5.99);
        \draw[style=help lines] (-5,-5) grid[step=1cm] (4.95,4.95);
        
        \foreach \x in {-5,...,4}
                    \node[draw=none,fill=none] at (\x+0.5,-4.5) {\x};
    	\foreach \y in {-5,...,4}     		
        			\node[draw=none,fill=none] at (-4.5,\y+0.5) {\y};
        
        \node[] at (-0.5,5.5) {$M$};
        \node[] at (5.5,-0.5) {$j$};
        
        \node[draw,circle,fill,scale=0.4] at (0.5,0.5) {};
        \node[draw,circle,fill,scale=0.4] at (0.5,-0.5) {};
        \node[draw,circle,fill,scale=0.4] at (1.5,2.5) {};
        \node[draw,circle,fill,scale=0.4] at (1.5,1.5) {};
        \node[draw,circle,fill,scale=0.4] at (2.5,4.5) {};
        \node[draw,circle,fill,scale=0.4] at (2.5,3.5) {};
        \node[draw,circle,fill,scale=0.4] at (-0.5,-1.5) {};
        \node[draw,circle,fill,scale=0.4] at (-0.5,-2.5) {};
        \node[draw,circle,fill,scale=0.4] at (-1.5,-3.5) {};
 \end{tikzpicture}
 \hspace{.5cm}
 \begin{tikzpicture}[scale =.65]
        \coordinate (Origin)   at (0,0);
        \coordinate (XAxisMin) at (-5,0);
        \coordinate (XAxisMax) at (5,0);
        \coordinate (YAxisMin) at (0,-5);
        \coordinate (YAxisMax) at (0,5);
        \draw [thick, black,-latex] (XAxisMin) -- (XAxisMax);
        \draw [thick, black,-latex] (YAxisMin) -- (YAxisMax);
        \draw [thick, black,-latex] (-5,1) -- (5,1);
        \draw [thick, black,-latex] (1,-5) -- (1,5);
        \clip (-4.9,-4.9) rectangle (5.99,5.99);
        \draw[style=help lines] (-5,-5) grid[step=1cm] (4.95,4.95);
        
        \foreach \x in {-5,...,4}
                    \node[draw=none,fill=none] at (\x+0.5,-4.5) {\x};
    	\foreach \y in {-5,...,4}     		
        			\node[draw=none,fill=none] at (-4.5,\y+0.5) {\y};
        
        \node[] at (-0.5,5.5) {$M$};
        \node[] at (5.5,-0.5) {$j$};
        
        \node[draw,circle,fill,scale=0.4] at (0.5,0.5) {};
        \node[draw,circle,fill,scale=0.4] at (0.5-1/6,-0.5-1/6) {};
         \node[draw,circle,fill,scale=0.4] at (0.5+1/6,-0.5+1/6) {};
        \node[draw,circle,fill,scale=0.4] at (0.5,-1.5) {};
        \node[draw,circle,fill,scale=0.4] at (1.5,2.5) {};
        \node[draw,circle,fill,scale=0.4] at (1.5-1/6,1.5-1/6) {};
         \node[draw,circle,fill,scale=0.4] at (1.5+1/6,1.5+1/6) {}; 
        \node[draw,circle,fill,scale=0.4] at (1.5,0.5) {};
        \node[draw,circle,fill,scale=0.4] at (2.5,4.5) {};
        \node[draw,circle,fill,scale=0.4] at (2.5-1/6,3.5-1/6) {};
         \node[draw,circle,fill,scale=0.4] at (2.5+1/6,3.5+1/6) {}; 
        \node[draw,circle,fill,scale=0.4] at (2.5,2.5) {};
        \node[draw,circle,fill,scale=0.4] at (3.5,4.5) {};
        \node[draw,circle,fill,scale=0.4] at (-0.5,-1.5) {};
        \node[draw,circle,fill,scale=0.4] at (-0.5-1/6,-2.5-1/6) {};
         \node[draw,circle,fill,scale=0.4] at (-0.5+1/6,-2.5+1/6) {};
        \node[draw,circle,fill,scale=0.4] at (-0.5,-3.5) {};
        \node[draw,circle,fill,scale=0.4] at (-1.5,-3.5) {};
 \end{tikzpicture}
 \caption{Maslov gradings and algebraic filtration for $2$- (left) and $3$-component links (right). The algebraic level $j$ is on the $x$-axis and the Maslov grading on the $y$-axis.}
 \label{Homology2-3}
 \end{figure}
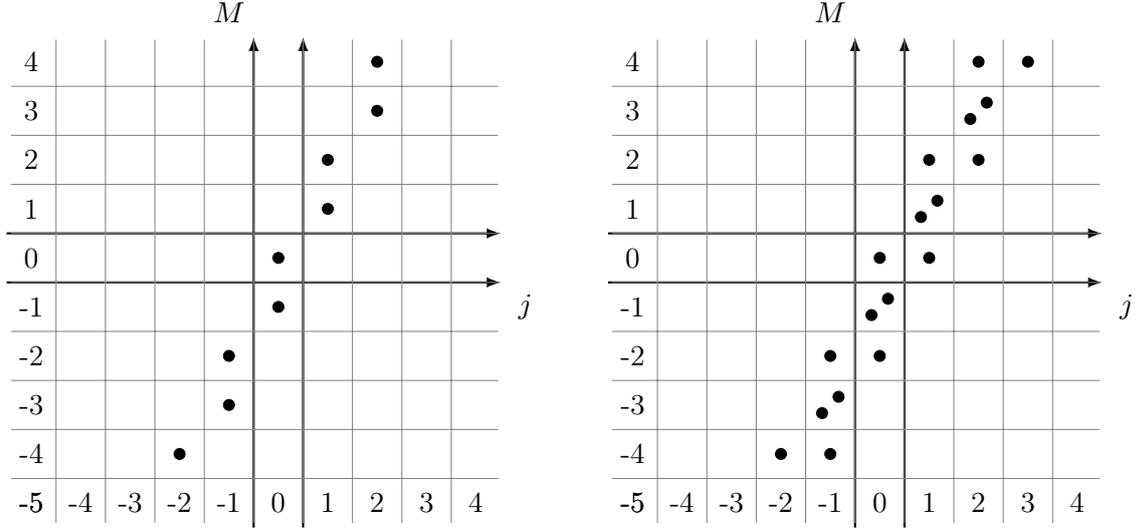
 At this point, in order to prove the theorem, we only need to observe that the multiplication by $U^{\pm 1}$ drops the minimal level of the algebraic filtration by $\pm 1$ and the Maslov grading by $\pm 2$. 
\end{proof}
Figure \ref{Homology2-3} shows the distribution of the Maslov grading and the minimal $j$-level for two- and three-component links.
 
\subsection{The Alexander and the \texorpdfstring{$\mathcal F$}{F} filtrations}
\label{subsection:Alexander}
In the same way as the Maslov grading, we can assign to every $x\in\T$ another absolute $\mathbb Z$-grading: the \emph{Alexander grading} $A(x)$, which also is extended to $\cC^\infty(\mathcal D)$ by taking \[A(U^{\pm 1}p)=A(p)\mp 1\] for any $p$ homogeneous.
We recall that in \cite{OSlinks} the Alexander grading is introduced as a multi-grading $\textbf A(x)=(A_1(x),...,A_n(x))$; in this paper we denote $A(x):=A_1(x)+...+A_n(x)$. 

In this case, the differential $\partial^-$ does not preserve $A(x)$; for this reason we introduce the \emph{Alexander filtration}. Let us consider the $\F[U]$-subspace $\mathcal A^s\cC^\infty(\mathcal D)$ generated by all the elements $p$ in $\cC^\infty(\mathcal D)$ such that $A(p)\leq s$. 
The $\mathcal A$-filtration is an increasing filtration like $j$ and it is such that
\begin{equation}\{0\}
\label{eq:filtration}
\cong\mathcal A^{\underline s}\cC^\infty_d(\mathcal D)\subset...\subset\mathcal A^{\overline s}\cC^\infty_d(\mathcal D)=\cC^\infty_d(\mathcal D)\:,\end{equation} which follows from \cite{OSlinks}; moreover, it is again easy to show that it is preserved by $\partial^-$. Note that $\overline s$ and $\underline s$ depend on $d$.

We define $\mathcal F$ for now as a double-increasing filtration. More specifically, we say that \[\mathcal F^{t,s}\cC^\infty(\mathcal D)=j^t\cC^\infty(\mathcal D)\cap\mathcal A^s\cC^\infty(\mathcal D)\] and clearly $\partial^-$ also respects $\mathcal F$.
We now extend the $\mathcal F$-filtration on the homology group, in the way that it is indexed by south-west regions of the lattice $\mathbb Z^2$ (resp. the plane $\mathbb R^2$), using an idea of Alfieri in \cite{Antonio}.
A \emph{south-west region} $S\subset\mathbb Z^2$ (resp. $\mathbb R^2$) is a subset of $\mathbb Z^2$ (resp. a topological submanifold of $\mathbb R^2$) such that if $(\overline t,\overline s)\in S$ then $s\leq\overline s$ and $t\leq\overline t$ imply $(t,s)\in S$. Moreover, we require $S$ to differ from $\emptyset$ and $\mathbb Z^2$ (resp. $\mathbb R^2$). 

Consider again the map $\pi_d:\Ker\partial^-_d\rightarrow\cH^\infty_d(L)$. Define \[\begin{aligned}\Ker\partial^-_{d,S}=&\Ker\partial^-_d\cap\text{Span}\left\{\mathcal F^{t,s}\cC^\infty_d(\mathcal D)\:|\:(t,s)\in S\right\}:=\\:=&\Ker\partial^-_d\cap\mathcal F^S\cC^\infty_d(\mathcal D)\:.\end{aligned}\] Then we say that \[\mathcal F^S\cH^\infty_d(L)=\pi_d(\Ker\partial^-_{d,S})\subset\cH^\infty_d(L)\] for any $d\in\mathbb Z$. 
Note that the level $\mathcal F^{t,s}$ corresponds to the south-west region $V_{t,s}=\{(j,A)\:|\:j\leq t,A\leq s\}$, while $j^t:=\mathcal F^{\{j\leq t\}}$ and $\mathcal A^s:=\mathcal F^{\{A\leq s\}}$ correspond to $\{(j,A)\:|\:j\leq t\}$ and $\{(j,A)\:|\:A\leq s\}$ respectively.

The filtration $\mathcal F$ is increasing in the sense that if $S_1\subset S_2$ are two south-west regions then $\mathcal F^{S_1}\cH^\infty_*(L)\subset\mathcal F^{S_2}\cH^\infty_*(L)$;
moreover, it has the following property.
\begin{prop}
 \label{prop:bounded}
 Fix an integer $d$, denote with $W_{t,s}$ the south-west region in Figure \ref{Wsouthwestregion} and take $V_{t,s}$ as before. 
 Then there exists a pair $(t,s)$ such that $\mathcal F^S\cH^\infty_d(L)\cong\{0\}$ for every south-west region $S\subset W_{t,s}$.
 
 Furthermore, there is another pair $(t',s')$ such that $\mathcal F^T\cH^\infty_d(L)\cong\cH^\infty_d(L)$ for every south-west region $T\supset V_{t',s'}$.
\end{prop}
\begin{proof}
 Since $\cC^\infty(\mathcal D)$ is finitely generated as $\F[U,U^{-1}]$-module, we have that  $\cC^\infty_d(\mathcal D)$ is a finite dimensional $\F$-vector space. 
 \begin{figure}[t]
 \centering
 \begin{tikzpicture}[scale =.5]
        \coordinate (Origin)   at (0,0);
        \coordinate (XAxisMin) at (-5,0);
        \coordinate (XAxisMax) at (5,0);
        \coordinate (YAxisMin) at (0,-5);
        \coordinate (YAxisMax) at (0,5);
        \draw [thick, black,-latex] (XAxisMin) -- (XAxisMax);
        \draw [thick, black,-latex] (YAxisMin) -- (YAxisMax);
        \clip (-4.9,-4.9) rectangle (5.99,5.99);
        \draw[style=help lines] (-5,-5) grid[step=1cm] (4.95,4.95);
        
        \node[] at (-0.5,5.5) {$A$};
        \node[] at (5.5,-0.5) {$j$};
        
        \node[draw,circle,fill,scale=0.4] at (2,2) {};
        \node[above left] at (2,2) {$(t,s)$};
        \draw [very thick, black] (2,2) -- (2,5);
        \draw [very thick, black] (2,2) -- (5,2);
        
        \draw[fill, opacity=.5, gray] (2,5) rectangle (-5,-5);
        \draw[fill, opacity=.5, gray] (2,-5) rectangle (5,2);
 \end{tikzpicture}
 \caption{The south-west region $W_{t,s}$ is the subset $\{(j,A)\:|\:j\leq t\text{ or }A\leq s\}$ of $\mathbb R^2$.}
 \label{Wsouthwestregion}
\end{figure}
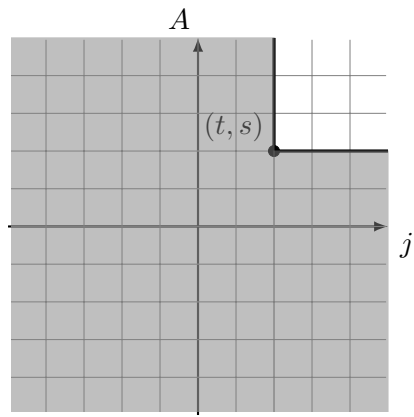
 Then there are integers $A$, the minimal Alexander level containing a generator of $\cC^\infty_d(\mathcal D)$, and $B$, the same considering algebraic levels, because of Equation \eqref{eq:filtration}. If we choose $t< B$ and $s< A$ then $\mathcal F^{W_{t,s}}\cC^\infty_d(\mathcal D)\cong\{0\}$ and so $\mathcal F^{W_{t,s}}\cH^\infty_d(L)$ is also zero. The first claim now follows from the fact that $\mathcal F$ is an increasing filtration; for the second one we reason exactly in the same way.
\end{proof}
From \cite{OSlinks} we have the following important theorem.
\begin{teo}[Ozsv\'ath and Szab\'o] 
 \label{teo:quasi_iso}
 The $\mathcal F$-filtered chain homotopy type of $\cC^\infty(\mathcal D)$, together with the Maslov grading, is a link invariant of $L$, where $\mathcal D$ is a Heegaard diagram for $L$.
\end{teo}
For simplicity from now on we may denote our chain complex with $\cC^\infty_*(L)$, implicitly referring to any of the representatives of the filtered chain homotopy type.
This result guarantees that also the $\mathcal F$-filtration on $\cH^\infty_*$ is a link invariant, justifying our notation. 

We call a graded isomorphism $F$ between the homology groups of two links $L_1$ and $L_2$ a \emph{filtered isomorphism} if $F$ and its inverse $F^{-1}$ both preserve the filtration $\mathcal F$. This is equivalent to say that $F$ restricts to isomorphisms
\[\mathcal F^S\cH^\infty_d(L_1)\cong_\F\mathcal F^S\cH^\infty_d(L_2)\] for every $d\in\mathbb Z$ and south-west region $S$ of $\mathbb Z^2$.

When $L_1$ and $L_2$ are isotopic links Theorem \ref{teo:quasi_iso} implies that $\cC^\infty(L_1)$ is locally equivalent to $\cC^\infty(L_2)$, following the notation of Zemke in \cite{Zemke}. This means we can find chain maps $f:\cC^\infty(L_1)\rightarrow\cC^\infty(L_2)$ and $g:\cC^\infty(L_2)\rightarrow\cC^\infty(L_1)$ which both preserve $\mathcal F$ and induce $\mathcal F$-filtered isomorphisms between $\cH^\infty(L_1)$ and $\cH^\infty(L_2)$.
\begin{cor}
 \label{cor:quasi_iso}
  Suppose that the link $L_1$ is isotopic to the link $L_2$ in $S^3$; then there is a local equivalence between $\cC^\infty(L_1)$ and $\cC^\infty(L_2)$. 
\end{cor} 
Note that we can assume $f$ to be a chain homotopy equivalence, but in general a local equivalence is not necessarily an $\mathcal F$-filtered chain homotopy equivalence. This would happen if the chain homotopies between $f$ and its homotopy inverse also preserve $\mathcal F$.

We call a south-west region $S$ of $\mathbb R^2$ \textbf{centered} if $(0,0)$ belongs to the boundary $\partial S$ of $S$. Consider \[S_k=\left\{(t,s)\in\mathbb R^2\:|\:\left(t+\dfrac{k}{2},s+\dfrac{k}{2}\right)\in S\right\}\:,\] where $k\in\mathbb R$. We define the invariant $\Upsilon_S(L)$ as follows. Given a centered south-west region $S$ of $\mathbb R^2$, we say that
\[\Upsilon_S(L):=\max_{k\in\mathbb R}\left\{k\:|\:\mathcal F^{S_k}\cH^\infty_0(L)\supset\mathcal F^{\{j\leq 0\}}\cH^\infty_0(L)\right\}\:.\] We recall that the $\mathcal F$-level $\{j\leq 0\}$ coincides with the level $j^0$ of the algebraic filtration. Note also that Theorem \ref{teo:dim} implies $\dim_{\F}\cH^\infty_0(L)>1$ for links with three or more components, but $\dim_{\F}\mathcal F^{\{j\leq 0\}}\cH^\infty_0(L)$ is always equal to one. For this reason, in the definition of $\Upsilon_S$, we need the region $S_k$ not only to contain a generator of the total homology in Maslov grading zero, but also that such element is homologous to one which lives in the algebraic level $j^0$.
\begin{cor} 
 The real number $\Upsilon_S(L)$ is a link invariant for every south-west region $S$ of $\mathbb R^2$.
\end{cor}
\begin{proof}
 It follows immediately from Proposition \ref{prop:bounded} and Corollary \ref{cor:quasi_iso}.
\end{proof}
 
\subsection{Duality and mirror images}
\label{subsection:duality}
Let us start this subsection from a Heegaard diagram $\mathcal D$ for an oriented link $L$ in $S^3$. As we recalled in Subsection \ref{subsection:complex}, from $\mathcal D$ we obtain the chain complex $(\cC^\infty(\mathcal D),\partial^-)$. We now define the corresponding dual complex $(\cC^\infty(\mathcal D)^*,\partial^-_*)$ as follows.

The space $\cC^\infty(\mathcal D)^*$ as an $\F[U,U^{-1}]$-module is isomorphic to 
\begin{equation}
 \Hom_{\F[U,U^{-1}]}\left(\cC^\infty(\mathcal D),\F[U,U^{-1}]\right)\:.
 \label{Dual}
\end{equation}
If $x$ is an intersection point then its dual $x^*$ is the functional which sends $x$ into 1 and the other intersection points into 0; and this implies $p^*\in\cC^\infty(\mathcal D)^*$ can be defined by $\F[U,U^{-1}]$-linearization of the dual of the intersection points.
More specifically, we say that \[\cC^\infty_d(\mathcal D)^*:=\left(\cC^\infty_{-d}(\mathcal D)\right)^*=\Big\{p^*\in\cC^\infty(\mathcal D)^*\:\big|\:1\in p^*(\cC^\infty_{m}(\mathcal D))\text{ implies }m=-d\Big\}\:.\]
Notice that \[\cC^\infty(\mathcal D)^*\cong_\F\bigoplus_{d\in\mathbb Z}\cC^\infty_d(\mathcal D)^*\:,\] but $\cC^\infty(\mathcal D)^*\not\cong_\F\Hom_\F\left(\cC^\infty(\mathcal D),\F\right)$.  
In particular, we have $U^{\pm 1}p^*:=\left(U^{\mp 1}p\right)^*$ and thus \[M(U^{\pm 1}p^*)=M\left((U^{\mp 1}p)^*\right)=-M(U^{\mp 1}p)=-M(p)\mp 2=M(p^*)\mp 2\] as expected.

We can also define the dual filtration $\mathcal F^*$. 
\begin{figure}[t]
 \centering
 \begin{tikzpicture}[scale =.5]
        \coordinate (Origin)   at (0,0);
        \coordinate (XAxisMin) at (-5,0);
        \coordinate (XAxisMax) at (5,0);
        \coordinate (YAxisMin) at (0,-5);
        \coordinate (YAxisMax) at (0,5);
        \draw [thick, black,-latex] (XAxisMin) -- (XAxisMax);
        \draw [thick, black,-latex] (YAxisMin) -- (YAxisMax);
        \clip (-4.9,-4.9) rectangle (5.99,5.99);
        \draw[style=help lines] (-5,-5) grid[step=1cm] (4.95,4.95);
        
        \node[] at (-0.5,5.5) {$A$};
        \node[] at (5.5,-0.5) {$j$};
        
        \node[above left] at (-2,-2) {$S$};
        \draw [very thick, black] (-5,1) -- (-1,1);
        \draw [very thick, black] (-1,-1) -- (-1,1);
        \draw [very thick, black] (-1,-1) -- (2,-1);
        \draw [very thick, black] (2,-1) -- (2,-5);
        
        \draw[fill, opacity=.5, gray] (-1,1) rectangle (-5,-5);
        \draw[fill, opacity=.5, gray] (-1,-5) rectangle (2,-1);
 \end{tikzpicture}
 \hspace{0.5cm}
 \begin{tikzpicture}[scale =.5]
        \coordinate (Origin)   at (0,0);
        \coordinate (XAxisMin) at (-5,0);
        \coordinate (XAxisMax) at (5,0);
        \coordinate (YAxisMin) at (0,-5);
        \coordinate (YAxisMax) at (0,5);
        \draw [thick, black,-latex] (XAxisMin) -- (XAxisMax);
        \draw [thick, black,-latex] (YAxisMin) -- (YAxisMax);
        \clip (-4.9,-4.9) rectangle (5.99,5.99);
        \draw[style=help lines] (-5,-5) grid[step=1cm] (4.95,4.95);
        
        \node[] at (-0.5,5.5) {$A$};
        \node[] at (5.5,-0.5) {$j$};
        
        \node[above left] at (-2,-2) {$\iota S$};
        \draw [very thick, dotted ,black] (-2,5) -- (-2,1);
        \draw [very thick, dotted ,black] (-2,1) -- (1,1);
        \draw [very thick, dotted, black] (1,1) -- (1,-1);
        \draw [very thick,dotted, black] (1,-1) -- (5,-1);
        
        \draw[fill, opacity=.5, gray] (-2,5) rectangle (-5,-5);
        \draw[fill, opacity=.5, gray] (-2,-5) rectangle (1,1);
        \draw[fill, opacity=.5, gray] (1,-5) rectangle (5,-1);
 \end{tikzpicture}
 \caption{The dotted boundary in the picture on the right is not part of $\iota S$.}
 \label{Inverseregion}
\end{figure}
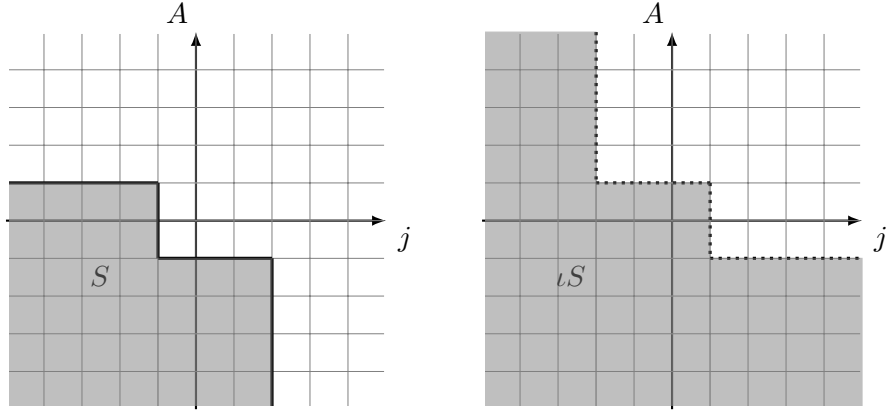
In order to do this, we introduce the concept of inverse $\iota S$ of a south-west region $S$ in $\mathbb Z^2$ (resp. $\mathbb R^2$). We take $\iota S$ as the complement of the image of $S$ under the symmetry centered in the origin of the plane, see Figure \ref{Inverseregion} for an example.
\begin{lemma}
 If $S$ is a south-west region then $\iota S$ is also a south-west region.
\end{lemma}
\begin{proof}
 The mirror image of $S$ is a north-east region. The complement of a north-east region is a south-west region; in fact, if $(x,y)\in\iota S$ and $(\overline x,\overline y)\not\in\iota S$ with $\overline x\leq x$ and $\overline y\leq y$ then $(\overline x,\overline y)$ belongs to the north-east region $(\iota S)^{\mathsf c}$, which means that $(x,y)$ is also in $(\iota S)^{\mathsf c}$. This is a contradiction.
\end{proof}
The dual filtration is defined in the following way: 
\[\begin{aligned}(\mathcal F^*)^S\cC^\infty_d(\mathcal D)^*:=&\Ann\mathcal F^{\iota S}\cC^\infty_{-d}(\mathcal D)= \\
=&\left\{p^*\in\cC^\infty_d(\mathcal D)^*\:|\:p^*\left(\mathcal F^{\iota S}\cC^\infty_{-d}(\mathcal D)\right)=0\right\}\end{aligned}\] for any south-west region $S$. We observe that if $S'\subset S$ then $\iota S\subset\iota S'$ and so $\Ann\mathcal F^{\iota S'}\subset\Ann\mathcal F^{\iota S}$. This means that $\mathcal F^*$ is still an increasing filtration.

The only missing part in the dual complex is the differential. We introduce $\partial^-_*$ as follows. 
For every $x^*\in\cC^\infty(\mathcal D)^*$ and $y\in\cC^\infty(\mathcal D)$ we have \[(\partial^-_*x^*)(y)=x^*(\partial^-y)\:;\] moreover, we take $\partial^-_*(Up^*)=U\cdot\partial^-_*p^*$.
\begin{lemma}
 The map $\partial^-_*$ is a differential, drops the Maslov grading by $1$ and preserves the filtration $\mathcal F^*$.
\end{lemma}
\begin{proof}
 First we have \[\partial^-_*(\partial^-_*x^*(y))=(\partial^-_*x^*)(\partial^-y)=x^*(\partial^-\circ\partial^-y)=0=0(y)\] for any $y\in\cC^\infty(\mathcal D)$.
 For the second claim, suppose that $p^*\in\cC^\infty_d(\mathcal D)^*$. Then one has $\partial^-_*p^*(q)=p^*(\partial^-q)$ so if $q\in\cC^\infty_{-d+1}(\mathcal D)$ then we have that $\partial^-q\in\cC^\infty_{-d}(\mathcal D)$; 
 in addition, if $r\notin\cC^\infty_{-d+1}(\mathcal D)$ is homogeneous then $\partial^-_*p^*(r)=0$ 
 and this implies
 \[\partial^-_*p^*=\partial^-_*p^*\lvert_{\cC^\infty_{-d+1}(\mathcal D)}\in\left(\cC^\infty_{-d+1}(\mathcal D)\right)^*=\cC^\infty_{d-1}(\mathcal D)^*\:.\]
 Finally, suppose that $p^*\in(\mathcal F^*)^S\cC^\infty_d(\mathcal D)^*$ for a south-west region $S$. Then one has $p^*\left(\mathcal F^{\iota S}\cC^\infty_{-d}(\mathcal D)\right)=0$. If $q\in\mathcal F^{\iota S}\cC^\infty_{-d+1}(\mathcal D)$ then $(\partial^-_*p^*)(q)=0$, since $\partial^-q\in\mathcal F^{\iota S}\cC^\infty_{-d}(\mathcal D)$; this implies that \[\partial^-_*p^*\in\Ann\mathcal F^{\iota S}\cC^\infty_{-d+1}(\mathcal D)=(\mathcal F^*)^S\cC^\infty_{d-1}(\mathcal D)^*\:.\] This completes the proof.
\end{proof}
We can now prove that the dual complex that we have just defined is related to the complex obtained from a Heegaard diagram of the mirror image of $L$. We denote with $\mathcal C_d\lkhov a\rkhov$ the graded complex given by $\mathcal C_{d-a}$. 
\begin{teo}
 \label{teo:mirror}
 If $(\cC^\infty(\mathcal D),\partial^-)$ is the chain complex associated to a Heegaard diagram $\mathcal D$ for $L$ then there is a diagram $\mathcal D^*$, representing the mirror image $L^*$ of $L$, such that \[(\cC^\infty(\mathcal D^*),\partial^-_{\mathcal D^*})=(\cC^\infty(\mathcal D)^*,\partial^-_*)\lkhov1-n\rkhov\] as $\mathcal F$-filtered, graded chain complexes.
\end{teo}
\begin{proof}
 If $\mathcal D=(\Sigma,\alpha,\beta,\textbf w,\textbf z)$ then $\mathcal D^*=(-\Sigma,\alpha,\beta,\textbf w,\textbf z)$. This identifies the domain $\phi\in\pi_2(x,y)$ with $-\phi\in\pi_2(y,x)$, see \cite{OSSz,OSlinks}; moreover, using the formula in \cite{Lipshitz} it is easy to check that $\phi$ and $-\phi$ have the same Maslov index. The identification that proves the theorem is $x\rightarrow x^*$, extended $U$-equivariantly to the whole complexes, where $x^*$ denotes the dual of $x$ as before. We first show that such a map is indeed a chain map:
 \[(\partial^-_{\mathcal D^*}(x))^*(t)=\sum_{y\in\T}\sum_{\substack{\phi\in\pi_2(y,x)\\ \mu(\phi)=1}}m(\phi)\cdot U^{n_{\textbf w}(\phi)}y^*(t)=
 \sum_{\substack{\phi\in\pi_2(t,x)\\ \mu(\phi)=1}}m(\phi)\cdot U^{n_{\textbf w}(\phi)}\]
 \[\partial^-_*(x^*(t))=x^*(\partial^-t)=x^*\left(\sum_{y\in\T}\sum_{\substack{\phi\in\pi_2(t,y)\\ \mu(\phi)=1}}m(\phi)\cdot U^{n_{\textbf w}(\phi)}y\right)=\sum_{\substack{\phi\in\pi_2(t,x)\\ \mu(\phi)=1}}m(\phi)\cdot U^{n_{\textbf w}(\phi)}\:,\] which holds for every generator $t$ of $\cC^\infty(\mathcal D)$ and so the claim is proved.
 
 Now we argue that our identification correctly shifts the Maslov and Alexander gradings. Suppose that $M(x)=d$, then by definition it is $M_*(x^*)=-d$. We observe that from \cite{OSlinks} we have \[M(x)-M(y)=\mu(\phi)-2n_{\textbf w}(\phi)=\mu(-\phi)-2n_{\textbf w}(-\phi)=M_{\mathcal D*}(y^*)-M_{\mathcal D*}(x^*)\] with $\phi\in\pi_2(x,y)$ and then $M_{\mathcal D*}$ is reversed as a relative grading, which means $M_{\mathcal D*}(x)=-d+c$ with $c\in\mathbb Z$. Now we use the fact that the Maslov grading is always normalized in the way that the top grading, where the total homology is non-trivial, is zero (\cite{OSlinks}). This gives $c=1-n$ as wanted.
 
 Finally, consider $x$ such that $A(x)=s$. As before, using the relation \[A(x)-A(y)=n_\textbf z(\phi)-n_\textbf w(\phi)\] whenever $\phi\in\pi_2(x,y)$ and the fact that the Alexander grading is always symmetric, we obtain $A_{\mathcal D^*}(x)=-s$. We use the definition of $\mathcal F^*$ to recover \[\begin{aligned}A_*(x^*)=&\min_{a\in\mathbb Z}\left\{a\:|\:x^*\in(\mathcal A^*)^a\cC^\infty(\mathcal D)^*\right\}= \\ =&\min_{a\in\mathbb Z}\left\{a\:|\:x^*\left(\mathcal A^{-a-1}\cC^\infty(\mathcal D)\right)=0\right\}=-\max_{a\in\mathbb Z}\left\{a\:|\:x\notin\mathcal A^{a-1}\cC^\infty(\mathcal D)\right\}=\\ =&-\min_{a\in\mathbb Z}\left\{a\:|\:x\in\mathcal A^{a}\cC^\infty(\mathcal D)\right\}=-s\end{aligned}\]
 and the proof is completed.
\end{proof}
Note that the identification in Theorem \ref{teo:mirror} also gives that the homology group of the mirror image of $L$ is the dual of the homology group of $L$, where the latter group is defined exactly as in Equation \eqref{Dual}.
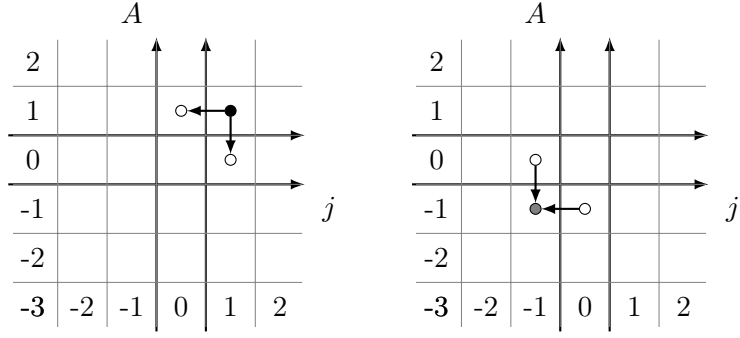
\begin{figure}[t]
 \centering
 \begin{tikzpicture}[scale =.65]
        \coordinate (Origin)   at (0,0);
        \coordinate (XAxisMin) at (-3,0);
        \coordinate (XAxisMax) at (3,0);
        \coordinate (YAxisMin) at (0,-3);
        \coordinate (YAxisMax) at (0,3);
        \draw [thick, black,-latex] (XAxisMin) -- (XAxisMax);
        \draw [thick, black,-latex] (YAxisMin) -- (YAxisMax);
        \draw [thick, black,-latex] (-3,1) -- (3,1);
        \draw [thick, black,-latex] (1,-3) -- (1,3);
        \clip (-2.9,-2.9) rectangle (3.99,3.99);
        \draw[style=help lines] (-3,-3) grid[step=1cm] (2.95,2.95);
        
        \foreach \x in {-3,...,2}
                    \node[draw=none,fill=none] at (\x+0.5,-2.5) {\x};
    	\foreach \y in {-3,...,2}     		
        			\node[draw=none,fill=none] at (-2.5,\y+0.5) {\y};
        
        \node[] at (-0.5,3.5) {$A$};
        \node[] at (3.5,-0.5) {$j$};
        
        \node[draw,circle,fill,scale=0.4]  (A) at (1.5,1.5) {};
        \node[draw,circle,fill=white,scale=0.4] (B) at (0.5,1.5) {};
        \node[draw,circle,fill=white,scale=0.4] (C) at (1.5,0.5) {};
        
        \path [thick,-latex] (A) edge node[left] {} (B);
        \path [thick,-latex] (A) edge node[left] {} (C);
 \end{tikzpicture}
 \hspace{.5cm}
 \begin{tikzpicture}[scale =.65]
        \coordinate (Origin)   at (0,0);
        \coordinate (XAxisMin) at (-3,0);
        \coordinate (XAxisMax) at (3,0);
        \coordinate (YAxisMin) at (0,-3);
        \coordinate (YAxisMax) at (0,3);
        \draw [thick, black,-latex] (XAxisMin) -- (XAxisMax);
        \draw [thick, black,-latex] (YAxisMin) -- (YAxisMax);
        \draw [thick, black,-latex] (-3,1) -- (3,1);
        \draw [thick, black,-latex] (1,-3) -- (1,3);
        \clip (-2.9,-2.9) rectangle (3.99,3.99);
        \draw[style=help lines] (-3,-3) grid[step=1cm] (2.95,2.95);
        
        \foreach \x in {-3,...,2}
                    \node[draw=none,fill=none] at (\x+0.5,-2.5) {\x};
    	\foreach \y in {-3,...,2}     		
        			\node[draw=none,fill=none] at (-2.5,\y+0.5) {\y};
        
        \node[] at (-0.5,3.5) {$A$};
        \node[] at (3.5,-0.5) {$j$};
        
        \node[draw,circle,fill=white,scale=0.4] (A) at (-0.5,0.5) {};
        \node[draw,circle,fill=white,scale=0.4] (B) at (0.5,-0.5) {};
        \node[draw,circle,fill=gray,scale=0.4] (C) at (-0.5,-0.5) {};
        
        \path [thick,-latex] (A) edge node[left] {} (C);
        \path [thick,-latex] (B) edge node[left] {} (C);
 \end{tikzpicture}
 \caption{The complex $\cC^\infty(T_{2,3})$ is on the left and $\cC^\infty(T_{2,3}^*)$ on the right, both chain complexes are pictured ignoring the $U$-action. Black, white and gray dots represent Maslov gradings $1,0$ and $-1$ respectively.}
 \label{Homologytrefoils}
 \end{figure}
Furthermore, as an example in Figure \ref{Homologytrefoils} we show the filtered chain complexes for the positive and the negative trefoil. 

\subsection{Local and stable equivalences of knot Floer chain complexes}
In \cite{Hom} Hom introduced a different equivalence relation for the complexes $CFK^\infty(K)=\cC^\infty(K)$, when $K$ is a knot. More specifically, we say that the Floer complexes associated to the knots $K_1$ and $K_2$ are stably equivalent if we have an $\mathcal F$-filtered chain homotopy equivalence \[CFK^\infty(K_1)\oplus A\simeq CFK^\infty(K_2)\oplus B\:,\] where $A$ and $B$ are acyclic chain complexes; in other words, it is $H_*(A)=H_*(B)=\{0\}$. 
Here we recall
\cite[Corollary 3.2]{HomHe}, which 
shows that the notion of stable equivalence coincides with the one of local equivalence, in the case of knots. 
\begin{lemma}[Hendricks and Hom]
 \label{lemma:acyclic}
 If $CFK^\infty(K)$ is locally equivalent to $CFK^\infty(\bigcirc)=\F[U,U^{-1}]_{(0)}$ then \[CFK^\infty(K)\simeq\F[U,U^{-1}]_{(0)}\oplus A\:,\] where $A$ is acyclic.
\end{lemma}
Thanks to the following result, we can see the local equivalence relation, that we defined for link Floer complexes in Subsection \ref{subsection:Alexander}, as a natural generalization to links of the stable equivalences introduced by Hom.
\begin{proof}[Proof of Theorem \ref{teo:equivalence}]
 If our chain complexes are stably equivalent then, in order to define the local equivalence, we just have to take the restriction of the filtered chain homotopy equivalence and its inverse. Conversely, let us suppose that $f:CFK^\infty(K_1)\rightarrow CFK^\infty(K_2)$ and $g:CFK^\infty(K_2)\rightarrow CFK^\infty(K_1)$ define a local equivalence. Then we have that  \[
   \begin{tikzcd}
 CFK^\infty(K_1)\otimes CFK^\infty(K_2)^* \ar[r,xshift=-.5cm,shift left=.75ex,"f\otimes\Id"]
 &
 \ar[l,xshift=-.85cm,yshift=-1ex,shift right=.75ex,"g\otimes\Id"]
 CFK^\infty(K_2)\otimes CFK^\infty(K_2)^* 
 \ar[r,xshift=-1cm,shift left=.75ex,"f'"]
 &
 \ar[l,xshift=-1.3cm,yshift=-1ex,shift right=.75ex,"g'"]
 &
 \F[U,U^{-1}]_{(0)}
 \end{tikzcd}\] where both the pairs of chain maps give local equivalences. The existence of $f'$ and $g'$ can be proved in the same way as in \cite[Lemma 2.18]{Zemke}. 
 
 From Lemma \ref{lemma:acyclic} one has \[CFK^\infty(K_2)\otimes CFK^\infty(K_2)^*\simeq\F[U,U^{-1}]_{(0)}\oplus A\] and \[CFK^\infty(K_1)\otimes CFK^\infty(K_2)^*\simeq\F[U,U^{-1}]_{(0)}\oplus B\:,\] where $A$ and $B$ are acyclic. Therefore, one has \[\begin{aligned}CFK^\infty(K_1)\oplus A\simeq& CFK^\infty(K_1)\otimes\left(CFK^\infty(K_2)\otimes CFK^\infty(K_2)^*\right)\simeq\\ \simeq& CFK^\infty(K_2)\otimes\left(CFK^\infty(K_1)\otimes CFK^\infty(K_2)^*\right)\simeq CFK^\infty(K_2)\oplus B\end{aligned}\] and the theorem is proved. 
\end{proof}
It is important to observe that, when $L$ is a link with $n$ components and $n$ is at least two, the chain complex $\cC^\infty(L)\otimes\cC^\infty(L)^*$ is not locally equivalent to the complex $\cC^\infty(\bigcirc_n)$ representing the unlink. In fact, these groups have different dimensions as $\F[U,U^{-1}]$-modules. Furthermore, in Subsection \ref{subsection:connected} we give an example of a link $L$ for which such a chain complex is not locally equivalent to any $\cC^\infty(\bigcirc_m)$ for $m\in\N$.

\section{Concordance} 
\label{section:concordance}
\subsection{Canonical form of oriented link cobordisms}
\label{subsection:canonical}
The definition of link cobordism is standard in literature; in particular, for this paper the reader might find helpful to look at \cite{Cavallo,Sarkar}. We only recall that we always assume the connected components of a smooth cobordism $\Sigma\hookrightarrow S^3\times I$, from $L_1$ to $L_2$, to have boundary 
on both the links.

Given a surface $\Sigma$ as before, we assume for now that $\Sigma$ is oriented; we study unoriented cobordisms only in the last section of the paper.  
\begin{figure}[t]
 \centering
 \def\svgwidth{15cm}
\begingroup%
  \makeatletter%
  \providecommand\color[2][]{%
    \errmessage{(Inkscape) Color is used for the text in Inkscape, but the package 'color.sty' is not loaded}%
    \renewcommand\color[2][]{}%
  }%
  \providecommand\transparent[1]{%
    \errmessage{(Inkscape) Transparency is used (non-zero) for the text in Inkscape, but the package 'transparent.sty' is not loaded}%
    \renewcommand\transparent[1]{}%
  }%
  \providecommand\rotatebox[2]{#2}%
  \newcommand*\fsize{\dimexpr\f@size pt\relax}%
  \newcommand*\lineheight[1]{\fontsize{\fsize}{#1\fsize}\selectfont}%
  \ifx\svgwidth\undefined%
    \setlength{\unitlength}{3310.02927982bp}%
    \ifx\svgscale\undefined%
      \relax%
    \else%
      \setlength{\unitlength}{\unitlength * \real{\svgscale}}%
    \fi%
  \else%
    \setlength{\unitlength}{\svgwidth}%
  \fi%
  \global\let\svgwidth\undefined%
  \global\let\svgscale\undefined%
  \makeatother%
  \begin{picture}(1,0.3847731)%
    \lineheight{1}%
    \setlength\tabcolsep{0pt}%
    \put(0,0){\includegraphics[width=\unitlength,page=1]{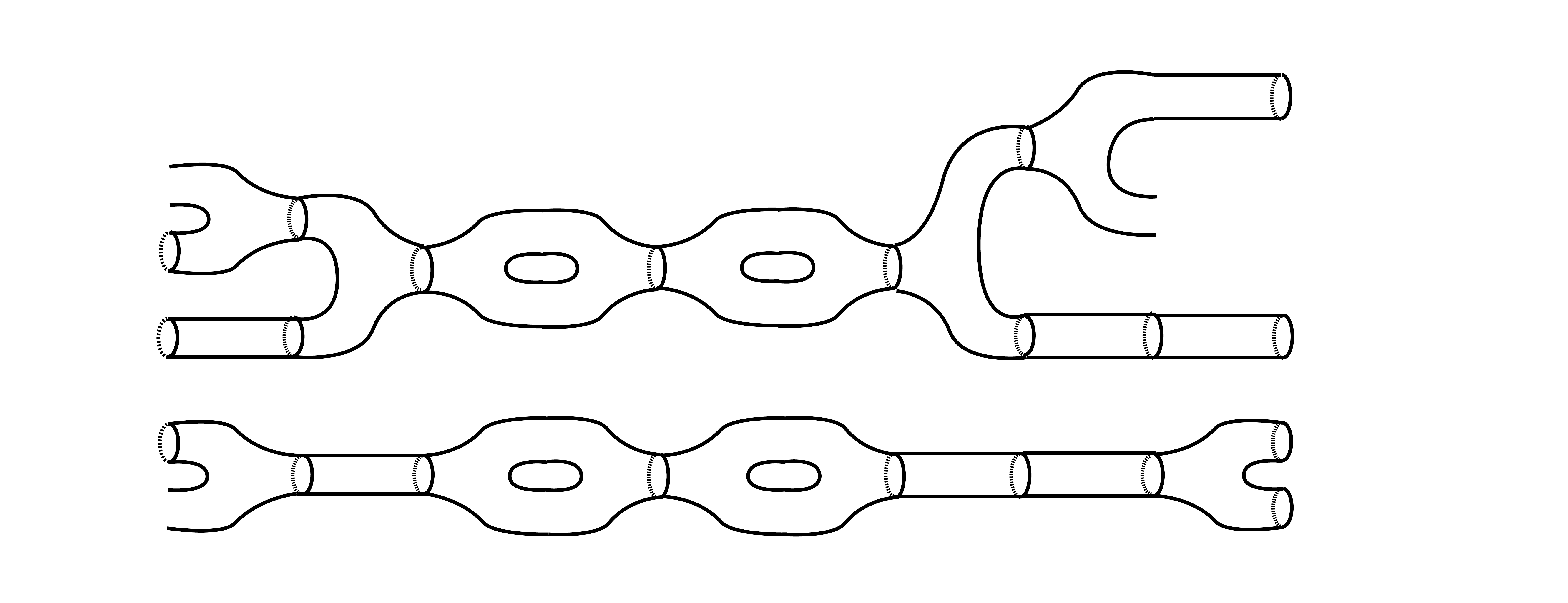}}%
    \put(-0.00250777,0.3416134){\color[rgb]{0,0,0}\makebox(0,0)[lt]{\lineheight{0}\smash{\begin{tabular}[t]{l}$L_1$\end{tabular}}}}%
    \put(0.91289199,0.34873461){\color[rgb]{0,0,0}\makebox(0,0)[lt]{\lineheight{0}\smash{\begin{tabular}[t]{l}$L_2$\end{tabular}}}}%
    \put(0,0){\includegraphics[width=\unitlength,page=2]{Cobordism.pdf}}%
  \end{picture}%
\endgroup%
   
 \caption{Canonical form of oriented cobordisms between two links.}
 \label{Cobordism}
\end{figure}
Then $\Sigma$ consists of four elementary pieces, three of them corresponding to a critical point in the cobordism: birth, band and death moves; while the fourth is a link isotopy, which represents a piece with no critical point. Band moves come in two types: \emph{split}, if the move turns one component into two, and \emph{merge moves}, when two components are joint into one.

For the purpose of this paper, it is more useful to consider what we call \textbf{extended birth} and \textbf{death moves}. These are the composition of a birth and a merge move and of a split and a death move respectively, see Figures \ref{Birth} and \ref{Death}. In addition, we call a \emph{torus move} the composition of a split move with a merge move which rejoins together the newly created components, see Figure \ref{Torus}.
Hence, if $L_i$ has $n_i$ for $i=1,2$ components, while $\Sigma$ has $k$ connected components and genus $g(\Sigma)$, then the \emph{canonical form} of $\Sigma$ is the composition (up to isotopies) of $b$ extended birth moves, $n_1-k$ merge moves, $g(\Sigma)$ torus moves, $n_2-k$ split moves and $d$ death moves in this specific order.
This implies that $\Sigma$ can be arranged as shown in Figure \ref{Cobordism}, see \cite{Cavallo}.

When $L_1$ and $L_2$ both have $n$ components $\Sigma$ is a concordance if it is the union of $n$ disjoint annuli $\Sigma_i$, which means that each $\Sigma_i$ is a knot concordance between the $i$-th components of the two links. From Figure \ref{Cobordism} we immediately see that in the case of a concordance there are no torus moves ($g(\Sigma)=0)$.
\begin{figure}[t]
 \centering
 \def\svgwidth{8cm}
\begingroup%
  \makeatletter%
  \providecommand\color[2][]{%
    \errmessage{(Inkscape) Color is used for the text in Inkscape, but the package 'color.sty' is not loaded}%
    \renewcommand\color[2][]{}%
  }%
  \providecommand\transparent[1]{%
    \errmessage{(Inkscape) Transparency is used (non-zero) for the text in Inkscape, but the package 'transparent.sty' is not loaded}%
    \renewcommand\transparent[1]{}%
  }%
  \providecommand\rotatebox[2]{#2}%
  \newcommand*\fsize{\dimexpr\f@size pt\relax}%
  \newcommand*\lineheight[1]{\fontsize{\fsize}{#1\fsize}\selectfont}%
  \ifx\svgwidth\undefined%
    \setlength{\unitlength}{1827.87908936bp}%
    \ifx\svgscale\undefined%
      \relax%
    \else%
      \setlength{\unitlength}{\unitlength * \real{\svgscale}}%
    \fi%
  \else%
    \setlength{\unitlength}{\svgwidth}%
  \fi%
  \global\let\svgwidth\undefined%
  \global\let\svgscale\undefined%
  \makeatother%
  \begin{picture}(1,0.54421717)%
    \lineheight{1}%
    \setlength\tabcolsep{0pt}%
    \put(0,0){\includegraphics[width=\unitlength,page=1]{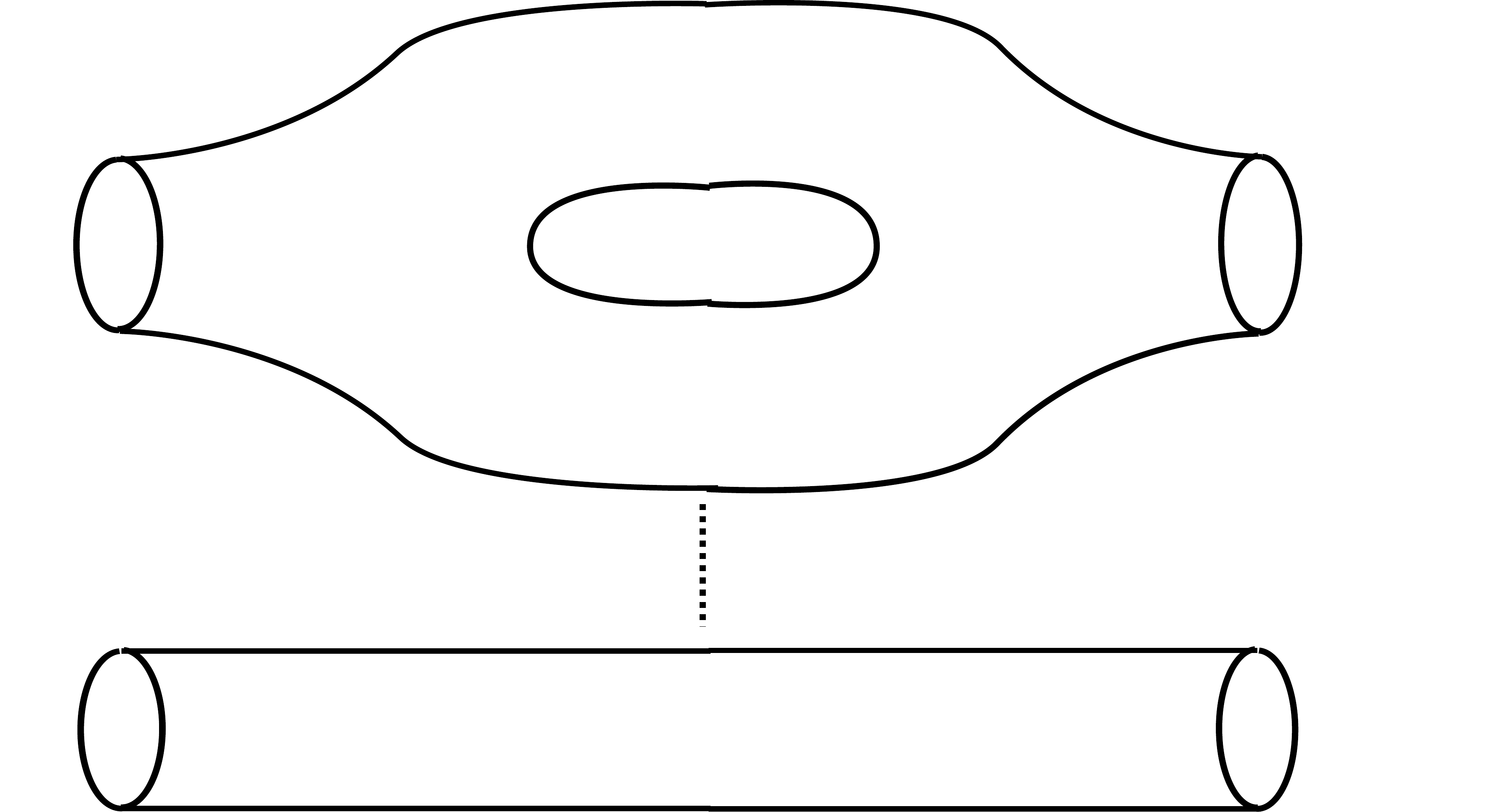}}%
    \put(-0.00340591,0.3688515){\color[rgb]{0,0,0}\makebox(0,0)[lt]{\lineheight{0}\smash{\begin{tabular}[t]{l}$L_1$\end{tabular}}}}%
    \put(0.88169482,0.36533454){\color[rgb]{0,0,0}\makebox(0,0)[lt]{\lineheight{0}\smash{\begin{tabular}[t]{l}$L_2$\end{tabular}}}}%
    \put(0,0){\includegraphics[width=\unitlength,page=2]{Torus.pdf}}%
  \end{picture}%
\endgroup%
   
 \caption{A torus move corresponds to two consecutive band moves on the same component.}
 \label{Torus}
\end{figure}
This means that the canonical form of a concordance can be decomposed into three standard pieces: extended birth moves, isotopies and extended death moves. 

In this section we define maps which relate the chain complexes of the links, this time constructed using grid diagrams, before and after each of these moves. Of course we do not need to study the isotopy cobordism; in fact, in this case Theorem \ref{teo:quasi_iso} and Corollary \ref{cor:quasi_iso} already tell us that the complexes are filtered chain homotopy equivalent and that in particular there exists a local equivalence. The strategy we follow is the same one that the author used in \cite{Cavallo}.

\subsection{Overview on grid diagrams}
\label{subsection:overview}
A grid diagram $D$ of an oriented $n$-component link $L$ in $S^3$ is a grid of $\text{grd}(D)\times\text{grd}(D)$ squares, representing the fundamental domain of a torus, together with a set of $\OO$-markings $\OO=\{O_1,...,O_{\text{grd}(D)}\}$ and one of $\X$-markings $\X=\{X_1,...,X_{\text{grd}(D)}\}$, such that there are exactly one $O$ and one $X$ in every column and every row. Moreover, we choose a non-empty subset $s\OO$ which consists of at most one $\OO$-marking for each component of $L$. We call these $\OO$-markings \emph{special} and the others \emph{normal}.

The link $L$ can be drawn in $D$ by connecting the $O$'s with the $X$'s on a row and the $X$'s with the $O$'s on a column, specifing an orientation on $L$. Vertical lines always overpass the horizontal lines.

The chain complex $\cC^\infty(D)$ is freely generated by the grid states $S(D)$ and it is an \\ $\F[V_1,V_1^{-1},...,V_m,V_m^{-1},U,U^{-1}]$-module, where $\text{grd}(D)-1\geq m\geq\text{grd}(D)-n$ is the number of normal $\OO$-markings. The differential is given by \[\partial^-x=\sum_{y\in S(D)}\sum_{r\in\text{Rect}^\circ(x,y)}V_1^{O_1(r)}\cdot...\cdot V_m^{O_m(r)}\cdot U^{O(r)}y\] for any $y\in S(D)$, where $O_i(r)$ is equal to one if the normal $\OO$-marking $O_i\in r$ and zero otherwise and $O(r)$ is the number of special $\OO$-markings in $r$. The set $\text{Rect}^\circ$ denotes some special rectangles in $D$, see \cite{Book} for details. As in Subsection \ref{subsection:complex} we extend the differential to $\cC^\infty(D)$ by taking $\partial^-(V_i^{\pm 1}p)=V_i^{\pm 1}\cdot\partial^-p$ for every $i=1,...,m$ and $p$ in the complex.
The variables $V_i$ are all homotopic to $U$ so our homology group still has a natural structure of $\F[U,U^{-1}]$-module.

The Maslov and Alexander gradings are also combinatorially defined from $D$ (\cite{Book}) and each variable drops them by $2$ and $1$ respectively; while to define the $j$-filtration we need to specify that the level $t$ is generated by the elements in \[V_1^{i_1}\cdot...\cdot V_m^{i_m}U^i\cdot\cC^-(D)\:,\] where $i_1+...+i_m+i=-t$ and $\cC^-(D)$ is the free $\F[V_1,...,V_m,U]$-module over $S(D)$. 
With this definitions in place we have the following theorem \cite{MOST}, see also \cite{Sarkar}.
\begin{teo}[Manolescu, Ozsv\'ath, Szab\'o and Thurston]
 \label{teo:quasi_iso_grids}
 The $\mathcal A$-filtered chain homotopy type as $\F[U,U^{-1}]$-module of $\cC^\infty_*(D)$ coincides with the one of $\cC^\infty_*(\mathcal D)$ together with the Maslov grading and the algebraic filtration, where $D$ is a grid and $\mathcal D$ is a Heegaard diagram for $L$. 
 
 In particular, if $D_1$ and $D_2$ represent isotopic links then $\cC^\infty_*(D_1)$ is locally equivalent to $\cC^\infty_*(D_2)$.
\end{teo}
The way the filtered homology group $\cH^\infty(L)$ is defined and how the filtration $\mathcal F$ descends into homology are the same as in the previous section.
\begin{remark}
 \label{remark:almost}
 More precisely, Theorem \ref{teo:quasi_iso_grids} tells us that $\cC^\infty_*(D_1)$ is $\mathcal A$-filtered, but not necessarily $\mathcal F$-filtered chain homotopy equivalent to $\cC^\infty_*(D_2)$, while all the maps induced
 by link isotopies preserve the algebraic filtration; in the sense
 that, say $m_1$ and $m_2$ are the numbers of normal $\OO$-markings in $D_1$ and $D_2$, the image of $V_1^{i_1}\cdots V_{m_1}^{i_{m_1}}U^i\cdot\cC^\infty_*(D_1)$ is contained into $V_1^{i_1}\cdots V_{m_1}^{i_{m_1}}U^i\cdot\cC^\infty_*(D_2)$ when $m_1\leq m_2$ or $V_1^{i_1}\cdots V_{m_2}^{i_{m_2}}U^{i+i_{m_2+1}+...+i_{m_1}}\cdot\cC^\infty_*(D_2)$ when $m_1>m_2$ for every $(m_1+1)$-tuple of integers $(i_1,...,i_{m_1},i)$.
 We call this relation between $\cC^\infty_*(D_1)$ and $\cC^\infty_*(D_2)$ (or between the complexes given by a grid and a Heegaard diagram for the same link) an \textbf{almost filtered chain homotopy equivalence} and it implies that the complexes are locally equivalent, as stated in Theorem \ref{teo:quasi_iso_grids}.
\end{remark}
We conclude this subsection with Figure \ref{Gridunlink}, which shows a grid diagram for the two-component unlink $\bigcirc_2$, and a lemma that we need for later.
\begin{figure}[t]
 \centering
   \begin{tikzpicture}[scale=0.5]
    \draw[help lines] (0,0) grid (3,3); 
    \draw[red] (0.5,2.5) circle [radius=0.3]; 
    \draw (0.5,2.5) node [cross] {};
    \draw[red] (1.5,1.5) circle [radius=0.3]; 
    \draw (2.5,0.5) circle [radius=0.3]; 
    \draw (2.5,1.5) node [cross] {};
    \draw (1.5,0.5) node [cross] {};
   \end{tikzpicture}
 \caption{A grid diagram representing the unlink $\bigcirc_2$. The red circles denote the special $\OO$-markings.}
 \label{Gridunlink}
\end{figure}
\begin{lemma}
 \label{lemma:algorithm}
 Given a grid diagram $D$ for a link, we can always change the $\X$-markings to obtain another diagram $D'$ which represents an unlink.
\end{lemma}
\begin{proof}
 We apply the following algorithm. Let us shift the rows of $D$ until there is a special $\OO$-marking in the top row (remember that $D$ is the fundamental domain of a torus); then, starting from this $\OO$-marking denoted by $O_1$, we put an $\X$-marking just below $O_1$. We keep doing this procedure with the $\OO$-markings in the row below, unless we reach an $O_i$ such that $O_{i+1}$ (in the row below) is special. 
 Note that this can happen also when $i=1$. In this case we put the new $\X$-marking in the same column of $O_i$, but not in the row below while in the row where the previous special $\OO$-marking appeared.
 
 When it happens that two consecutive rows $j,j+1$ have both special $\OO$-markings on them, we put the $\X$-marking in the same square of $O_j$ and we continue the algorithm from $O_{j+1}$. At some point we reach the lowest row; in this case, we assume the next row is the very top row (which contains a special $\OO$-marking) and we put $X$ accordingly.
 
 The reader may shift the rows back to the original ordering; in any case, it is easy to check that the new diagram $D'$ represents an unlink and the number of components coincides with the number of the special $\OO$-markings in $D$.
\end{proof}

\subsection{Extended birth moves}
\label{subsection:birth}
Let us study the concordance $\Sigma$ given as in Figure \ref{Birth}: we first need a suitable choice of grid diagrams $D_i$ for $i=1,...,4$, representing the links that appear in the extended birth move at the times shown in the picture.
\begin{figure}[t]
 \centering
 \def\svgwidth{15cm}
\begingroup%
  \makeatletter%
  \providecommand\color[2][]{%
    \errmessage{(Inkscape) Color is used for the text in Inkscape, but the package 'color.sty' is not loaded}%
    \renewcommand\color[2][]{}%
  }%
  \providecommand\transparent[1]{%
    \errmessage{(Inkscape) Transparency is used (non-zero) for the text in Inkscape, but the package 'transparent.sty' is not loaded}%
    \renewcommand\transparent[1]{}%
  }%
  \providecommand\rotatebox[2]{#2}%
  \newcommand*\fsize{\dimexpr\f@size pt\relax}%
  \newcommand*\lineheight[1]{\fontsize{\fsize}{#1\fsize}\selectfont}%
  \ifx\svgwidth\undefined%
    \setlength{\unitlength}{3435.25951122bp}%
    \ifx\svgscale\undefined%
      \relax%
    \else%
      \setlength{\unitlength}{\unitlength * \real{\svgscale}}%
    \fi%
  \else%
    \setlength{\unitlength}{\svgwidth}%
  \fi%
  \global\let\svgwidth\undefined%
  \global\let\svgscale\undefined%
  \makeatother%
  \begin{picture}(1,0.29818417)%
    \lineheight{1}%
    \setlength\tabcolsep{0pt}%
    \put(0,0){\includegraphics[width=\unitlength,page=1]{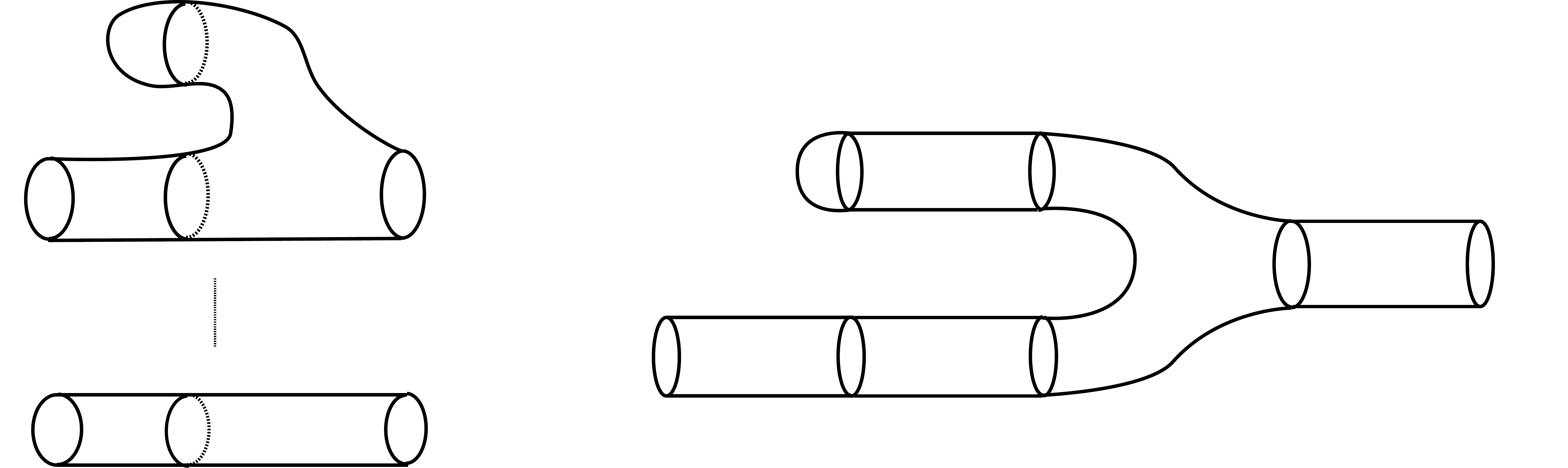}}%
    \put(-0.00193308,0.06275377){\color[rgb]{0,0,0}\makebox(0,0)[lt]{\lineheight{1.25}\smash{\begin{tabular}[t]{l}$L_1$\end{tabular}}}}%
    \put(0.25109143,0.1196553){\color[rgb]{0,0,0}\makebox(0,0)[lt]{\lineheight{1.25}\smash{\begin{tabular}[t]{l}$L_2$\end{tabular}}}}%
    \put(0.41327506,0.00986948){\color[rgb]{0,0,0}\makebox(0,0)[lt]{\lineheight{1.25}\smash{\begin{tabular}[t]{l}$D_1$\end{tabular}}}}%
    \put(0.53304143,0.01111704){\color[rgb]{0,0,0}\makebox(0,0)[lt]{\lineheight{1.25}\smash{\begin{tabular}[t]{l}$D_2$\end{tabular}}}}%
    \put(0.65530295,0.01174082){\color[rgb]{0,0,0}\makebox(0,0)[lt]{\lineheight{1.25}\smash{\begin{tabular}[t]{l}$D_3$\end{tabular}}}}%
    \put(0.92789619,0.07162398){\color[rgb]{0,0,0}\makebox(0,0)[lt]{\lineheight{1.25}\smash{\begin{tabular}[t]{l}$D_4$\end{tabular}}}}%
  \end{picture}%
\endgroup%
   
 \caption{An extended birth move, corresponding to a $0$-handle attachment followed by a $1$-handle (left). The picture on the right shows only the component of $\Sigma$ where the $0$-handle appears.}
 \label{Birth}
\end{figure}
Second, we define maps $b_1:D_1\rightarrow D_2$ and $b_2:D_3\rightarrow D_4$; the first map represents the disjoint union of $L_1$ with an unknot, while the second one the merge move that we need in order to join the new component to $L_1$. Note that the diagrams $D_2$ and $D_3$ present isotopic links; then the corresponding chain complexes are related by an almost filtered chain homotopy equivalence, as in Theorem \ref{teo:quasi_iso_grids}, and thus they are locally equivalent. 

Let us start with $b_2$: the merge move is described by the diagram fragments in Figure \ref{Bandmove}, where we assumed that no special $\OO$-markings were on the new unknotted component. 
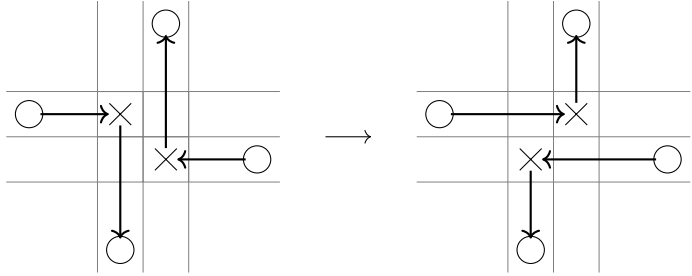
\begin{figure}[t]      
  \begin{center}
   \begin{tikzpicture}[scale=0.6]
    \draw[help lines] (0,0) grid (2,2);
    \draw[help lines] (-2,0)--(0,0); \draw[help lines] (2,0)--(4,0);
    \draw[help lines] (-2,1)--(0,1); \draw[help lines] (0,1)--(4,1);
    \draw[help lines] (-2,2)--(0,2); \draw[help lines] (0,2)--(4,2);
    \draw[help lines] (0,-2)--(0,0); \draw[help lines] (0,0)--(0,4);
    \draw[help lines] (1,-2)--(1,0); \draw[help lines] (1,0)--(1,4);
    \draw[help lines] (2,-2)--(2,0); \draw[help lines] (2,0)--(2,4);
    \draw (1.5,0.5) node [cross] {}; \draw (0.5,1.5) node [cross] {};
    \draw (-1.5,1.5) circle [radius=0.3]; \draw (3.5,0.5) circle [radius=0.3];
    \draw (1.5,3.5) circle [radius=0.3];  \draw (0.5,-1.5) circle [radius=0.3];   
    \draw [thick][<-] (0.5,-1.25) -- (0.5,1.25); \draw [thick][<-] (0.25,1.5) -- (-1.25,1.5);
    \draw [thick][<-] (1.5,3.25) -- (1.5,0.75); \draw [thick][<-] (1.75,0.5) -- (3.25,0.5);  
    \draw [thin][->] (5,1) -- (6,1);
    \draw[help lines] (7,0)--(9,0);     \draw[help lines] (9,0)--(13,0);
    \draw[help lines] (7,1)--(9,1);     \draw[help lines] (9,1)--(13,1);
    \draw[help lines] (7,2)--(9,2);     \draw[help lines] (9,2)--(13,2);
    \draw[help lines] (9,-2)--(9,0);    \draw[help lines] (9,0)--(9,4);
    \draw[help lines] (10,-2)--(10,0);  \draw[help lines] (10,0)--(10,4);
    \draw[help lines] (11,-2)--(11,0);  \draw[help lines] (11,0)--(11,4);
    \draw (9.5,0.5) node [cross] {}; \draw (10.5,1.5) node [cross] {};
    \draw (7.5,1.5) circle [radius=0.3]; \draw (12.5,0.5) circle [radius=0.3];
    \draw (10.5,3.5) circle [radius=0.3];  \draw (9.5,-1.5) circle [radius=0.3];   
    \draw [thick][<-] (9.5,-1.25) -- (9.5,0.25); \draw [thick][<-] (10.25,1.5) -- (7.75,1.5);
    \draw [thick][<-] (10.5,3.25) -- (10.5,1.75); \draw [thick][<-] (9.75,0.5) -- (12.25,0.5);
   \end{tikzpicture}
  \end{center}
  \caption{Band move in a grid diagram}
  \label{Bandmove}
\end{figure}
More explicitly, we are picking the diagram $D_3$ in the way that it contains the fragment on the left in Figure \ref{Bandmove}, while $D_4$ is the resulting diagram after applying the move.
At this point we define $D_3'$ and $D_4'$ as the grid diagrams obtained by applying the algorithm in Lemma \ref{lemma:algorithm} to $D_3$ and $D_4$. This means that they are both diagrams (the same ones!) for $\bigcirc_n$, where $n$ is the number of components of $L_1$ and $L_2$, with the $\OO$-markings in the same position as in $D_3$ and $D_4$.

Since the differential and the $j$-filtration do not depend on the position of the $\X$-markings as we see from their definition in Subsection \ref{subsection:overview}, and this holds also for the Maslov grading \cite{Book,OSlinks}, then the Identity map \[\Id:\cC^\infty(D_3')\longrightarrow\cC^\infty(D_4')\] is a chain map, which clearly induces a graded isomorphism in homology and preserves the algebraic filtration.  
\begin{prop}
 \label{prop:b2}
 The map $b_2:=\Id:\cC^\infty(D_3)\rightarrow\cC^\infty(D_4)$ preserves the Maslov grading and the $\mathcal F$-filtration and induces an isomorphism in homology.
\end{prop}
\begin{proof}
 In order to prove the claim we have to show that the map induces a graded isomorphism in homology and that preserves the two filtrations $j$ and $\mathcal A$. The first two properties only depend on the $\OO$-markings so they hold because $b_2$ is defined as the Identity map; then we only need to show that \[b_2(\mathcal A^s\cC^\infty(D_3))\subset\mathcal A^s\cC^\infty(D_4)\:.\] This can be checked by proving that $A_{D_4}(x)\leq A_{D_3}(x)$ for every $x\in S(D_3)$.
 Note that this is not obvious, even if $b_2$ is the Identity; in fact, this time we need to consider the $\X$-markings in their original position, not like in $D_3'$ and $D_4'$, and the Alexander grading depend on the $X$'s.
 Hence, we need to use a result of Sarkar (\cite[Subsection 3.4]{Sarkar}), which gives us exactly what we need.
\end{proof}
We now want to define $b_1$. We suppose $D_1$ has an $X$ in top-right corner; then we use the move in Figure \ref{Birthmove}.
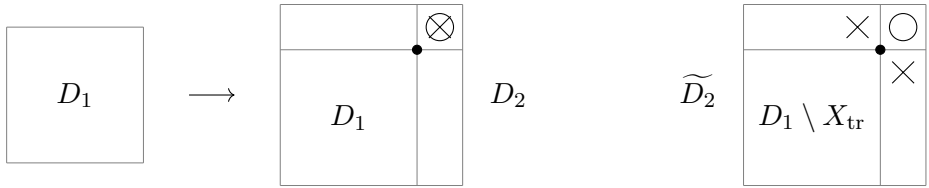
\begin{figure}[t]      
  \begin{center}
   \begin{tikzpicture}[scale=0.6]
    \draw[help lines] (0,0.5)--(3,0.5); \draw[help lines] (0,0.5)--(0,3.5);
    \draw[help lines] (0,3.5)--(3,3.5); \draw[help lines] (3,0.5)--(3,3.5);
    \node at (1.5,2) {$D_1$}; 
    \draw [thin][->] (4,2) -- (5,2);
    \draw[help lines] (6,0)--(10,0); \draw[help lines] (6,0)--(6,4);
    \draw[help lines] (6,3)--(10,3); \draw[help lines] (9,0)--(9,4);
    \draw[help lines] (6,4)--(10,4); \draw[help lines] (10,0)--(10,4);
    \node at (7.5,1.5) {$D_1$}; 
    \draw (9.5,3.5) node [cross] {}; 
    \draw (9.5,3.5) circle [radius=0.3]; 
    \draw[fill] (9,3) circle [radius=0.1];
    \node at (11,2) {$D_2$};
   \end{tikzpicture}
   \hspace{1.5cm}
   \begin{tikzpicture}[scale=0.6]
    \node at (5,2.1) {$\widetilde{D_2}$}; 
    \draw[help lines] (6,0)--(10,0); \draw[help lines] (6,0)--(6,4);
    \draw[help lines] (6,3)--(10,3); \draw[help lines] (9,0)--(9,4);
    \draw[help lines] (6,4)--(10,4); \draw[help lines] (10,0)--(10,4);
    \node at (7.5,1.5) {$D_1\setminus X_{\text{tr}}$}; 
    \draw (8.5,3.5) node [cross] {}; \draw (9.5,2.5) node [cross] {};
    \draw (9.5,3.5) circle [radius=0.3]; 
    \draw[fill] (9,3) circle [radius=0.1];
   \end{tikzpicture}
  \end{center}
  \caption{Birth move in a grid diagram. In the diagram $\widetilde{D_2}$ the top-right $\X$-marking $X_{\text{tr}}$ in $D_1$ does not appear.}
  \label{Birthmove}
\end{figure}
Of course the new doubly-marked square is not a special $\OO$-marking. We consider the filtered NE-stabilization map \[s^{\text{NE}}:\cC^\infty(D_1)\longrightarrow\cC^\infty(\widetilde{D_2})\] defined in \cite{MOST,Book,Sarkar}. Stabilizations relate isotopic links; therefore, such a map is an almost filtered chain homotopy equivalence for Theorem \ref{teo:quasi_iso_grids} and thus a local equivalence. 

We say that $b_1:=s^{\text{NE}}:\cC^\infty(D_1)\rightarrow\cC^\infty(D_2)$. This makes sense because the stabilization maps, in the filtered setting, are independent of the position of the $\X$-markings. Hence, we have the following proposition.
\begin{prop}
 \label{prop:b1}
 The map $b_1:\cC^\infty(D_1)\longrightarrow\cC^\infty(D_2)$ preserves the Maslov grading and the $\mathcal F$-filtration and induces an isomorphism in homology.
\end{prop}
\begin{proof}
 We cannot argue that $b_1$ is an $\mathcal A$-filtered chain homotopy equivalence, because the $\X$-markings in $D_2$ are different with respect to the ones in $\widetilde{D_2}$. On the other hand, we still have that it is a chain homotopy equivalence and preserves the $j$-filtration; in fact, as in Proposition \ref{prop:b2} these two properties ignore the $X$'s. 
 Therefore, we just need to show that $A_{D_2}(s^{\text{NE}}(x))\leq A_{D_1}(x)$ for every $x\in S(D_1)$. This follows from another result of Sarkar (\cite[Subsection 3.4]{Sarkar}).
\end{proof}
Going back to the concordance $\Sigma$, we obtain the following theorem.
\begin{teo}
 \label{teo:birth}
 There is a map $b_{\Sigma}:\cC^\infty(D_1)\rightarrow\cC^\infty(D_4)$, which preserves the Maslov grading and the $\mathcal F$-filtration and induces an isomorphism \[b_{\Sigma}^*:\cH^\infty(L_1)\longrightarrow\cH^\infty(L_2)\:.\]
 In particular, this means that \[b^*_{\Sigma}\left(\mathcal F^S\cH^\infty_d(L_1)\right)\subset\mathcal F^S\cH^\infty_d(L_2)\] for every $d\in\mathbb Z$ and $S$ south-west region of $\mathbb Z^2$.
\end{teo}
\begin{proof}
 We have that $b_{\Sigma}=b_2\circ b\circ b_1$, where $b$ is the almost filtered chain homotopy equivalence between the complexes given by $D_2$ and $D_3$. Then the statement follows from Theorem \ref{teo:quasi_iso_grids} and Propositions \ref{prop:b2} and \ref{prop:b1}, because each piece of the map induces a graded isomorphism in homology and preserves the filtration $\mathcal F$.
\end{proof}

\subsection{Extended death moves and invariance}
\label{subsection:invariance}
An extended death cobordism is described in Figure \ref{Death}. If $\Sigma\hookrightarrow S^3\times I$ is such a cobordism between two $n$-component links $L_1$ and $L_2$ then $\Sigma^*$, the same cobordism seen in the ambient manifold $S^3\times I$ with reversed orientation, can be considered an extended birth cobordism from $L_2^*$ to $L_1^*$.  
\begin{figure}[t]
 \centering
 \def\svgwidth{6cm}
\begingroup%
  \makeatletter%
  \providecommand\color[2][]{%
    \errmessage{(Inkscape) Color is used for the text in Inkscape, but the package 'color.sty' is not loaded}%
    \renewcommand\color[2][]{}%
  }%
  \providecommand\transparent[1]{%
    \errmessage{(Inkscape) Transparency is used (non-zero) for the text in Inkscape, but the package 'transparent.sty' is not loaded}%
    \renewcommand\transparent[1]{}%
  }%
  \providecommand\rotatebox[2]{#2}%
  \newcommand*\fsize{\dimexpr\f@size pt\relax}%
  \newcommand*\lineheight[1]{\fontsize{\fsize}{#1\fsize}\selectfont}%
  \ifx\svgwidth\undefined%
    \setlength{\unitlength}{1647.87963867bp}%
    \ifx\svgscale\undefined%
      \relax%
    \else%
      \setlength{\unitlength}{\unitlength * \real{\svgscale}}%
    \fi%
  \else%
    \setlength{\unitlength}{\svgwidth}%
  \fi%
  \global\let\svgwidth\undefined%
  \global\let\svgscale\undefined%
  \makeatother%
  \begin{picture}(1,0.83391234)%
    \lineheight{1}%
    \setlength\tabcolsep{0pt}%
    \put(0,0){\includegraphics[width=\unitlength,page=1]{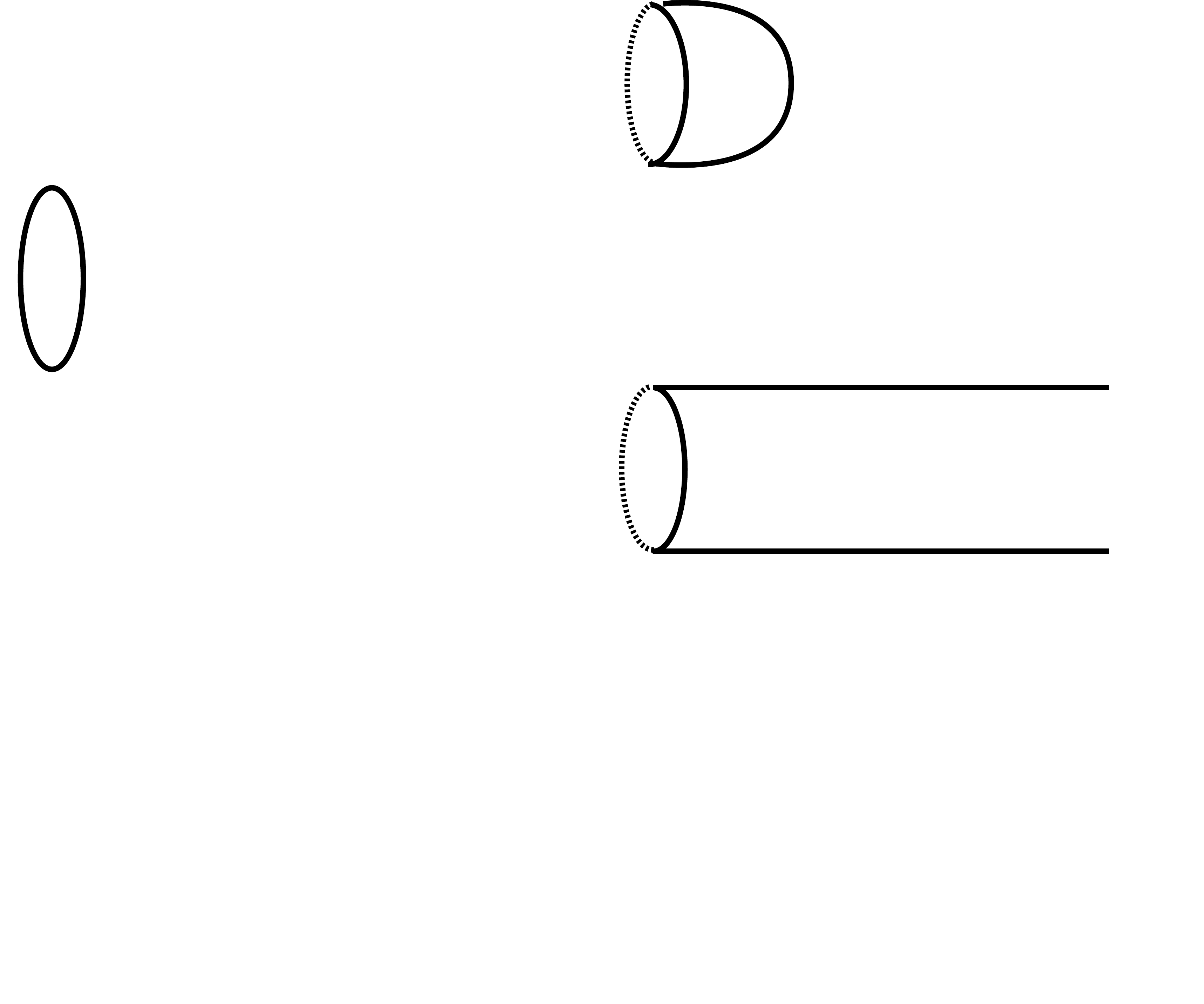}}%
    \put(0.86877224,0.30832688){\color[rgb]{0,0,0}\makebox(0,0)[lt]{\lineheight{0}\smash{\begin{tabular}[t]{l}$L_2$\end{tabular}}}}%
    \put(-0.00377794,0.45266823){\color[rgb]{0,0,0}\makebox(0,0)[lt]{\lineheight{0}\smash{\begin{tabular}[t]{l}$L_1$\end{tabular}}}}%
    \put(0,0){\includegraphics[width=\unitlength,page=2]{Death.pdf}}%
  \end{picture}%
\endgroup%
   
 \caption{An extended death cobordism, corresponding to a $2$-handle attachment together with a $1$-handle.}
 \label{Death}
\end{figure}
Then we can prove the following proposition.
\begin{teo}
 \label{teo:death}
 There is a map $d_{\Sigma}:\cC^\infty(L_1)\longrightarrow\cC^\infty(L_2)$ which preserves the Maslov grading and the $\mathcal F$-filtration and induces an isomorphism in homology.
\end{teo} 
Note that, since $\cC^\infty_d(D)$ is usually not a finite dimensional $\F$-vector space when $D$ is a grid diagram, we cannot directly apply Theorem \ref{teo:mirror} in this case, although this can be done after more work.
\begin{proof}
 Denote with $\cC^\infty(L_i)$ the filtered chain homotopy type of the complexes associated to $L_i$. From Theorems \ref{teo:mirror} and \ref{teo:quasi_iso_grids} we have that the dual complex $\cC^\infty(L_i)^*$ represents the almost filtered chain homotopy type of $\cC^\infty(D_i^*)$. 
 
 We use Theorem \ref{teo:birth} to say that, up to composing with some $j$-filtration preserving $\mathcal A$-filtered chain homotopy equivalences, we can suppose the existence of a map $b_{\Sigma^*}:\cC^\infty(L_2)^*\longrightarrow\cC^\infty(L_1)^*$ which has all the property we want. If we take $b_{\Sigma^*,*}$ as the dual of this map then \[b_{\Sigma^*,*}:\cC^\infty(L_1)\longrightarrow\cC^\infty(L_2)\] preserves the filtration $\mathcal F$ and induces precisely a graded isomorphism in homology; this because the definition of the dual complex in Subsection \ref{subsection:duality} implies that $\cC^\infty(L)^{**}$ has a natural identification with $\cC^\infty(L)$ for every link $L$.
 
 We conclude by saying that $d_{\Sigma}:=b_{\Sigma^*,*}$ again up to compose with some $j$-filtration preserving $\mathcal A$-filtered chain homotopy equivalences.
\end{proof}
Now with this theorem set we can prove one of the main results of the paper.
\begin{proof}[Proof of Theorem \ref{teo:concordance}]
 After applying Theorems \ref{teo:birth} and \ref{teo:death}, by considering the maps induced by a concordance $\Sigma$ from $L_1$ to $L_2$, we obtain a graded isomorphism $F$, between the homology groups, such that $F(\mathcal F^S\cH^\infty_d(L_1))\subset\mathcal F^S\cH^\infty_d(L_2)$, which gives 
 \begin{equation}
     \dim_\F\mathcal F^S\cH^\infty_d(L_1)\leq\dim_\F\mathcal F^S\cH^\infty_d(L_2)\:.
     \label{Filtered}
 \end{equation} 
 In order for $F$ to be a filtered isomorphism we also need that it restricts to an isomorphism on each level of the $\mathcal F$-filtration. To see this we take another concordance $\Sigma'$ from $L_2$ to $L_1$ and, in the same way as before, we get the opposite inequality with respect to Equation \eqref{Filtered}, which proves the claim.
\end{proof}
We now show that the $\Upsilon$-type invariants are indeed concordance invariants. In order to prove this fact, we only need that the $\mathcal F$-filtered isomorphism type of the homology group is a concordance invariant.
\begin{teo}
 \label{teo:upsilon_concordance}
 The real number $\Upsilon_S(L)$ is a concordance invariant for every centered south-west region $S$ in $\mathbb R^2$.
\end{teo}
\begin{proof}
 From Theorem \ref{teo:concordance} we have that $\mathcal F^{S_k}\cH^\infty_0(L_1)\supset\mathcal F^{\{j\leq 0\}}\cH^\infty_0(L_1)$ if and only if 
 $\mathcal F^{S_k}\cH^\infty_0(L_2)\supset\mathcal F^{\{j\leq 0\}}\cH^\infty_0(L_2)$ for every $k\in\mathbb R$, since $L_1$ is concordant to $L_2$. 
 By definition this immediately implies that $\Upsilon_S(L_1)=\Upsilon_S(L_2)$ for every south-west region $S$.
\end{proof}

\section{Upsilon-type invariants}
\label{section:four}
\subsection{Definition of \texorpdfstring{$\Upsilon_S^*(L)$}{Upsilon*S(L)} and the   \texorpdfstring{$\Upsilon$}{Upsilon}-function for links}
In Subsection \ref{subsection:Alexander} we saw that $\mathcal F^{\{j\leq 0\}}\cH^\infty_0(L)$ is isomorphic to $\F$ for every link. Using Theorem \ref{teo:dim} we can also argue that \[\dfrac{\mathcal F^{\{j\leq 0\}}\cH^\infty_{1-n}(L)}{\mathcal F^{\{j\leq -1\}}\cH^\infty_{1-n}(L)}\cong_\F\F\:,\] where $n$ is the number of components of $L$.
\begin{figure}[t]
 \centering
 \begin{tikzpicture}[scale =.42]
        \coordinate (Origin)   at (0,0);
        \coordinate (XAxisMin) at (-5,0);
        \coordinate (XAxisMax) at (5,0);
        \coordinate (YAxisMin) at (0,-5);
        \coordinate (YAxisMax) at (0,5);
        \draw [thick, black,-latex] (XAxisMin) -- (XAxisMax);
        \draw [very thick, black,-latex] (YAxisMin) -- (YAxisMax);
        \clip (-4.9,-4.9) rectangle (5.99,5.99);
        \draw[style=help lines] (-5,-5) grid[step=1cm] (4.95,4.95);
        
        \node[] at (-0.5,5.5) {$A$};
        \node[] at (5.5,-0.5) {$j$};
        
        \draw[fill, opacity=.5, gray] (0,5) rectangle (-5,-5);
 \end{tikzpicture}
 \hspace{.5cm}
 \begin{tikzpicture}[scale =.42]
        \coordinate (Origin)   at (0,0);
        \coordinate (XAxisMin) at (-5,0);
        \coordinate (XAxisMax) at (5,0);
        \coordinate (YAxisMin) at (0,-5);
        \coordinate (YAxisMax) at (0,5);
        \draw [thick, black,-latex] (XAxisMin) -- (XAxisMax);
        \draw [thick, black,-latex] (YAxisMin) -- (YAxisMax);
        \clip (-4.9,-4.9) rectangle (5.99,5.99);
        \draw[style=help lines] (-5,-5) grid[step=1cm] (4.95,4.95);
        
        \node[] at (-0.5,5.5) {$A$};
        \node[] at (5.5,-0.5) {$j$};
        
        \draw[very thick, black] (-5,5) -- (5,-5);
        \fill [gray, opacity=.5, domain=-5:5, variable=\x]
      (-5, -5)
      -- plot ({\x}, {-\x})
      -- (5, -5)
      -- cycle;
 \end{tikzpicture}
 \hspace{.5cm}
 \begin{tikzpicture}[scale =.42]
        \coordinate (Origin)   at (0,0);
        \coordinate (XAxisMin) at (-5,0);
        \coordinate (XAxisMax) at (5,0);
        \coordinate (YAxisMin) at (0,-5);
        \coordinate (YAxisMax) at (0,5);
        \draw [very thick, black,-latex] (XAxisMin) -- (XAxisMax);
        \draw [thick, black,-latex] (YAxisMin) -- (YAxisMax);
        \clip (-4.9,-4.9) rectangle (5.99,5.99);
        \draw[style=help lines] (-5,-5) grid[step=1cm] (4.95,4.95);
        
        \node[] at (-0.5,5.5) {$A$};
        \node[] at (5.5,-0.5) {$j$};
        
        \draw[fill, opacity=.5, gray] (5,0) rectangle (-5,-5);
 \end{tikzpicture}
 \caption{The centered south-west regions $A_0$ (left), $A_1$ (middle) and $A_2$ (right) of $\mathbb R^2$.}
 \label{Classicupsilon}
\end{figure}
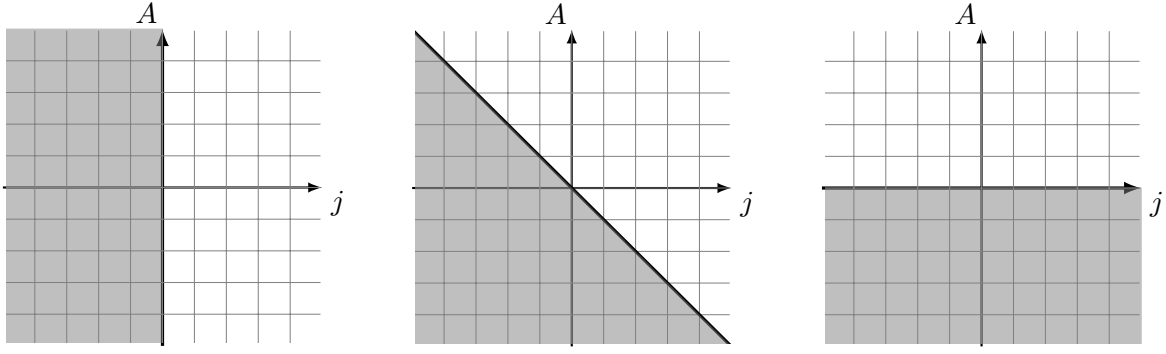
Then, for a given centered south-west region $S\subset\mathbb R^2$, we can define \[\Upsilon_S^*(L):=\max_{k\in\mathbb R}\left\{k\:|\:\mathcal F^{S_k}\cH^\infty_{1-n}(L)\not\subset\mathcal F^{\{j\leq-1\}}\cH^\infty_{1-n}(L)\right\}\:.\] Theorem \ref{teo:concordance} implies that $\Upsilon_S^*(L)$ is also a concordance invariant. Moreover, we observe that for knots $\Upsilon^*$ coincides with $\Upsilon$.

In \cite{OSSz} the $\Upsilon$-invariant is described as a piece-wise linear function $\Upsilon_K(t):[0,2]\rightarrow\mathbb R$ such that $\Upsilon_K(2-t)=\Upsilon_K(t)$ for every knot $K$ and $t$. 
We call this function the \emph{classical $\Upsilon$-invariant}. In the case of links we give a similar definition, which can be seen as a particular case of $\Upsilon_S$.

Consider the centered south-west region \[A_t:=\left\{(j,A)\in\mathbb R^2\:|\:A\cdot\dfrac{t}{2}+j\left(1-\dfrac{t}{2}\right)\leq0\right\}\] for $t\in[0,2]$, see Figure \ref{Classicupsilon}.
It can be showed, see \cite{Antonio}, that $\Upsilon_{A_t}(K)=\Upsilon_K(t)$ for every knot $K$. Moreover, we define \[\Upsilon_L(t):=\Upsilon_{A_t}(L)\quad\text{ and }\quad\Upsilon_L^*(t):=\Upsilon_{A_t}^*(L)\] for every link. 
The reader can easily check that these $\mathbb R$-valued functions are piece-wise linear and $\Upsilon_L(0)=\Upsilon_L^*(0)=0$.
\begin{ex}
In Figure \ref{T33} we show the chain complexes for the $(3,3)$-torus link, which can be computed from the Heegaard diagram in Figure \ref{T33diagram}. 
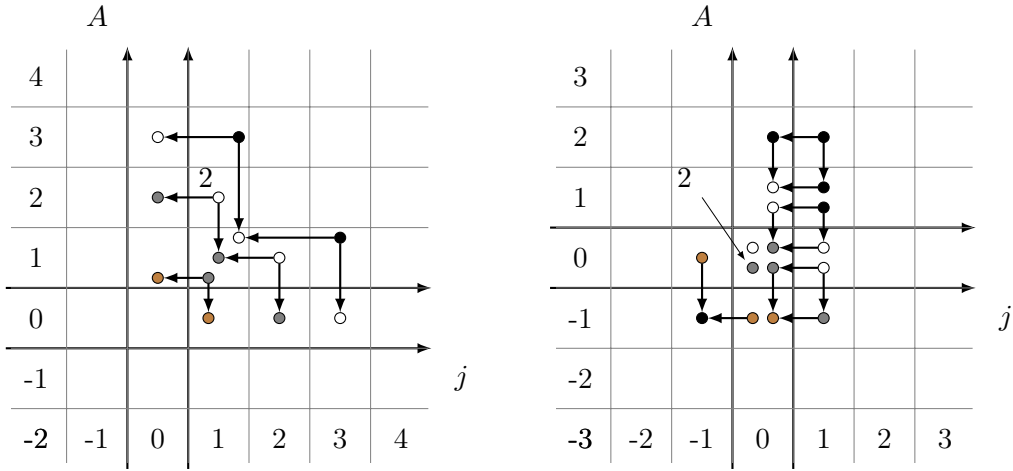
\begin{figure}[t]
 \centering
 \begin{tikzpicture}[scale =.8]
        \coordinate (Origin)   at (0,0);
        \coordinate (XAxisMin) at (-2,0);
        \coordinate (XAxisMax) at (5,0);
        \coordinate (YAxisMin) at (0,-2);
        \coordinate (YAxisMax) at (0,5);
        \draw [thick, black,-latex] (XAxisMin) -- (XAxisMax);
        \draw [thick, black,-latex] (YAxisMin) -- (YAxisMax);
        \draw [thick, black,-latex] (-2,1) -- (5,1);
        \draw [thick, black,-latex] (1,-2) -- (1,5);
        \clip (-1.9,-1.9) rectangle (5.99,5.99);
        \draw[style=help lines] (-2,-2) grid[step=1cm] (4.95,4.95);
        
        \foreach \x in {-2,...,4}
                    \node[draw=none,fill=none] at (\x+0.5,-1.5) {\x};
    	\foreach \y in {-2,...,4}     		
        			\node[draw=none,fill=none] at (-1.5,\y+0.5) {\y};
        
        \node[] at (-0.5,5.5) {$A$};
        \node[] at (5.5,-0.5) {$j$};
        
        \node[draw,circle,fill,scale=0.4] (D1) at (1.5+1/3,3.5) {}; 
        \node[draw,circle,fill,scale=0.4] (D2) at (3.5,1.5+1/3) {};
        \node[draw,circle,fill=white,scale=0.4] (E1) at (0.5,3.5) {};
        \node[draw,circle,fill=white,scale=0.4] (E2) at (1.5+1/3,1.5+1/3) {};
        \node[draw,circle,fill=white,scale=0.4] (E3) at (3.5,0.5) {};
        
        \path [thick,-latex] (D1) edge node[left] {} (E1);
        \path [thick,-latex] (D1) edge node[left] {} (E2);
        \path [thick,-latex] (D2) edge node[left] {} (E2);
        \path [thick,-latex] (D2) edge node[left] {} (E3);
        
        \node[draw,circle,fill=white,scale=0.4] (G1) at (1.5,2.5) {};
        \node[draw,circle,fill=white,scale=0.4] (G2) at (2.5,1.5) {};
        \node[draw,circle,fill=gray,scale=0.4] (F1) at (0.5,2.5) {};
        \node[draw,circle,fill=gray,scale=0.4] (F2) at (2.5,0.5) {};
        \node[draw,circle,fill=gray,scale=0.4] (F3) at (1.5,1.5) {};
        
        \path [thick,-latex] (G1) edge node[left] {} (F1);
        \path [thick,-latex] (G1) edge node[left] {} (F3);
        \path [thick,-latex] (G2) edge node[left] {} (F3);
        \path [thick,-latex] (G2) edge node[left] {} (F2);
        
        \node[draw,circle,fill=gray,scale=0.4]  (A) at (1.5-1/6,1.5-1/3) {};
        \node[draw,circle,fill=brown,scale=0.4] (B) at (0.5,1.5-1/3) {};
        \node[draw,circle,fill=brown,scale=0.4] (C) at (1.5-1/6,0.5) {};
        
        \path [thick,-latex] (A) edge node[left] {} (B);
        \path [thick,-latex] (A) edge node[left] {} (C);
        
        \node[above right] at (1,2.5) {$2$};
 \end{tikzpicture}
 \hspace{.5cm}
 \begin{tikzpicture}[scale =.8]
        \coordinate (Origin)   at (0,0);
        \coordinate (XAxisMin) at (-3,0);
        \coordinate (XAxisMax) at (4,0);
        \coordinate (YAxisMin) at (0,-3);
        \coordinate (YAxisMax) at (0,4);
        \draw [thick, black,-latex] (XAxisMin) -- (XAxisMax);
        \draw [thick, black,-latex] (YAxisMin) -- (YAxisMax);
        \draw [thick, black,-latex] (-3,1) -- (4,1);
        \draw [thick, black,-latex] (1,-3) -- (1,4);
        \clip (-2.9,-2.9) rectangle (4.99,4.99);
        \draw[style=help lines] (-3,-3) grid[step=1cm] (3.95,3.95);
        
        \foreach \x in {-3,...,3}
                    \node[draw=none,fill=none] at (\x+0.5,-2.5) {\x};
    	\foreach \y in {-3,...,3}     		
        			\node[draw=none,fill=none] at (-2.5,\y+0.5) {\y};
        
        \node[] at (-0.5,4.5) {$A$};
        \node[] at (4.5,-0.5) {$j$};
        \node[above left] at (-0.5,1.5) {$2$}; \draw[very thin, black,-latex] (-0.5,1.5) -- (1/3-1/8,1/3+1/8); 
        
        \node[draw,circle,fill=white,scale=0.4] at (1/3,0.5+1/6) {};
        \node[draw,circle,fill=gray,scale=0.4] at (1/3,1/3) {};
        
        \node[draw,circle,fill=brown,scale=0.4] (A1) at (-0.5,0.5) {};
        \node[draw,circle,fill=brown,scale=0.4] (A2) at (1/3,-0.5) {};
        \node[draw,circle,fill,scale=0.4] (B) at (-0.5,-0.5) {};
        
        \path [thick,-latex] (A1) edge node[left] {} (B);
        \path [thick,-latex] (A2) edge node[left] {} (B);
        
        \node[draw,circle,fill,scale=0.4] (C1) at (1.5,2.5) {};
        \node[draw,circle,fill,scale=0.4] (C2) at (0.5+1/6,2.5) {};
        \node[draw,circle,fill,scale=0.4] (C3) at (1.5,1.5+1/6) {};
        \node[draw,circle,fill=white,scale=0.4] (C4) at (0.5+1/6,1.5+1/6) {};
        
        \path [thick,-latex] (C1) edge node[left] {} (C2);
        \path [thick,-latex] (C1) edge node[left] {} (C3);
        \path [thick,-latex] (C2) edge node[left] {} (C4);
        \path [thick,-latex] (C3) edge node[left] {} (C4);
        
        \node[draw,circle,fill,scale=0.4] (D1) at (1.5,1.5-1/6) {};
        \node[draw,circle,fill=white,scale=0.4] (D2) at (0.5+1/6,1.5-1/6) {};
        \node[draw,circle,fill=white,scale=0.4] (D3) at (1.5,0.5+1/6) {};
        \node[draw,circle,fill=gray,scale=0.4] (D4) at (0.5+1/6,0.5+1/6) {};
        
        \path [thick,-latex] (D1) edge node[left] {} (D2);
        \path [thick,-latex] (D1) edge node[left] {} (D3);
        \path [thick,-latex] (D2) edge node[left] {} (D4);
        \path [thick,-latex] (D3) edge node[left] {} (D4);
        
        \node[draw,circle,fill=white,scale=0.4] (E1) at (1.5,0.5-1/6) {};
        \node[draw,circle,fill=gray,scale=0.4] (E2) at (0.5+1/6,0.5-1/6) {};
        \node[draw,circle,fill=gray,scale=0.4] (E3) at (1.5,-0.5) {};
        \node[draw,circle,fill=brown,scale=0.4] (E4) at (0.5+1/6,-0.5) {};
        
        \path [thick,-latex] (E1) edge node[left] {} (E2);
        \path [thick,-latex] (E1) edge node[left] {} (E3);
        \path [thick,-latex] (E2) edge node[left] {} (E4);
        \path [thick,-latex] (E3) edge node[left] {} (E4);
 \end{tikzpicture}
 \caption{The complex $\cC^\infty(T_{3,3})$ on the left and $\cC^\infty(T'_{3,3})$ on the right. The $2$ on the central staircase is the multiplicity of the subcomplex. White, gray and brown dots represent Maslov gradings $0,-1$ and $-2$ respectively, while black dots represent the others.}
 \label{T33}
\end{figure}
We write $T_{3,3}$ when we orient the three components in the same direction, while $T'_{3,3}$ denotes the same link with the orientation reversed on one component.
From this picture we can easy compute the $\Upsilon$-functions: 
\[\Upsilon_{T_{3,3}'}(t)=0\quad t\in\left[0,2\right]
  \quad\quad\quad
  \Upsilon^*_{T_{3,3}'}(t)=\left\{\begin{aligned}
                                 &t\quad\quad\hspace{0.25cm} t\in[0,1]\\
                                 &2-t\quad t\in\left[1,2\right]
                                \end{aligned}\right.\:.\]
\[\Upsilon_{T_{3,3}}(t)=\left\{\begin{aligned}
                                 &-3t\quad\quad\hspace{0.3cm} t\in\left[0,\frac{2}{3}\right]\\
                                 &-2\quad\quad\hspace{0.4cm} t\in\left[\frac{2}{3},\frac{4}{3}\right]\\
                                 &-6+3t\quad t\in\left[\frac{4}{3},2\right]
                                \end{aligned}\right.
  \quad\quad\quad
  \Upsilon^*_{T_{3,3}}(t)=\left\{\begin{aligned}
                                 &-t\quad\quad\hspace{0.25cm} t\in\left[0,1\right]\\
                                 &-2+t\quad t\in\left[1,2\right]
                                \end{aligned}\right.\] 
                                \end{ex}
Finally, we show that the classical $\Upsilon$-invariants do not determine the $\mathcal F$-filtered isomorphism type of $\cH^\infty(L)$.
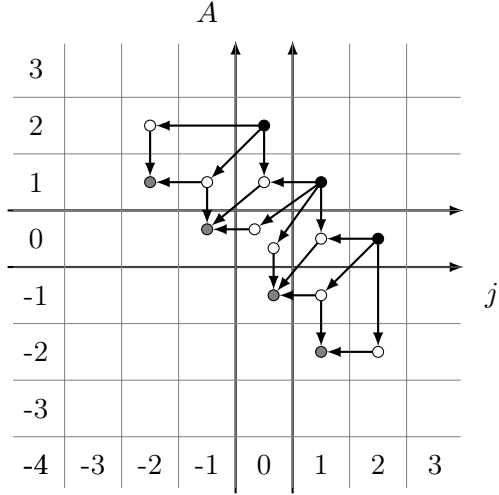
\begin{figure}[t]
 \centering
 \begin{tikzpicture}[scale =.75]
        \coordinate (Origin)   at (0,0);
        \coordinate (XAxisMin) at (-4,0);
        \coordinate (XAxisMax) at (4,0);
        \coordinate (YAxisMin) at (0,-4);
        \coordinate (YAxisMax) at (0,4);
        \draw [thick, black,-latex] (XAxisMin) -- (XAxisMax);
        \draw [thick, black,-latex] (YAxisMin) -- (YAxisMax);
        \draw [thick, black,-latex] (-4,1) -- (4,1);
        \draw [thick, black,-latex] (1,-4) -- (1,4);
        \clip (-3.9,-3.9) rectangle (4.99,4.99);
        \draw[style=help lines] (-4,-4) grid[step=1cm] (3.95,3.95);
        
        \foreach \x in {-4,...,3}
                    \node[draw=none,fill=none] at (\x+0.5,-3.5) {\x};
    	\foreach \y in {-4,...,3}     		
        			\node[draw=none,fill=none] at (-3.5,\y+0.5) {\y};
        
        \node[] at (-0.5,4.5) {$A$};
        \node[] at (4.5,-0.5) {$j$};
        
        \node[draw,circle,fill,scale=0.4]  (A1) at (0.5,2.5) {};
        \node[draw,circle,fill,scale=0.4]  (A2) at (1.5,1.5) {};
        \node[draw,circle,fill,scale=0.4]  (A3) at (2.5,0.5) {};
        \node[draw,circle,fill=white,scale=0.4]  (B1) at (-1.5,2.5) {};
        \node[draw,circle,fill=white,scale=0.4]  (B2) at (-0.5,1.5) {};
        \node[draw,circle,fill=white,scale=0.4]  (B3) at (0.5,1.5) {};
        \node[draw,circle,fill=white,scale=0.4]  (B4) at                 (0.5-1/6,0.5+1/6) {};
        \node[draw,circle,fill=white,scale=0.4]  (B5) at                 (0.5+1/6,0.5-1/6) {};
        \node[draw,circle,fill=white,scale=0.4]  (B6) at (1.5,0.5) {};
        \node[draw,circle,fill=white,scale=0.4]  (B7) at (1.5,-0.5) {};
        \node[draw,circle,fill=white,scale=0.4]  (B8) at (2.5,-1.5) {};
        \node[draw,circle,fill=gray,scale=0.4]  (C1) at (-1.5,1.5) {};
        \node[draw,circle,fill=gray,scale=0.4]  (C2) at (-0.5,0.5+1/6) {};
        \node[draw,circle,fill=gray,scale=0.4]  (C3) at (0.5+1/6,-0.5) {};
        \node[draw,circle,fill=gray,scale=0.4]  (C4) at (1.5,-1.5) {};
        
        \path [thick,-latex] (A1) edge node[left] {} (B1);
        \path [thick,-latex] (A1) edge node[left] {} (B2);
        \path [thick,-latex] (A1) edge node[left] {} (B3);
        \path [thick,-latex] (A2) edge node[left] {} (B3);
        \path [thick,-latex] (A2) edge node[left] {} (B4);
        \path [thick,-latex] (A2) edge node[left] {} (B5);
        \path [thick,-latex] (A2) edge node[left] {} (B6);
        \path [thick,-latex] (A3) edge node[left] {} (B6);
        \path [thick,-latex] (A3) edge node[left] {} (B7);
        \path [thick,-latex] (A3) edge node[left] {} (B8);
        \path [thick,-latex] (B1) edge node[left] {} (C1);
        \path [thick,-latex] (B2) edge node[left] {} (C1);
        \path [thick,-latex] (B2) edge node[left] {} (C2);
        \path [thick,-latex] (B3) edge node[left] {} (C2);
        \path [thick,-latex] (B4) edge node[left] {} (C2);
        \path [thick,-latex] (B5) edge node[left] {} (C3);
        \path [thick,-latex] (B6) edge node[left] {} (C3);
        \path [thick,-latex] (B7) edge node[left] {} (C3);
        \path [thick,-latex] (B7) edge node[left] {} (C4);
        \path [thick,-latex] (B8) edge node[left] {} (C4);
 \end{tikzpicture}
 \caption{The non-acyclic summand of the chain complex $CFK^\infty(K)$, where the knot $K$ is $T_{4,5}\#T^*_{2,3;2,5}\#T^*_{2,5}$.}
 \label{Figure9Livingston}
\end{figure}
In fact, take the knot $K=T_{4,5}\#T^*_{2,3;2,5}\#T^*_{2,5}$ whose homology is shown in Figure \ref{Figure9Livingston}, where $T_{2,3;2,5}$ is the $(2,5)$-cable of $T_{2,3}$. 
In \cite{Livingston} is proved that $\Upsilon_K(t)=\Upsilon^*_K(t)=0$ for every $t\in[0,2]$. 
On the other hand, it is easy to check that $\Upsilon_{V_0}(K)=-2$, where $V_{0}=\{(a,b)\:|\:a\leq 0,b\leq 0\}$, while $\Upsilon_{V_0}(\bigcirc)=0$.                       

\subsection{Symmetries}
\label{subsection:symmetries}
In this subsection we study some of the main properties of the $\Upsilon$-invariants. 
\begin{figure}[t]
 \centering
 \def\svgwidth{11cm}
\begingroup%
  \makeatletter%
  \providecommand\color[2][]{%
    \errmessage{(Inkscape) Color is used for the text in Inkscape, but the package 'color.sty' is not loaded}%
    \renewcommand\color[2][]{}%
  }%
  \providecommand\transparent[1]{%
    \errmessage{(Inkscape) Transparency is used (non-zero) for the text in Inkscape, but the package 'transparent.sty' is not loaded}%
    \renewcommand\transparent[1]{}%
  }%
  \providecommand\rotatebox[2]{#2}%
  \newcommand*\fsize{\dimexpr\f@size pt\relax}%
  \newcommand*\lineheight[1]{\fontsize{\fsize}{#1\fsize}\selectfont}%
  \ifx\svgwidth\undefined%
    \setlength{\unitlength}{2981.429019bp}%
    \ifx\svgscale\undefined%
      \relax%
    \else%
      \setlength{\unitlength}{\unitlength * \real{\svgscale}}%
    \fi%
  \else%
    \setlength{\unitlength}{\svgwidth}%
  \fi%
  \global\let\svgwidth\undefined%
  \global\let\svgscale\undefined%
  \makeatother%
  \begin{picture}(1,0.60819158)%
    \lineheight{1}%
    \setlength\tabcolsep{0pt}%
    \put(0,0){\includegraphics[width=\unitlength,page=1]{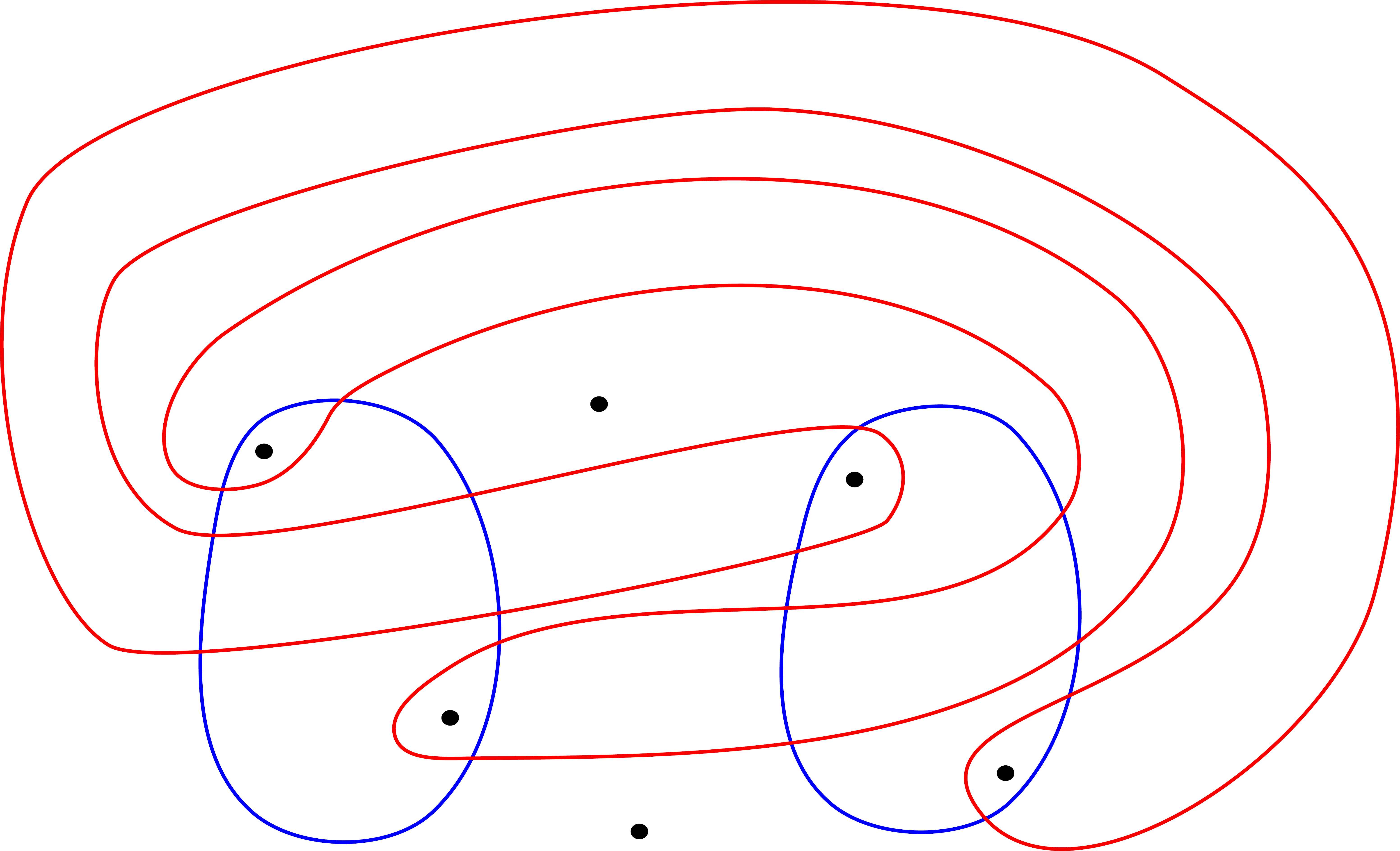}}%
    \put(0.21739384,0.26760667){\color[rgb]{0,0,0}\makebox(0,0)[lt]{\lineheight{1.25}\smash{\begin{tabular}[t]{l}$z_1$\end{tabular}}}}%
    \put(0.36042209,0.08792294){\color[rgb]{0,0,0}\makebox(0,0)[lt]{\lineheight{1.25}\smash{\begin{tabular}[t]{l}$w_1$\end{tabular}}}}%
    \put(0.44882648,0.32223052){\color[rgb]{0,0,0}\makebox(0,0)[lt]{\lineheight{1.25}\smash{\begin{tabular}[t]{l}$z_2$\end{tabular}}}}%
    \put(0.48116955,0.0131745){\color[rgb]{0,0,0}\makebox(0,0)[lt]{\lineheight{1.25}\smash{\begin{tabular}[t]{l}$w_2$\end{tabular}}}}%
    \put(0.59401092,0.23885727){\color[rgb]{0,0,0}\makebox(0,0)[lt]{\lineheight{1.25}\smash{\begin{tabular}[t]{l}$z_3$\end{tabular}}}}%
    \put(0.74853892,0.0376115){\color[rgb]{0,0,0}\makebox(0,0)[lt]{\lineheight{1.25}\smash{\begin{tabular}[t]{l}$w_3$\end{tabular}}}}%
  \end{picture}%
\endgroup%
   
 \caption{A Heegaard diagram for the link $T_{3,3}$. The $\alpha$-curves are red, while the $\beta$'s are blue.}
 \label{T33diagram}
\end{figure}
We start from this proposition from \cite{OSlinks}.
\begin{prop}[Ozsv\'ath and Szab\'o]
 The $\mathcal F$-filtered chain homotopy type $\cC^\infty(L)$ of a link Floer complex is independent of the (global) orientation of $L$.
\end{prop}
In particular, we can identify the homology group of a link $L$ and its reverse.
\begin{cor}
 \label{cor:reflection1}
 There is an $\mathcal F$-filtered isomorphism $\cH^\infty(L)\leftarrow\joinrel\rightarrow\cH^\infty(-L)$.
 In particular, we have $\Upsilon_S(L)=\Upsilon_S(-L)$ and $\Upsilon_S^*(L)=\Upsilon_S^*(-L)$ for every centered south-west region $S$ of $\mathbb R^2$.
\end{cor}
We remind the reader that this result is not true if we reverse the orientation only on some of the components of $L$, as we saw in the previous subsection with the link $T_{3,3}$.

Say $-S$ is the south-west region obtained from $S$ after applying the reflection $r$ of the plane with respect to the line $\{A-j=0\}$.
We prove the following property.
\begin{teo}
 \label{teo:reflection}
 We have $\Upsilon_S(L)=\Upsilon_{-S}(L)$ and $\Upsilon_S^*(L)=\Upsilon_{-S}^*(L)$ for every centered south-west region $S$ of $\mathbb R^2$. In particular, one has $\Upsilon_L(t)=\Upsilon_L(2-t)$ and $\Upsilon_L^*(t)=\Upsilon_L^*(2-t)$ for every $t\in[0,2]$.
\end{teo}
\begin{proof}
 Since a chain complex for $-L$ is obtained by switching the role of $\textbf w$ and $\textbf z$ in a Heegaard diagram for $L$, and then of the filtrations $\mathcal A$ and $j$, Corollary \ref{cor:reflection1} tells us that $\cC^\infty(L)$ is symmetric under $r$ up to homotopy. Moreover, this symmetry is chain homotopic to the identity by \cite[Lemma 4.6]{Sarkar2} and the claim follows.
 
 For the second part of the statement, we just need to observe that the reflected south-west region $-A_t$ corresponds to $A_{2-t}$.
\end{proof}
With this theorem set, from now on we consider the $\Upsilon$-functions as defined on $[0,1]$, since their values on $[1,2]$ are then determined automatically. 

Now we want to study the relation between the $\Upsilon$'s of $L$ and its mirror image.
We recall that, in Subsection \ref{subsection:duality}, we defined $\iota S$ as the complement of the region obtained from $S$ by applying a central symmetry. Then we say that $\overline{\iota S}$ is the topological closure of $\iota S$. 
\begin{prop}
 \label{prop:mirror}
 For an $n$-component link $L$ we have that \[\Upsilon_S(L^*)=-\Upsilon_{\overline{\iota S}}^*(L)\] for every centered south-west region $S$ of $\mathbb R^2$. In particular, we obtain  $\Upsilon_{L^*}(t)=-\Upsilon^*_L(t)$ for every $t\in[0,1]$ and for a knot $K$ one has $\Upsilon_S(K^*)=-\Upsilon_{\overline{\iota S}}(K)$. 
\end{prop}
\begin{proof}
 We apply Theorem \ref{teo:mirror} to argue that there is an identification \[\mathcal F^S\cH^\infty_0(L^*)\longleftarrow\joinrel\longrightarrow(\mathcal F^*)^S\cH^\infty_{n-1}(L)^*=\Ann\mathcal F^{\iota S}\cH^\infty_{1-n}(L)\] that preserves the containment relations. Hence, we only need to use the definition of $\Upsilon$:
 \[\begin{aligned}
    \Upsilon_S(L^*)=&\max_{k\in\mathbb R}\left\{k\:|\:\mathcal F^{S_k}\cH^\infty_0(L^*)\supset\mathcal F^{\{j\leq0\}}\cH^\infty_0(L^*)\right\}=\\ =&\max_{k\in\mathbb R}\left\{k\:|\:\Ann\mathcal F^{\iota S_k}\cH^\infty_{1-n}(L)\supset\Ann\mathcal F^{\{j\leq-1\}}\cH^\infty_{1-n}(L)\right\}=\\ =&\max_{k\in\mathbb R}\left\{k\:|\:\mathcal F^{(\iota S)_{-k}}\cH^\infty_{1-n}(L)\subset\mathcal F^{\{j\leq-1\}}\cH^\infty_{1-n}(L)\right\}=\\
    =&-\min_{k\in\mathbb R}\left\{k\:|\:\mathcal F^{(\iota S)_{k}}\cH^\infty_{1-n}(L)\subset\mathcal F^{\{j\leq-1\}}\cH^\infty_{1-n}(L)\right\}=\end{aligned}\]
    \[=-\max_{k\in\mathbb R}\left\{k\:|\:\mathcal F^{\overline{\iota S}_{k}}\cH^\infty_{1-n}(L)\not\subset\mathcal F^{\{j\leq-1\}}\cH^\infty_{1-n}(L)\right\}=-\Upsilon^*_{\overline{\iota S}}(L)\]
    for every centered south-west region $S$ in $\mathbb R^2$.
   
 The third claim is trivial, while for the second one we note that $\overline{\iota A_t}=A_t$ for every $t\in[0,1]$. 
\end{proof}
We observe that the south-west regions $A_t$ are not the only $S$ such that $\overline{\iota S}=S$ as we see from Figure \ref{SWregionssymmetric}.

Let us recall that the homology group $\widehat{HFL}(L)$ (resp. $\widehat{\mathcal{HFL}}(L)$) is defined as the bigraded homology of the associated graded object (resp. the $\mathcal A$-filtered graded homology) of the complex $\widehat{CFL}(L)$, given by setting $U=0$ in $\cC^\infty(L)$, see \cite{Cavallo,OSlinks} for details.
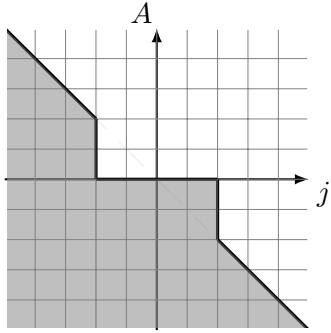
\begin{figure}[t]
 \centering
 \begin{tikzpicture}[scale =.4]
        \coordinate (Origin)   at (0,0);
        \coordinate (XAxisMin) at (-5,0);
        \coordinate (XAxisMax) at (5,0);
        \coordinate (YAxisMin) at (0,-5);
        \coordinate (YAxisMax) at (0,5);
        \draw [thick, black,-latex] (XAxisMin) -- (XAxisMax);
        \draw [thick, black,-latex] (YAxisMin) -- (YAxisMax);
        \clip (-4.9,-4.9) rectangle (5.99,5.99);
        \draw[style=help lines] (-5,-5) grid[step=1cm] (4.95,4.95);
        
        \node[] at (-0.5,5.5) {$A$};
        \node[] at (5.5,-0.5) {$j$};
        
        \draw [very thick, black] (-5,5) -- (-2,2);
        \draw [very thick, black] (-2,2) -- (-2,0);
        \draw [very thick, black] (-2,0) -- (2,0);
        \draw [very thick, black] (2,0) -- (2,-2);
        \draw [very thick, black] (2,-2) -- (5,-5);
        
        \draw[fill, opacity=.5, gray] (-2,2) rectangle (-5,-5);
        \draw[fill, opacity=.5, gray] (-2,-5) rectangle (2,0);
        \fill [gray, opacity=.5, domain=-5:5, variable=\x]
      (-5,2)
      -- plot ({\x}, {-\x})
      -- (-2,2)
      -- cycle;
      \fill [gray, opacity=.5, domain=2:5, variable=\x]
      (2,-5)
      -- plot ({\x}, {-\x})
      -- (5,-5)
      -- cycle;
 \end{tikzpicture}
 \caption{The centered south-west region $T$ in the picture is such that $T=\overline{\iota T}$.}
 \label{SWregionssymmetric}
 \end{figure}
\begin{lemma}
 \label{lemma:generator}
 If a cycle in $\mathcal F^{\{j\leq0\}}\cC^\infty(L)$ is a generator of the homology group $\cH^\infty(L)$, and its homology class has minimal $j$-level zero, then its projection to $\widehat{CFL}(L)$ is a generator of $\widehat{\mathcal{HFL}}(L)$. 
\end{lemma}
\begin{proof}
 By \cite[Lemma 4.5]{Rasmussen} we know that, up to changing basis, the complex $\cC^\infty(L)$ is such that the differential of the bigraded object associated to $\widehat{CFL}(L)$ is zero. Therefore, if we pick a generator with minimal $j$-level zero then its projection cannot be zero in $\widehat{\mathcal{HFL}}(L)$, because clearly it would be homologous to an element of $U\cdot\mathcal F^{\{j\leq0\}}\cC^\infty(L)=\mathcal F^{\{j\leq-1\}}\cC^\infty(L)$.
\end{proof}
We use the mirror image symmetry to prove the following proposition. We assume the reader to be familiar with the definition of the concordance invariants $\tau(L)$ and $\tau^*(L)$, given by the author in \cite{Cavallo}.  
\begin{prop}
 For a link $L$ we have that \[\tau(L)=-\Upsilon_L'(0)\quad\text{ and }\quad\tau^*(L)=-(\Upsilon_L^*)'(0)\:.\] Furthermore, each slope of $\Upsilon_L(t)$ and $\Upsilon_L^*(t)$ is an integer $s$ such that the Alexander grading subgroup $\widehat{HFL}_{*,s}(L)$ is non-zero and if $t_0\in(0,1)$ is a point where the slope changes from $s_1$ to $s_2$ then $t_0\in\frac{\mathbb Z}{|s_2-s_1|}$ and $|s_2-s_1|>2$.
\end{prop}
\begin{proof}
 We prove the first part of the statement. We take $t\in[0,\epsilon)$ with $\epsilon$ very small and we show that for such $t$'s one has $\Upsilon_L(t)\leq-t\cdot\tau(L)$. Suppose that the homology class of $x$ is a generator of $\mathcal F^{\{j\leq0\}}\cH^\infty_0(L)$. By Lemma \ref{lemma:generator} we have that $\overline x$, the projection of $x$ to $\widehat{CFL}_{0,*}(L)$, is a generator of $\widehat{\mathcal{HFL}}_{0}(L)$. Hence, assuming $\Upsilon_L(t)>-t\cdot\tau(L)$ contradicts the fact that $\tau(L)$ is the minimum $\mathcal A$-level $s$ such that $\mathcal A^s\widehat{\mathcal{HFL}}_0(L)$ has dimension one, see \cite{Cavallo}.
 
 We now show that $\Upsilon_L(t)\geq-t\cdot\tau(L)$. In fact, the same argument we used before also shows that $\Upsilon_{L^*}^*(t)\leq-t\cdot\tau^*(L^*)$ for $t\in[0,\epsilon)$ and then $-\Upsilon_L(t)\leq-t(-\tau(L))$ from Proposition \ref{prop:mirror} and the symmetry properties of $\tau^*$, see \cite{Cavallo}. This proves the claim; in fact, the version for the $\Upsilon^*$-function can be proved applying Proposition \ref{prop:mirror}.
 
 The second part of the proposition follows from the same proof of \cite[Proposition 1.4]{OSSz} and \cite[Observation 2.2]{FellerAru}.
\end{proof}
Using this result we immediately compute the $\Upsilon$-functions for the Hopf links $H_\pm$. In fact $H_\pm$ is a non-split $2$-component link that bounds an annulus in $S^3$. Since $\widehat{HFL}$ detects the Thurston norm \cite[Theorem 1.1]{Ni}, this implies that $\widehat{HFL}_{*,s}(H_\pm)$ is non-zero only when $s=-1,0,1$ and then $\Upsilon_{H_\pm}$ and $\Upsilon^*_{H_\pm}$ are determined by the $\tau$-invariants, which are computed in \cite[Corollary 3.7]{Cavallo}. 
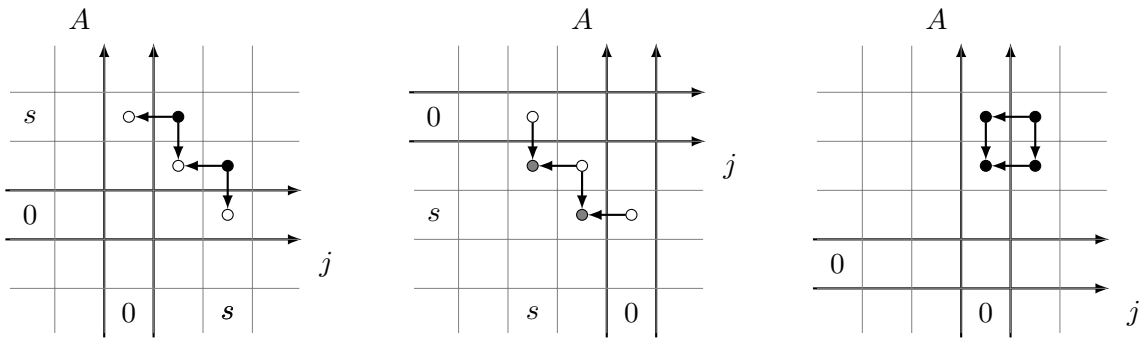
\begin{figure}[t]
 \centering
 \begin{tikzpicture}[scale =.65]
        \coordinate (Origin)   at (0,0);
        \coordinate (XAxisMin) at (-2,0);
        \coordinate (XAxisMax) at (4,0);
        \coordinate (YAxisMin) at (0,-2);
        \coordinate (YAxisMax) at (0,4);
        \draw [thick, black,-latex] (XAxisMin) -- (XAxisMax);
        \draw [thick, black,-latex] (YAxisMin) -- (YAxisMax);
        \draw [thick, black,-latex] (-2,1) -- (4,1);
        \draw [thick, black,-latex] (1,-2) -- (1,4);
        \clip (-1.9,-1.9) rectangle (4.99,4.99);
        \draw[style=help lines] (-2,-2) grid[step=1cm] (3.95,3.95);
        
        \node[] at (-1.5,2.5) {$s$}; \node[] at (0.5,-1.5) {$0$};
        \node[] at (2.5,-1.5) {$s$}; \node[] at (-1.5,0.5) {$0$};
        \node[] at (2.5,-1.5) {$s$};
        \node[] at (-0.5,4.5) {$A$};
        \node[] at (4.5,-0.5) {$j$};
        
        \node[draw,circle,fill,scale=0.4]  (A1) at (1.5,2.5) {};
        \node[draw,circle,fill,scale=0.4]  (A2) at (2.5,1.5) {};
        \node[draw,circle,fill=white,scale=0.4] (B) at (0.5,2.5) {};
        \node[draw,circle,fill=white,scale=0.4] (C) at (1.5,1.5) {};
        \node[draw,circle,fill=white,scale=0.4] (D) at (2.5,0.5) {};
        
        \path [thick,-latex] (A1) edge node[left] {} (B);
        \path [thick,-latex] (A1) edge node[left] {} (C);
        \path [thick,-latex] (A2) edge node[left] {} (C);
        \path [thick,-latex] (A2) edge node[left] {} (D);
 \end{tikzpicture}
 \hspace{.5cm}
 \begin{tikzpicture}[scale =.65]
        \coordinate (Origin)   at (0,0);
        \coordinate (XAxisMin) at (-4,0);
        \coordinate (XAxisMax) at (2,0);
        \coordinate (YAxisMin) at (0,-4);
        \coordinate (YAxisMax) at (0,2);
        \draw [thick, black,-latex] (XAxisMin) -- (XAxisMax);
        \draw [thick, black,-latex] (YAxisMin) -- (YAxisMax);
        \draw [thick, black,-latex] (-4,1) -- (2,1);
        \draw [thick, black,-latex] (1,-4) -- (1,2);
        \clip (-3.9,-3.9) rectangle (2.99,2.99);
        \draw[style=help lines] (-4,-4) grid[step=1cm] (1.95,1.95);
        
        \node[] at (-3.5,0.5) {$0$}; \node[] at (-1.5,-3.5) {$s$};
        \node[] at (0.5,-3.5) {$0$}; \node[] at (-3.5,-1.5) {$s$};
        \node[] at (-0.5,2.5) {$A$};
        \node[] at (2.5,-0.5) {$j$};
        
        \node[draw,circle,fill=white,scale=0.4] (A) at (-1.5,0.5) {};
        \node[draw,circle,fill=white,scale=0.4] (B) at (-0.5,-0.5) {};
        \node[draw,circle,fill=white,scale=0.4] (C) at (0.5,-1.5) {};
        \node[draw,circle,fill=gray,scale=0.4] (D1) at (-1.5,-0.5) {};
        \node[draw,circle,fill=gray,scale=0.4] (D2) at (-0.5,-1.5) {};
        
        \path [thick,-latex] (A) edge node[left] {} (D1);
        \path [thick,-latex] (B) edge node[left] {} (D1);
        \path [thick,-latex] (B) edge node[left] {} (D2);
        \path [thick,-latex] (C) edge node[left] {} (D2);
 \end{tikzpicture}
 \hspace{.5cm}
 \begin{tikzpicture}[scale =.65]
        \coordinate (Origin)   at (0,0);
        \coordinate (XAxisMin) at (-3,0);
        \coordinate (XAxisMax) at (3,0);
        \coordinate (YAxisMin) at (0,-1);
        \coordinate (YAxisMax) at (0,5);
        \draw [thick, black,-latex] (XAxisMin) -- (XAxisMax);
        \draw [thick, black,-latex] (YAxisMin) -- (YAxisMax);
        \draw [thick, black,-latex] (-3,1) -- (3,1);
        \draw [thick, black,-latex] (1,-1) -- (1,5);
        \clip (-2.9,-0.9) rectangle (3.99,5.99);
        \draw[style=help lines] (-3,-1) grid[step=1cm] (2.95,4.95);
        
        \node[] at (-2.5,0.5) {$0$};
        \node[] at (0.5,-0.5) {$0$};
        \node[] at (-0.5,5.5) {$A$};
        \node[] at (3.5,-0.5) {$j$};
        
        \node[draw,circle,fill,scale=0.4] (A) at (0.5,3.5) {};
        \node[draw,circle,fill,scale=0.4] (B) at (1.5,2.5) {};
        \node[draw,circle,fill,scale=0.4] (C) at (0.5,2.5) {};
        \node[draw,circle,fill,scale=0.4] (D) at (1.5,3.5) {};
        
        \path [thick,-latex] (A) edge node[left] {} (C);
        \path [thick,-latex] (B) edge node[left] {} (C);
        \path [thick,-latex] (D) edge node[left] {} (A);
        \path [thick,-latex] (D) edge node[left] {} (B);
 \end{tikzpicture}
 \caption{A positive staircase (left), a negative staircase (middle) and an acyclic square (right). The acyclic subcomplex of $\cC^\infty(L)$, when $L$ is as in Theorem \ref{teo:alternating}, is the direct sum of acyclic squares.}
 \label{Alternating}
\end{figure}
Therefore, we obtain
\[\Upsilon_{H_+}(t)=-t\:,\quad\quad\Upsilon_{H_-}^*(t)=t\quad\text{ and }\quad\Upsilon^*_{H_+}(t)=\Upsilon_{H_-}(t)=0\] for every $t\in[0,1]$. 
 
We conclude this subsection by stating a result of Petkova \cite{Petkova} that allows us to determine $\cC^\infty(L)$ for every non-split alternating link. 
We recall that an $n$-component link $L$ is \emph{$\widehat{HFL}$-thin} if its homology group $\widehat{HFL}_{d,s}(L)$ is supported on the line $s=d+\frac{n-1-\sigma(L)}{2}$, where $\sigma(L)$ is the signature of $L$.
\begin{teo}[Petkova]
 \label{teo:alternating}
 Suppose that the link $L$ has $n$ components and it is $\widehat{HFL}$-thin. Then the chain complex $\cC^\infty(L)$ is given as the direct sum of some $\F[U,U^{-1}]$-subcomplexes as in Figure \ref{Alternating}.
 More specifically, for every \[s\in\left\{\frac{n-1-\sigma(L)}{2}-k\right\}\quad\text{ with }\quad k=0,...,n-1\:,\] we have $\binom{n-1}{k}$ positive (resp. negative)
 staircases when $s$ is positive (resp. negative). Moreover, the acyclic subcomplex is determined by \[\chi\left(\widehat{HFL}(L)\right)(t,t^{-1})=\sum_{d\in\mathbb Z}(-1)^d\dim_\F\widehat{HFL}_{d,s}(L)\cdot t^{s}=\left(t^{\frac{1}{2}}-t^{-\frac{1}{2}}\right)^{n-1}\cdot\nabla_L\left(t^{\frac{1}{2}}-t^{-\frac{1}{2}}\right)\:,\] where $\nabla_L(z)$ is the Conway normalization of the Alexander polynomial of $L$.
\end{teo}
Note that quasi-alternating links (and then non-split alternating links) are $\widehat{HFL}$-thin, see \cite{Cavallo,Book}. In Figure \ref{Whitehead} we show a Whitehead link and its corresponding complex. 
 
\subsection{Connected sums and disjoint unions} 
\label{subsection:connected}
It follows from the work of Ozsv\'ath and Szab\'o that the chain complex for a connected sum of the links $L_1$ and $L_2$ is given by the tensor product between the ones of $L_1$ and $L_2$.
\begin{teo}[Ozsv\'ath and Szab\'o]
 \label{teo:connected_sum}
 Given two links $L_1$ and $L_2$, 
 denote with $L_1\#_{i,j}L_2$ the connected sum performed on the $i$-th and the $j$-component of $L_1$ and $L_2$ respectively. 
 Then we have that \[\cC^\infty(L_1\#_{i,j}L_2)\cong\cC^\infty(L_1)\otimes_{\F[U,U^{-1}]}\cC^\infty(L_2)\:.\] In particular, 
 the complex $\cC^\infty(L_1\#L_2)$ does not depend on $i$ and $j$.  
\end{teo}
Since $\F[U,U^{-1}]$ is a principal ideal domain, 
using the K\"unneth formula and Theorem \ref{teo:dim} on the identification in Theorem \ref{teo:connected_sum} gives
\[\mathcal F^{\{j\leq0\}}\cH^\infty_0(L_1\#L_2)\cong_\F\mathcal F^{\{j\leq0\}}\cH^\infty_0(L_1)\otimes_\F\mathcal F^{\{j\leq0\}}\cH^\infty_0(L_2)\] and \[\dfrac{\mathcal F^{\{j\leq0\}}\cH^\infty_{2-n_1-n_2}(L_1\#L_2)}{\mathcal F^{\{j\leq-1\}}\cH^\infty_{2-n_1-n_2}(L_1\#L_2)}\cong_\F\dfrac{\mathcal F^{\{j\leq0\}}\cH^\infty_{1-n_1}(L_1)}{\mathcal F^{\{j\leq-1\}}\cH^\infty_{1-n_1}(L_1)}\otimes_\F\dfrac{\mathcal F^{\{j\leq0\}}\cH^\infty_{1-n_2}(L_2)}{\mathcal F^{\{j\leq-1\}}\cH^\infty_{1-n_2}(L_2)}\:,\] where $n_i$ is the number of components of $L_i$ and we recall that $n_1+n_2-1$ is the one of $L_1\#L_2$.
\begin{figure}
    \centering
    \includegraphics[width=7.5cm]{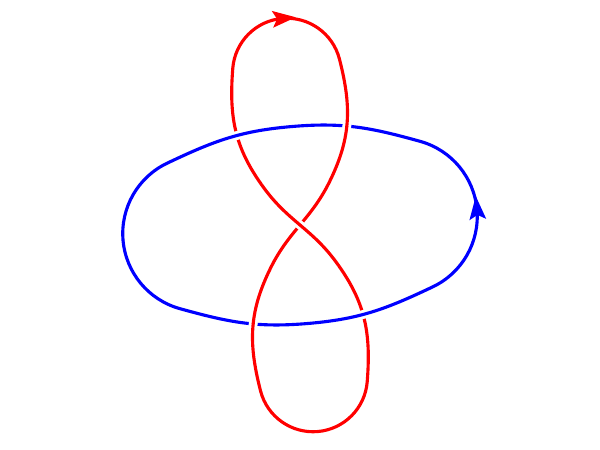}
    \hspace{1cm}
 \begin{tikzpicture}[scale =.75]
        \coordinate (Origin)   at (0,0);
        \coordinate (XAxisMin) at (-3,0);
        \coordinate (XAxisMax) at (4,0);
        \coordinate (YAxisMin) at (0,-3);
        \coordinate (YAxisMax) at (0,4);
        \draw [thick, black,-latex] (XAxisMin) -- (XAxisMax);
        \draw [thick, black,-latex] (YAxisMin) -- (YAxisMax);
        \draw [thick, black,-latex] (-3,1) -- (4,1);
        \draw [thick, black,-latex] (1,-3) -- (1,4);
        \clip (-2.9,-2.9) rectangle (4.99,4.99);
        \draw[style=help lines] (-3,-3) grid[step=1cm] (3.95,3.95);
        
        \foreach \x in {-3,...,3}
                    \node[draw=none,fill=none] at (\x+0.5,-2.5) {\x};
    	\foreach \y in {-3,...,3}     		
        			\node[draw=none,fill=none] at (-2.5,\y+0.5) {\y};
        
        \node[] at (-0.5,4.5) {$A$};
        \node[] at (4.5,-0.5) {$j$};
        
        \node[draw,circle,fill=gray,scale=0.4] at (1/3,0.5+1/6) {};
        
        \node[draw,circle,fill=white,scale=0.4] (A1) at (0.5+1/6,1.5-1/6) {};
        \node[draw,circle,fill=white,scale=0.4] (A2) at (1.5,0.5+1/6) {};
        \node[draw,circle,fill,scale=0.4] (B) at (1.5,1.5-1/6) {};
        
        \path [thick,-latex] (B) edge node[left] {} (A1);
        \path [thick,-latex] (B) edge node[left] {} (A2);
        
        \node[draw,circle,fill,scale=0.4] (C1) at (1.5,2.5) {};
        \node[draw,circle,fill,scale=0.4] (C2) at (0.5+1/6,2.5) {};
        \node[draw,circle,fill,scale=0.4] (C3) at (1.5,1.5+1/6) {};
        \node[draw,circle,fill=white,scale=0.4] (C4) at (0.5+1/6,1.5+1/6) {};
        
        \path [thick,-latex] (C1) edge node[left] {} (C2);
        \path [thick,-latex] (C1) edge node[left] {} (C3);
        \path [thick,-latex] (C2) edge node[left] {} (C4);
        \path [thick,-latex] (C3) edge node[left] {} (C4);
        
        \node[draw,circle,fill=gray,scale=0.4] (D1) at (0.5-1/6,0.5-1/6) {};
        \node[draw,circle,fill=brown,scale=0.4] (D2) at (-0.5,0.5-1/6) {};
        \node[draw,circle,fill=brown,scale=0.4] (D3) at (0.5-1/6,-0.5) {};
        \node[draw,circle,fill,scale=0.4] (D4) at (-0.5,-0.5) {};
        
        \path [thick,-latex] (D1) edge node[left] {} (D2);
        \path [thick,-latex] (D1) edge node[left] {} (D3);
        \path [thick,-latex] (D2) edge node[left] {} (D4);
        \path [thick,-latex] (D3) edge node[left] {} (D4);
        
        \node[draw,circle,fill=white,scale=0.4] (E1) at (1.5,0.5-1/6) {};
        \node[draw,circle,fill=gray,scale=0.4] (E2) at (0.5+1/6,0.5-1/6) {};
        \node[draw,circle,fill=gray,scale=0.4] (E3) at (1.5,-0.5) {};
        \node[draw,circle,fill=brown,scale=0.4] (E4) at (0.5+1/6,-0.5) {};
        
        \path [thick,-latex] (E1) edge node[left] {} (E2);
        \path [thick,-latex] (E1) edge node[left] {} (E3);
        \path [thick,-latex] (E2) edge node[left] {} (E4);
        \path [thick,-latex] (E3) edge node[left] {} (E4);
 \end{tikzpicture}
 \caption{The complex $\cC^\infty(W)$ (right) of the Whitehead link $W$ (left).}
 \label{Whitehead}
\end{figure}
Furthermore, if the homology classes of $x_i$ are generators for $\mathcal F^{\{j\leq0\}}\cH^\infty_0(L_i)$ then $[x_1\otimes x_2]$ is a generator of the homology group $\mathcal F^{\{j\leq0\}}\cH^\infty_0(L_1\#L_2)$. In the same way, if $y_i$ is such that $[y_i]$ is a generator of $\cH^\infty_{1-n_i}(L_i)$, with minimal $j$-level zero, then $[y_1\otimes y_2]$ is a generator of $\cH^\infty_{2-n_1-n_2}(L_1\#L_2)$ and its minimal algebraic level is again zero.

We can now study how the $\Upsilon$-invariants behave under connected sums. For every centered south-west region $S$ of $\mathbb R^2$ we define \[\text{env}(S)=\left\{(j,A)\in\mathbb R^2\:|\:j=a_1+a_2\text{ and }A=b_1+b_2\:,\text{ where }(a_i,b_i)\in S\text{ for }i=1,2\right\}\:.\]
Clearly, the region $\text{env}(S)$ is still a south-west region (unless it coincides with the whole $\mathbb R^2$) and $S\subset\text{env}(S)$.
Moreover, we take $h(S)\in\mathbb Z_{\geq0}\cup\{+\infty\}$ as \[\inf_{k\in\N}\{k\:|\:S_{-k}\supset\text{env}(S)\}\] and we state the following proposition.
\begin{prop}
 \label{prop:super-additive}
 Let us consider a link $L_i$ with $n_i$ components for $i=1,2$ and $S$ a centered south-west region of $\mathbb R^2$. We have that \[\Upsilon_S(L_1\#L_2)\geq\Upsilon_S(L_1)+\Upsilon_S(L_2)-h(S)\] and \[\Upsilon^*_S(L_1\#L_2)\geq\Upsilon^*_S(L_1)+\Upsilon^*_S(L_2)-h(S)\:.\] In particular, if $S=\emph{env}(S)$ then the $\Upsilon$'s and $\Upsilon^*$'s are super-additive under connected sums. 
\end{prop}
\begin{proof}
 The proof of the two inequalities is exactly the same; hence, we only do the first one. From what we said at the beginning of the subsection we can take $x$ and $y$, such that their homology classes are generators of $\mathcal F^{\{j\leq0\}}\cH^\infty_0(L_i)$ for $i=1,2$, in the region $S_{\Upsilon_S(L_i)}=S_{\gamma_i}$ and we obtain that $[x\otimes y]$ is a generator of $\mathcal F^{\{j\leq0\}}\cH^\infty_0(L_1\#L_2)$ and $x\otimes y\in\mathcal F^{\text{env}(S)_{\gamma_1+\gamma_2}}\cC^\infty_0(L_1\#L_2)$. Therefore, from the definition of $h(S)$ it follows that \[\text{env}(S)_{\gamma_1+\gamma_2}\subset S_{\gamma_1+\gamma_2-h(S)}\] and $x\otimes y\in\mathcal F^{S_{\gamma_1+\gamma_2-h(S)}}cCFL^\infty_0(L_1\#L_2)$ proving the inequality. 
\end{proof}
There are examples of south-west regions $S$ with $h(S)\neq0$ and $\Upsilon$ is not super-additive. 
Take the region $V_1=\{(j,A)\in\mathbb R^2\:|\:j\leq0,A\leq1\}$, then
\[-4=\Upsilon_{V_1}(T_{2,3}\#T_{2,7})<\Upsilon_{V_1}(T_{2,3})+\Upsilon_{V_1}(T_{2,7})=0+(-2)=-2\:.\]
\begin{cor}
 \label{cor:additive}
 If a centered south-west region $S$ is such that $\overline{\iota S}=S$ and $h(S)=0$ then \[\Upsilon_S(L_1\#L_2)=\Upsilon_S(L_1)+\Upsilon_S(L_2)\quad\text{ and }\quad\Upsilon^*_S(L_1\#L_2)=\Upsilon^*_S(L_1)+\Upsilon^*_S(L_2)\] for every links $L_1$ and $L_2$. In particular, this holds for the classical $\Upsilon$'s functions.
\end{cor}
\begin{proof}
 From Propositions \ref{prop:mirror} and \ref{prop:super-additive} we have \[\Upsilon_S(L_1)+\Upsilon_S(L_2)-h(S)\leq\Upsilon_S(L_1\#L_2)\leq\Upsilon_S(L_1)+\Upsilon_S(L_2)+h(\overline{\iota S})\:.\] 
 The claim follows by using the assumption that $h(S)=h(\overline{\iota S})=0$. The same proof works for $\Upsilon^*$.
\end{proof}
We observe that there are centered south-west regions, different from the $A_t$'s, for which $h(S)=0$ and then their $\Upsilon$-invariants are super-additive, see Figure \ref{Super-additive}. 
\begin{ex}
Corollary \ref{cor:additive} gives that for the positive and negative Hopf link one has $\Upsilon_{H_+\#H_-}(t)=-t$ and $\Upsilon^*_{H_+\#H_-}(t)=t$ for every $t\in[0,1]$.
In particular, we have an example when $\cC^\infty(L)\otimes\cC^\infty(L^*)$ is not locally equivalent to the chain complex of an unlink; in fact, it is $\Upsilon_{\bigcirc_m}=0$ for every $m\in\N$.
\end{ex} 
The disjoint union of two links can be seen as a special connected sum. In fact, the link $L_1\sqcup L_2$ is isotopic to $L_1\#\bigcirc_2\#L_2$, where the two connected sums are performed on different components of the unlink $\bigcirc_2$.
\begin{prop}
 \label{prop:disjoint}
 The chain complex of the link $L_1\sqcup L_2$ is given by \[\begin{aligned}\cC^\infty(L_1\sqcup L_2)\cong\:&\cC^\infty(L_1\#L_2)\otimes_{\F[U,U^{-1}]}\cC^\infty(\bigcirc_2)\cong\\
 \cong\:&\cC^\infty(L_1\#L_2)\oplus\cC^\infty(L_1\#L_2)\lkhov 1\rkhov\:,\end{aligned}\] where $\lkhov\cdot\rkhov$ denotes a shift in the Maslov grading.
\end{prop}
\begin{proof}
 It is easy to compute that $\cC^\infty_*(\bigcirc_2)\cong\F[U,U^{-1}]_{(0)}\oplus\F[U,U^{-1}]_{(-1)}$. 
 \begin{figure}[t]
 \centering
 \begin{tikzpicture}[scale =.4]
        \coordinate (Origin)   at (0,0);
        \coordinate (XAxisMin) at (-5,0);
        \coordinate (XAxisMax) at (5,0);
        \coordinate (YAxisMin) at (0,-5);
        \coordinate (YAxisMax) at (0,5);
        \draw [thick, black,-latex] (XAxisMin) -- (XAxisMax);
        \draw [thick, black,-latex] (YAxisMin) -- (YAxisMax);
        \clip (-4.9,-4.9) rectangle (5.99,5.99);
        \draw[style=help lines] (-5,-5) grid[step=1cm] (4.95,4.95);
        
        \node[] at (-0.5,5.5) {$A$};
        \node[] at (5.5,-0.5) {$j$};
        \node[above left] at (-3,-3) {$T$};
        
        \draw [very thick, black] (-5,5) -- (0,0);
        \draw [very thick, black] (0,0) -- (0,-5);
       
        \draw[fill, opacity=.5, gray] (0,0) rectangle (-5,-5);
        \fill [gray, opacity=.5, domain=-5:0, variable=\x]
      (-5,0)
      -- plot ({\x}, {-\x})
      -- (0,0)
      -- cycle;
 \end{tikzpicture}
 \caption{The centered south-west region $T$ is such that $h(T)=0$.}
 \label{Super-additive}
\end{figure}
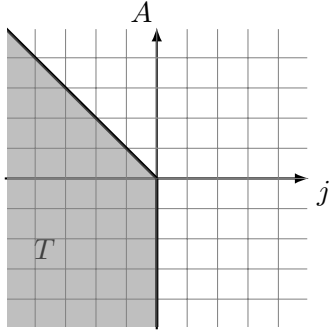
 Hence, the claim follows from Theorem \ref{teo:connected_sum}.
\end{proof}
Note that, since the chain complex for the connected sum is independent of the choice of the components, we have that $\cC^\infty(L_1\sqcup L_2)=\cC^\infty((L_1\#L_2)\sqcup\bigcirc)$; in other words, there is an identification between the chain complexes for the disjoint union and the link gotten by adding an unknot to any connected sum of $L_1$ and $L_2$.
\begin{cor}
 Given two links $L_1$ and $L_2$ we have that \[\Upsilon_S(L_1\sqcup L_2)=\Upsilon_S(L_1\#L_2)\hspace{1cm}\text{ and }\hspace{1cm}\Upsilon^*_S(L_1\sqcup L_2)=\Upsilon^*_S(L_1\#L_2)\] for every south-west region $S$ of $\mathbb R^2$.
\end{cor}
\begin{proof}
 It follows immediately from Theorem \ref{teo:dim} and Proposition \ref{prop:disjoint}.
\end{proof}
 
\subsection{Slice genus}
\label{subsection:slice}
Suppose that a link $L$ has $n$ components and bounds a smooth, compact, oriented surface $\Sigma\hookrightarrow D^4$ with genus $g(\Sigma)$ and $k$ connected components.
Then, after removing $k$ open disks from it, we can see $\Sigma$ as a smooth cobordism between the $k$-component unlink $\bigcirc_k$ and $L$. If we look at the canonical form of link cobordisms described in Subsection \ref{subsection:canonical} then $\Sigma$ is such that, from left to right, there are no merge moves, the torus moves are $g(\Sigma)$ in total and there are exactly $n-k$ split moves. Other than these, the cobordism $\Sigma$ might have pieces representing concordances, which induce local equivalences as shown in Section \ref{section:concordance}. 

The goal of this subsection is to study how much the $\Upsilon$-invariants of $L$ differ from zero ($\Upsilon_S(\bigcirc_n)=0$ for every $S$) when $L$ bounds a surface $\Sigma$ as before. We use grid diagrams like in Section \ref{section:concordance}.

Let us start from the torus move, see Figure \ref{Torus}. We define a map $t$ as the Identity between the grid diagram representing the link before the move and the one obtained by applying Figure \ref{Bandmove} twice. Such a map is a chain map, induces a graded isomorphism in homology and preserves the $j$-filtration by the same argument in Subsection \ref{subsection:birth}: since the links before and after the moves have both $k$ components, the corresponding diagrams have the same $\OO$-markings (both normal and special).
Previously we used a result of Sarkar (\cite{Sarkar}) to show that $b_2$ is $\mathcal A$-filtered of degree zero. Since now we are composing the same map twice, but the first time the number of components is increasing, this is no longer true. In fact, the map $t$ is $\mathcal A$-filtered of degree $1$, see \cite[Subsection 3.4]{Sarkar}.

Now we study the split moves as in the left side of Figure \ref{Split}. We may want to define a map $s$ in a similar way as what we do for $t$: using the same procedure for the map $b_2$, but this is not possible. In fact, the link $L_2$ has one more component than $L_1$, so the number of special $\OO$-markings is different and $s$ would not be a chain map. 
\begin{figure}[t]
 \centering
 \def\svgwidth{13cm}
\begingroup%
  \makeatletter%
  \providecommand\color[2][]{%
    \errmessage{(Inkscape) Color is used for the text in Inkscape, but the package 'color.sty' is not loaded}%
    \renewcommand\color[2][]{}%
  }%
  \providecommand\transparent[1]{%
    \errmessage{(Inkscape) Transparency is used (non-zero) for the text in Inkscape, but the package 'transparent.sty' is not loaded}%
    \renewcommand\transparent[1]{}%
  }%
  \providecommand\rotatebox[2]{#2}%
  \newcommand*\fsize{\dimexpr\f@size pt\relax}%
  \newcommand*\lineheight[1]{\fontsize{\fsize}{#1\fsize}\selectfont}%
  \ifx\svgwidth\undefined%
    \setlength{\unitlength}{3343.93961543bp}%
    \ifx\svgscale\undefined%
      \relax%
    \else%
      \setlength{\unitlength}{\unitlength * \real{\svgscale}}%
    \fi%
  \else%
    \setlength{\unitlength}{\svgwidth}%
  \fi%
  \global\let\svgwidth\undefined%
  \global\let\svgscale\undefined%
  \makeatother%
  \begin{picture}(1,0.3982988)%
    \lineheight{1}%
    \setlength\tabcolsep{0pt}%
    \put(0,0){\includegraphics[width=\unitlength,page=1]{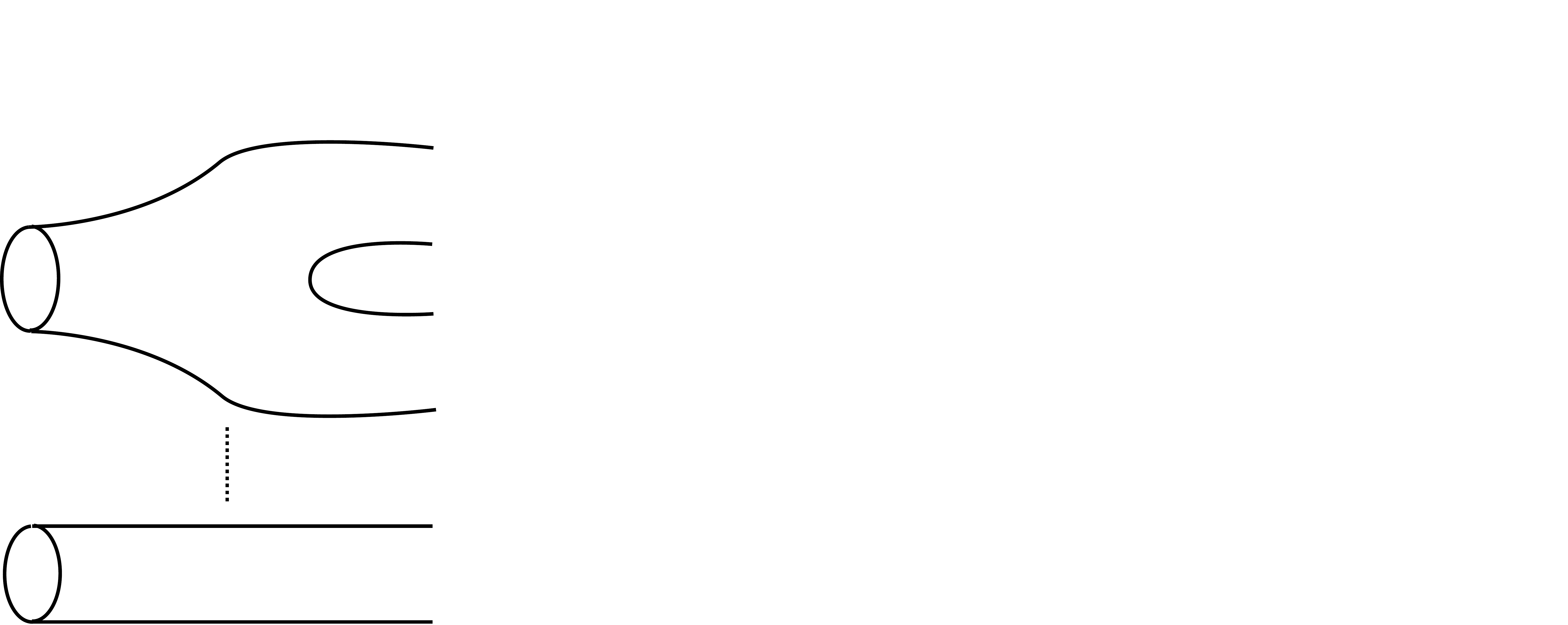}}%
    \put(0.47906995,0.07765609){\color[rgb]{0,0,0}\makebox(0,0)[lt]{\lineheight{0}\smash{\begin{tabular}[t]{l}$L_1$\end{tabular}}}}%
    \put(0.88214455,0.20453808){\color[rgb]{0,0,0}\makebox(0,0)[lt]{\lineheight{0}\smash{\begin{tabular}[t]{l}$L_2$\end{tabular}}}}%
    \put(0,0){\includegraphics[width=\unitlength,page=2]{Split.pdf}}%
    \put(0.52264559,0.26221172){\color[rgb]{0,0,0}\makebox(0,0)[lt]{\lineheight{1.25}\smash{\begin{tabular}[t]{l}$\text{Unknot}$\end{tabular}}}}%
    \put(0,0){\includegraphics[width=\unitlength,page=3]{Split.pdf}}%
    \put(0.69822975,0.38012226){\color[rgb]{0,0,0}\makebox(0,0)[lt]{\lineheight{1.25}\smash{\begin{tabular}[t]{l}$\text{Connected sum}$\end{tabular}}}}%
    \put(0.00107704,0.15572779){\color[rgb]{0,0,0}\makebox(0,0)[lt]{\lineheight{0}\smash{\begin{tabular}[t]{l}$L_1$\end{tabular}}}}%
    \put(0.27278393,0.21724634){\color[rgb]{0,0,0}\makebox(0,0)[lt]{\lineheight{0}\smash{\begin{tabular}[t]{l}$L_2$\end{tabular}}}}%
  \end{picture}%
\endgroup%
   
 \caption{A split move. The two cobordisms in the picture are isotopic in $S^3\times I$ after capping the unknot.}
 \label{Split}
\end{figure}
To avoid this problem, before applying the split move we add a disjoint unknot to $L_1$ and after the split move we connect sum the unknot to the component without special $\OO$-markings. This is pictured on the right side of Figure \ref{Split}. In this way, we can define a map \[s_2:\cC^\infty_0(D_1\sqcup\bigcirc)\longrightarrow\cC^\infty_0(D_2)\:,\] where $D_i$ is a grid diagram for $L_i$, exactly in the same way as $t$. 
Now from Proposition \ref{prop:disjoint} we have that the map \[s_1:=\cC^\infty_0(D'_1)\longrightarrow\cC^\infty_0(D_1'\sqcup\bigcirc)=\cC^\infty_0(D_1')\oplus\cC^\infty_{1}(D_1')\] is the inclusion of $\cC^\infty_0(D_1')$ as the first summand of $\cC^\infty_0(D_1'\sqcup\bigcirc)$; and we recall that $D_i'$ is the grid diagram obtained from $D_i$ by applying the algorithm in Subsection \ref{subsection:overview}. Hence, the map $s_1$ preserves the Maslov grading and the filtration $\mathcal F$. We conclude that the composition $s:=s_2\circ s_1:\cC^\infty_0(D_1)\longrightarrow\cC^\infty_0(D_2)$ induces a graded injective homomorphism in homology, preserves $j$ and it is $\mathcal A$-filtered of degree $1$.

Given a centered south-west region $S$ of $\mathbb R^2$, we say that \[S+m:=\left\{(j,A)\in\mathbb R^2\:|\:(j,A-m)\in S\right\}\] for every $m\in\N$, an example is given in Figure \ref{SWsum}. 
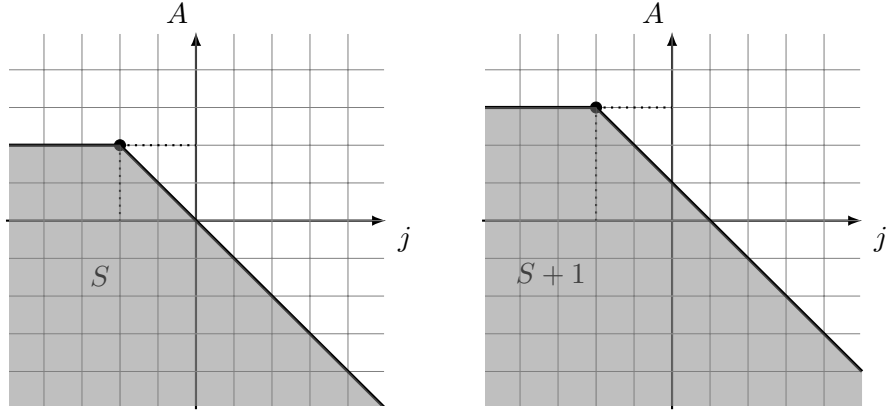
\begin{figure}[t]
 \centering
 \begin{tikzpicture}[scale =.5]
        \coordinate (Origin)   at (0,0);
        \coordinate (XAxisMin) at (-5,0);
        \coordinate (XAxisMax) at (5,0);
        \coordinate (YAxisMin) at (0,-5);
        \coordinate (YAxisMax) at (0,5);
        \draw [thick, black,-latex] (XAxisMin) -- (XAxisMax);
        \draw [thick, black,-latex] (YAxisMin) -- (YAxisMax);
        \clip (-4.9,-4.9) rectangle (5.99,5.99);
        \draw[style=help lines] (-5,-5) grid[step=1cm] (4.95,4.95);
        
        \node[] at (-0.5,5.5) {$A$}; 
        \node[] at (5.5,-0.5) {$j$};
        \node[above left] at (-2,-2) {$S$};
        
        \node[draw,circle,fill,scale=0.4] at (-2,2) {}; 
        \draw [very thick, black] (-5,2) -- (-2,2);
        \draw [very thick, black] (-2,2) -- (5,-5);
        \draw [thick,dotted,black] (-2,2) -- (-2,0);
        \draw [thick,dotted,black] (-2,2) -- (0,2);
        
        \draw[fill, opacity=.5, gray] (-2,2) rectangle (-5,-5);
        \fill [gray, opacity=.5, domain=-2:5, variable=\x]
      (-2,-5)
      -- plot ({\x}, {-\x})
      -- (5,-5)
      -- cycle;
 \end{tikzpicture}
 \hspace{0.5cm}
 \begin{tikzpicture}[scale =.5]
        \coordinate (Origin)   at (0,0);
        \coordinate (XAxisMin) at (-5,0);
        \coordinate (XAxisMax) at (5,0);
        \coordinate (YAxisMin) at (0,-5);
        \coordinate (YAxisMax) at (0,5);
        \draw [thick, black,-latex] (XAxisMin) -- (XAxisMax);
        \draw [thick, black,-latex] (YAxisMin) -- (YAxisMax);
        \clip (-4.9,-4.9) rectangle (5.99,5.99);
        \draw[style=help lines] (-5,-5) grid[step=1cm] (4.95,4.95);
        
        \node[] at (-0.5,5.5) {$A$}; 
        \node[] at (5.5,-0.5) {$j$};
        \node[above left] at (-2,-2) {$S+1$};
        
        \node[draw,circle,fill,scale=0.4] at (-2,3) {}; 
        \draw [very thick, black] (-5,3) -- (-2,3);
        \draw [very thick, black] (-2,3) -- (5,-4);
        \draw [thick,dotted,black] (-2,3) -- (-2,0);
        \draw [thick,dotted,black] (-2,3) -- (0,3);
        
        \draw[fill, opacity=.5, gray] (-2,3) rectangle (-5,-5);
        \fill [gray, opacity=.5, domain=-2:5, variable=\x]
      (-2,-5)
      -- plot ({\x}, {-\x+1})
      -- (5,-5)
      -- cycle;
 \end{tikzpicture}
 \caption{A centered south-west region $S$ on the left and the south-west region $S+1$ on the right.}
 \label{SWsum}
\end{figure}
We define the non-negative integer $h_S(m)$ as\[\min_{k\in\N}\left\{k\:|\:(0,m)\in S_{-k}\right\}\] and we recall that the reversed region $-S$ is defined in Subsection \ref{subsection:symmetries} by applying to $S$ the reflection of $\mathbb R^2$ with respect to the line $\{A-j = 0\}$.
Then we can prove that each $\Upsilon$ gives a lower bound for the genus of $\Sigma$.
\begin{prop}
 \label{prop:lower_bound}
 If $L$ is an $n$-component link in $S^3$, which bounds a surface $\Sigma$ as before, then \[-\Upsilon_S(L)\leq h_{\pm S}(g(\Sigma)+n-k)\] for every centered south-west region $S$ of $\mathbb R^2$.
\end{prop}
\begin{proof}
 We construct a map $f_{\Sigma}$ by composing the maps $t$ and $s$ defined in this subsection, together with the concordance maps in Section \ref{section:concordance}. We obtain that \[f_{\Sigma}(\mathcal F^S\cC^\infty_0(\bigcirc_k))\subset\mathcal F^{S+g(\Sigma)+n-k}\cC^\infty_0(L)\] for every $S$. In particular, if the homology class of $x$ is a generator of $\mathcal F^{\{j\leq0\}}\cH^\infty_0(\bigcirc_k)$ then $f^*_{\Sigma}[x]$ is a generator of $\mathcal F^{\{j\leq0\}}\cH^\infty_0(L)$. This immediately implies that $\Upsilon_S(L)\geq-h_S(g(\Sigma)+n-k)$ and we complete the proof by observing that $\Upsilon_{-S}(L)=\Upsilon_S(L)$ from Theorem \ref{teo:reflection}.
\end{proof}
A similar lower bound holds with $\Upsilon^*$ in place of $\Upsilon$, but it is clear that the proof cannot work as the one of Proposition \ref{prop:lower_bound}. 
In fact, we used the map $s$ that preserves the Maslov grading, while the $\Upsilon^*$-invariants of $L_i$ are computed by finding generators in $\cH^\infty_{1-m}(L_1)$ and $\cH^\infty_{-m}(L_2)$ respectively, where $m$ is the number of components of $L_1$. To jump this hurdle, in the following lemma we introduce another map $s'$ induced by the split move.
\begin{lemma}
 \label{lemma:split_special}
 Suppose that $L_1$ and $L_2$ are as in the left side of Figure \ref{Split} and $D_1$ and $D_2$ are corresponding grid diagrams. Then we can find a chain map  \[s':\cC^\infty_d(D_1)\longrightarrow\cC^\infty_{d-1}(D_2)\] for every $d\in\mathbb Z$, which preserves the $\mathcal F$-filtration and induces an isomorphism \[\dfrac{\mathcal F^{\{j\leq0\}}\cH^\infty_{1-m}(L_1)}{\mathcal F^{\{j\leq-1\}}\cH^\infty_{1-m}(L_1)}\cong_\F\dfrac{\mathcal F^{\{j\leq0\}}\cH^\infty_{-m}(L_2)}{\mathcal F^{\{j\leq-1\}}\cH^\infty_{-m}(L_2)}\] when $m$ is the number of component of $L_1$.
\end{lemma}
\begin{proof}
 We represent the split move using the fragments of $D_1$ and $D_2$ as in Figure \ref{Split_special}, where this time the number of special $\OO$-markings on each component is the same both before and after the move. 
 We define $s'$ as follows:
 \[s'(x)=\left\{\begin{aligned}&x\:\:\:\:\quad\text{ if }\quad c\in x\\ &Ux\quad\text{ otherwise}\end{aligned}\right.\quad\text{ and }\quad s'(V_1p)=U\cdot s'(p)\] for every grid state $x\in S(D_1)$, every $p\in\cC^\infty(D_1)$ and $V_i$-equivariant for $i>1$, where $V_1$ is the variable associated to the normal $\OO$-marking $O_1$, see Figure \ref{Split_special}. 
 
 Consider the diagrams $D_1'$ and $D_2'$, obtained by applying the algorithm in Subsection \ref{subsection:overview} to $D_1$ and $D_2$; hence, the diagrams $D_i$ and $D_i'$ have same $\OO$-markings.
 Denote by $D_3'$ the diagram obtained by removing the row $\alpha$ and the column $\beta$ from $D_2'$; we have that $\cC^\infty(D_2')$ is isomorphic to 
 $\cC^\infty(D_3')\oplus\cC^\infty(D_3')\lkhov 1\rkhov$ because of Proposition \ref{prop:disjoint}. Finally, let us call $\pi:\cC^\infty(D_2')\rightarrow\cC^\infty(D_3')$ the map given by adapting to our setting the homotopy inverse of the map $i\circ H$ used in \cite[Proposition 3.1]{Cavallo}. This map coincide with the special destabilization in \cite[Subsection 3.3]{Sarkar}.
 
 The map $s'$ was also studied by Sarkar, see \cite[Subsection 3.3]{Sarkar}, and he proves that $\pi\circ s'=d^{NO}$ is one of the destabilization maps in \cite{MOST,Book}. Such a map is induced by the link isotopy relating $D_1'$ and $D_3'$ which means it is an almost filtered chain homotopy equivalence, and a local equivalence by Remark \ref{remark:almost}. 
 \begin{figure}[t]      
  \begin{center}
   \begin{tikzpicture}[scale=0.6]
    \draw[help lines] (0,0) grid (2,2);
    \draw[help lines] (-2,0)--(0,0); \draw[help lines] (2,0)--(4,0);
    \draw[help lines] (-2,1)--(0,1); \draw[help lines] (0,1)--(4,1);
    \draw[help lines] (-2,2)--(0,2); \draw[help lines] (0,2)--(4,2);
    \draw[help lines] (0,-2)--(0,0); \draw[help lines] (0,0)--(0,4);
    \draw[help lines] (1,-2)--(1,0); \draw[help lines] (1,0)--(1,4);
    \draw[help lines] (2,-2)--(2,0); \draw[help lines] (2,0)--(2,4);
    \draw (1.5,0.5) circle [radius=0.3]; \draw[red] (0.5,1.5) circle [radius=0.3];
    \draw (-1.5,1.5) node [cross] {}; \draw (3.5,0.5) node [cross] {};
    \draw (1.5,3.5) node [cross] {};  \draw (0.5,-1.5) node [cross] {};
    \node[draw,circle,fill,scale=0.4] at (1,1) {};
    \node[above right] at (0.8,1) {$c$};
    \node[below] at (1.6,0) {$O_1$};
    
    \draw [thick][->] (0.5,-1.25) -- (0.5,1.25); \draw [thick][->] (0.25,1.5) -- (-1.25,1.5);
    \draw [thick][->] (1.5,3.25) -- (1.5,0.75); \draw [thick][->] (1.75,0.5) -- (3.25,0.5);  
    \draw [thin][->] (5,1) -- (6,1);
    \draw[help lines] (7,0)--(9,0);     \draw[help lines] (9,0)--(13,0);
    \draw[help lines] (7,1)--(9,1);     \draw[help lines] (9,1)--(13,1);
    \draw[help lines] (7,2)--(9,2);     \draw[help lines] (9,2)--(13,2);
    \draw[help lines] (9,-2)--(9,0);    \draw[help lines] (9,0)--(9,4);
    \draw[help lines] (10,-2)--(10,0);  \draw[help lines] (10,0)--(10,4);
    \draw[help lines] (11,-2)--(11,0);  \draw[help lines] (11,0)--(11,4);
    \draw[red] (9.5,0.5) circle [radius=0.3]; \draw[red] (10.5,1.5) circle [radius=0.3];
    \draw (7.5,1.5) node [cross] {}; \draw (12.5,0.5) node [cross] {};
    \draw (10.5,3.5) node [cross] {};  \draw (9.5,-1.5) node [cross] {}; \node[right] at (13,1.5) {$\alpha$}; \node[below] at (10.5,-2) {$\beta$};
    
    \draw [thick][->] (9.5,-1.25) -- (9.5,0.25); \draw [thick][->] (10.25,1.5) -- (7.75,1.5);
    \draw [thick][->] (10.5,3.25) -- (10.5,1.75); \draw [thick][->] (9.75,0.5) -- (12.25,0.5);
   \end{tikzpicture}
  \end{center}
  \caption{Another split move in a grid diagram. We recall that special $\OO$-markings are colored in red.}
  \label{Split_special}
\end{figure}
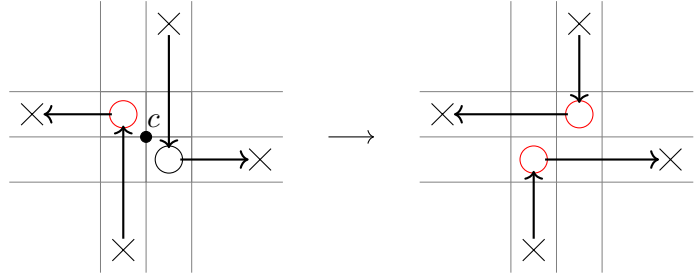
 Therefore, the map $s':\cC^\infty(D_1)\longrightarrow\cC^\infty(D_2)$ induces an injective homomorphism in homology and drops the Maslov grading by one, while the fact that $s'$ preserves the Alexander filtration $\mathcal A$ is shown in \cite[Subsection 3.4]{Sarkar}. 
 
 In order to conclude the proof we just need to observe that $s'$ does not change the minimal $j$-level of a generator of the homology in $\mathcal F^{\{j\leq0\}}\cC^\infty(D_1)$; and note that this only depends on the $\OO$-markings. Since the identification in Proposition \ref{prop:disjoint} is an isomorphism of chain complexes, we would have that if $s'$ would  drop the minimal $j$-level then the same should be true for $\pi\circ s'=d^{NO}$, but this is impossible because the latter is a local equivalence.
\end{proof}
Now we can prove the main result of this subsection.
\begin{prop}
 \label{prop:slice_genus}
 Suppose that $L$ is an $n$-component link in $S^3$ which bounds a smooth, compact, oriented surface $\Sigma\hookrightarrow D^4$, with $k$ connected components. Then we have that \[\begin{aligned}-h_{\pm\overline{\iota S}}(g(\Sigma))\leq&-\Upsilon_S(L)\leq h_{\pm S}(g(\Sigma)+n-k)\quad\text{ and}\\ -h_{\pm\overline{\iota S}}(g(\Sigma)+n-k)\leq&-\Upsilon^*_S(L)\leq h_{\pm S}(g(\Sigma))\end{aligned}\] for every centered south-west region $S$ of $\R^2$.
\end{prop}
\begin{proof}
 The fact that \[-\Upsilon^*_S(L)\leq h_{\pm S}(g(\Sigma))\] follows in the same way as in
 Proposition \ref{prop:lower_bound} by using Lemma \ref{lemma:split_special}. Then we apply Propositions \ref{prop:mirror} and \ref{prop:lower_bound}. 
\end{proof}
This theorem immediately gives the lower bound in Theorem \ref{teo:slice_genus} for the \emph{smooth slice genus} $g_4(L)$ of a link, which is defined as the minimum genus of a smooth, oriented, compact surface properly embedded in $D^4$ and that bounds $L$.
For knots such lower bounds agree with the ones of Alfieri in \cite{Antonio} and Ozsv\'ath, Stipsicz and Szab\'o in \cite{OSSz}. 
\begin{ex}
 We observe that, when $L$ bounds a planar (genus zero) surface in $D^4$, we have $\Upsilon_S(L)\leq0$ and $\Upsilon_S^*(L)\geq0$ for every $S$ centered.
\end{ex}

\subsection{Other concordance invariants from the link Floer complex}
\label{subsection:nu}
\subsubsection{The invariant \texorpdfstring{$\nu^+$}{nu+}}
Let us consider the centered south-west regions \[V_k:=\left\{(j,A)\in\mathbb R^2\:|\:j\leq0,A\leq k\right\}\] with $k\in\N$, see Figure \ref{VW}. We denote the $\Upsilon$-invariants associated to these regions with $-2\cdot V_L(k)=\Upsilon_{V_k}(L)$. 
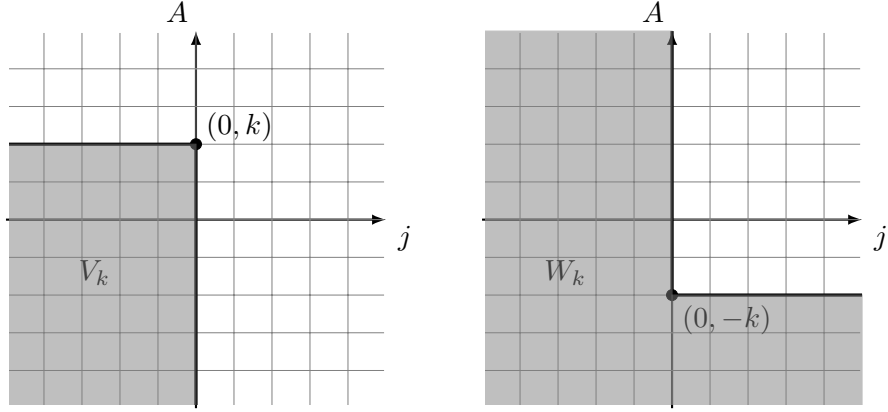
\begin{figure}[t]
 \centering
 \begin{tikzpicture}[scale =.5]
        \coordinate (Origin)   at (0,0);
        \coordinate (XAxisMin) at (-5,0);
        \coordinate (XAxisMax) at (5,0);
        \coordinate (YAxisMin) at (0,-5);
        \coordinate (YAxisMax) at (0,5);
        \draw [thick, black,-latex] (XAxisMin) -- (XAxisMax);
        \draw [thick, black,-latex] (YAxisMin) -- (YAxisMax);
        \clip (-4.9,-4.9) rectangle (5.99,5.99);
        \draw[style=help lines] (-5,-5) grid[step=1cm] (4.95,4.95);
        
        \node[] at (-0.5,5.5) {$A$};
        \node[] at (5.5,-0.5) {$j$};
        \node[above right] at (0,1.8) {$(0,k)$};
        \node[draw,circle,fill,scale=0.4] at (0,2) {};
        
        \node[above left] at (-2,-2) {$V_k$};
        \draw [very thick, black] (-5,2) -- (0,2);
        \draw [very thick, black] (0,2) -- (0,-5);
        
        \draw[fill, opacity=.5, gray] (0,2) rectangle (-5,-5);
 \end{tikzpicture}
 \hspace{0.5cm}
 \begin{tikzpicture}[scale =.5]
        \coordinate (Origin)   at (0,0);
        \coordinate (XAxisMin) at (-5,0);
        \coordinate (XAxisMax) at (5,0);
        \coordinate (YAxisMin) at (0,-5);
        \coordinate (YAxisMax) at (0,5);
        \draw [thick, black,-latex] (XAxisMin) -- (XAxisMax);
        \draw [thick, black,-latex] (YAxisMin) -- (YAxisMax);
        \clip (-4.9,-4.9) rectangle (5.99,5.99);
        \draw[style=help lines] (-5,-5) grid[step=1cm] (4.95,4.95);
        
        \node[] at (-0.5,5.5) {$A$};
        \node[] at (5.5,-0.5) {$j$}; \node[below right] at (0,-2) {$(0,-k)$};
        \node[draw,circle,fill,scale=0.4] at (0,-2) {};
        
        \node[above left] at (-2,-2) {$W_k$};
        \draw [very thick,black] (0,5) -- (0,-2);
        \draw [very thick,black] (0,-2) -- (5,-2);
        
        \draw[fill, opacity=.5, gray] (0,5) rectangle (-5,-5);
        \draw[fill, opacity=.5, gray] (0,-5) rectangle (5,-2);
 \end{tikzpicture}
 \caption{The centered south-west regions $V_k$ (left) and $W_k$ (right)  of $\mathbb R^2$ for any integer $k\geq0$.}
 \label{VW}
\end{figure}
It follows from \cite{Antonio} that the invariants $V_K(k)$ determine some of the invariants $h_k$ of $K$, which were introduced by Rasmussen in \cite{Rasmussen}.
\begin{prop}[Alfieri]
 Suppose that $K$ is a knot in $S^3$. Then $V_K(k)=h_k(K)$ for every $k\in\N$.
\end{prop}
Applying Proposition \ref{prop:slice_genus} we obtain that $h_k(K)=V_K(k)\leq g_4(K)-k$ for a knot $K$ and $0\leq k\leq g_4(K)$, which coincides with \cite[Corollary 7.4]{Rasmussen}; furthermore, one has \[\begin{aligned}0\leq V_L(k)\leq g_4(L)+n-k-1\quad&\text{ if }\quad k<g_4(L)+n-1\\ V_L(k)=0\hspace{2cm}\quad&\text{ if }\quad k\geq g_4(L)+n-1\end{aligned}\] and \[V_L(k)\geq V_L(k+1)\] for every link $L$. Finally, Theorem \ref{teo:upsilon_concordance} tells us that $V_L(k)$ is a concordance invariant for every $k\in\N$.

In \cite{HomWu} Hom and Wu define the knot concordance invariant $\nu^+$ and they prove that such invariant gives a lower bound for the slice genus $g_4$. Using our results we can easily extend $\nu^+$ to links: we say that \[\nu^+(L):=\min_{k\in\N}\left\{k\:|\:V_L(k)=0\right\}\:.\] It is easy to check (\cite{Hom}) that for knots such a definition coincides with the one in \cite{HomWu} and it generalizes its well-known properties.
\begin{prop}
 \label{prop:nu}
 The non-negative integer $\nu^+(L)$ is a concordance invariant of links.
\end{prop}
\begin{proof}
 If $L_1$ is concordant to $L_2$ then $V_{L_1}(k)=V_{L_2}(k)$ for every $k\in\N$ as we saw before. Hence, one has $V_{L_1}(k)=0$ if and only if $V_{L_2}(k)=0$.
\end{proof}
Consider the south-west regions $W_k$ in Figure \ref{VW}; we see immediately that one has $W_k=\overline{\iota V_k}$ and then $\Upsilon^*_{W_k}(L)=2\cdot V_{L^*}(k)$ for every $k\in\N$ because of Proposition \ref{prop:mirror}. We say that  \[\widehat\nu(L)=\max\{\nu^+(L),\nu^+(L^*)\}\:,\] where \[\nu^+(L^*)=\min_{k\in\N}\left\{k\:|\:V_{L^*}(k)=0\right\}=\min_{k\in\N}\left\{k\:|\:\Upsilon^*_{W_k}(L)=0\right\}\] which is also a concordance invariant. 
\begin{teo}
 \label{teo:nu}
 Suppose that $L$ is an $n$-component link in $S^3$. Then we have that \[0\leq\nu^+(L)\leq\widehat\nu(L)\leq g_4(L)+n-1\quad\text{ and }\quad \tau(L)\leq\nu^+(L)\:.\] Furthermore, the invariants $\nu^+(L)$ and $\widehat\nu(L)$ are sub-additive: \[\nu^+(L_1\# L_2)\leq\nu^+(L_1)+\nu^+(L_2)\hspace{1cm}\text{ and }\hspace{1cm}\widehat\nu(L_1\# L_2)\leq\widehat\nu(L_1)+\widehat\nu(L_2)\] for every pair of links $L_1$ and $L_2$.
\end{teo}
\begin{proof}
 We saw before that if $V_L(k)\neq0$ then $k< g_4(L)+n-1$. Since $\nu^+(L)$ is the minimal $k$ such that $V_L(k)=0$ and $g_4(L)=g_4(L^*)$ we conclude that $\widehat\nu(L)\leq g_4(L)+n-1$. We now show that $\tau(L)\leq\nu^+(L)$. Suppose that $s$ is the minimal integer such that $V_L(s)=0$; then there is an element $x$ in $\mathcal F^{V_s}\cC^\infty_0(L)$ whose homology class is a generator of the homology with minimal $j$-level zero. The claim follows from Lemma \ref{lemma:generator}.
 
 For the last part of the theorem, take elements $x_1$ and $x_2$ as before for $L_1$ and $L_2$ respectively. From Subsection \ref{subsection:connected} we know that \[x_1\otimes x_2\in\mathcal F^{V_{\nu^+(L_1)+\nu^+(L_2)}}\cC^\infty_0(L_1\#L_2)\] has the same properties. Since $\Upsilon_{V_k}(L_1\#L_2)\leq0$ for every $k$, this implies $V_{L_1\#L_2}(\nu^+(L_1)+\nu^+(L_2))=0$. Now, denote with $J_1$ and $J_2$ either the links $L_1$ and $L_2$ or the links $L_1^*$ and $L_2^*$, depending on which ones give the maximal $\nu^+(J_1\#J_2)$; then we have \[\widehat\nu(L_1\#L_2)=\nu^+(J_1\#J_2)\leq\nu^+(J_1)+\nu^+(J_2)\leq\widehat\nu(L_1)+\widehat\nu(L_2)\:.\]
\end{proof}
Theorem \ref{teo:nu} tells us that $\nu^+$ gives a lower bound to the slice genus at least as good as the one given by $\tau$. 
\begin{figure}[t]
 \centering
 \begin{tikzpicture}[scale =.75]
        \coordinate (Origin)   at (0,0);
        \coordinate (XAxisMin) at (-3,0);
        \coordinate (XAxisMax) at (5,0);
        \coordinate (YAxisMin) at (0,-3);
        \coordinate (YAxisMax) at (0,5);
        \draw [thick, black,-latex] (XAxisMin) -- (XAxisMax);
        \draw [thick, black,-latex] (YAxisMin) -- (YAxisMax);
        \draw [thick, black,-latex] (-3,1) -- (5,1);
        \draw [thick, black,-latex] (1,-3) -- (1,5);
        \clip (-2.9,-2.9) rectangle (5.99,5.99);
        \draw[style=help lines] (-3,-3) grid[step=1cm] (4.95,4.95);
        
        \foreach \x in {-3,...,4}
                    \node[draw=none,fill=none] at (\x+0.5,-2.5) {\x};
    	\foreach \y in {-3,...,4}     		
        			\node[draw=none,fill=none] at (-2.5,\y+0.5) {\y};
        
        \node[] at (-0.5,5.5) {$A$};
        \node[] at (5.5,-0.5) {$j$};
        
        \node[draw,circle,fill,scale=0.4]  (A1) at (0.5,3.5) {};
        \node[draw,circle,fill,scale=0.4]  (A2) at (1.5,2.5) {};
        \node[draw,circle,fill,scale=0.4]  (A3) at (2.5,1.5) {};
        \node[draw,circle,fill,scale=0.4]  (A4) at (3.5,0.5) {};
        \node[draw,circle,fill=white,scale=0.4]  (B1) at (-0.5,3.5) {};
        \node[draw,circle,fill=white,scale=0.4]  (B2) at (0.5,2.5) {};
        \node[draw,circle,fill=white,scale=0.4]  (B3) at (1.5,1.5) {};
        \node[draw,circle,fill=white,scale=0.4]  (B4) at (2.5,0.5) {};
        \node[draw,circle,fill=white,scale=0.4]  (B5) at (3.5,-0.5) {};
        \node[draw,circle,fill=white,scale=0.4]  (D1) at (0.5-1/6,0.5+1/6) {};
        \node[draw,circle,fill=white,scale=0.4]  (D2) at (0.5+1/6,0.5-1/6) {};
        \node[draw,circle,fill=gray,scale=0.4]  (C1) at (-0.5,0.5+1/6) {};
        \node[draw,circle,fill=gray,scale=0.4]  (C2) at (0.5+1/6,-0.5) {};
        
        \path [thick,-latex] (A1) edge node[left] {} (B1);
        \path [thick,-latex] (A1) edge node[left] {} (B2);
        \path [thick,-latex] (A2) edge node[left] {} (B2);
        \path [thick,-latex] (A2) edge node[left] {} (B3);
        \path [thick,-latex] (A2) edge node[left] {} (D1);
        \path [thick,-latex] (A3) edge node[left] {} (B3);
        \path [thick,-latex] (A3) edge node[left] {} (B4);
        \path [thick,-latex] (A3) edge node[left] {} (D2);
        \path [thick,-latex] (A4) edge node[left] {} (B4);
        \path [thick,-latex] (A4) edge node[left] {} (B5);
        \path [thick,-latex] (B1) edge node[left] {} (C1);
        \path [thick,-latex] (B2) edge node[left] {} (C1);
        \path [thick,-latex] (B4) edge node[left] {} (C2);
        \path [thick,-latex] (B5) edge node[left] {} (C2);
        \path [thick,-latex] (D1) edge node[left] {} (C1);
        \path [thick,-latex] (D2) edge node[left] {} (C2);
 \end{tikzpicture}
 \caption{The relevant non-acyclic summand of $\cC^\infty(L)$, where $L=T_{2,9}\#T_{2,3;2,5}^*\#H_-^{\#n-1}$. Here, by relevant we mean the summand which contains the generators of $\mathcal F^{\{j\leq0\}}\cH^\infty_0(L)$, in the decomposition induced by the connected sum, according to Theorem \ref{teo:connected_sum}. We have that $\tau(L)=0$ and $\nu^+(L)=2$.}
 \label{Nu}
\end{figure}
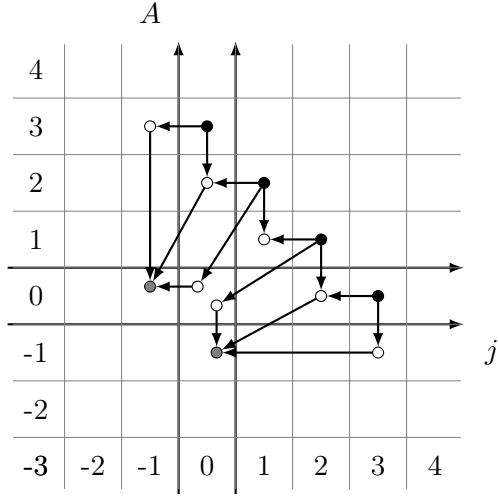
An example where this happens is shown in Figure \ref{Nu}.

Let us write $2\cdot W_L(k)=\Upsilon_{W_k}(L)$. 
In the same way as before, we have \[\begin{aligned}0\leq W_L(k)\leq g_4(L)-k\quad&\text{ if }\quad k<g_4(L)\\ W_L(k)=0\hspace{2cm}\quad&\text{ if }\quad k\geq g_4(L)\end{aligned}\] and \[W_L(k)\geq W_L(k+1)\] for every link $L$.

We call $\widecheck\nu(L)$ the non-negative integer \[\max\left\{\min_{k\in\N}\left\{k\:|\:W_{L}(k)=0\right\},\min_{k\in\N}\left\{k\:|\:W_{L^*}(k)=0\right\}\right\}\:,\] which shares similar properties with $\widehat\nu(L)$. In particular, we have $\widecheck\nu(L)=\widehat\nu(L)$ for knots.
\begin{teo}
 Suppose that $L$ is an $n$-component link in $S^3$. Then we have that $\widecheck\nu$ is a concordance invariant; moreover, one has \[0\leq\widecheck\nu(L)\leq g_4(L)\quad\text{ and }\quad \tau^*(L)\leq\widecheck\nu(L)\:.\] Furthermore, the invariant $\widecheck\nu(L)$ is also sub-additive: \[\widecheck\nu(L_1\# L_2)\leq\widecheck\nu(L_1)+\widecheck\nu(L_2)\] for every pair of links $L_1$ and $L_2$.
\end{teo}
\begin{proof}
 It follows in the same way as in the proof of Proposition \ref{prop:nu} and Theorem \ref{teo:nu}, by applying Lemma \ref{lemma:generator} and Proposition \ref{prop:slice_genus}. We just need to observe that \[\min_{k\in\N}\left\{k\:|\:W_{L^*}(k)=0\right\}=\min_{k\in\N}\left\{k\:|\:\Upsilon^*_{V_k}(L)=0\right\}\:.\]
\end{proof}
This result implies that if $L$ bounds a compact planar surface properly embedded in $D^4$ then $\widecheck\nu(L)=0$.

\subsubsection{Secondary upsilon invariants}
In a paper of Allen (\cite{Allen}) we find an example of two non-concordant knots with the same $\Upsilon$-invariants. These knots are the torus knot $T_{5,7}$ and the connected sum $T_{2,5}\#T_{5,6}$. Their chain complexes are pictured in \cite[Figures 4 and 6]{Allen}. 

Starting from this example, we build the links $J_1=T_{5,7}\#H_+^{\#n-1}$ and $J_2=T_{2,5}\#T_{5,6}\#H_+^{\#n-1}$.  Since we can compute the complex of the positive Hopf link:  \[\cC^\infty(H_+)= CFK^\infty(T_{2,3})\oplus\F[U,U^{-1}]_{(-1)}\:,\] we easily obtain that the homology groups of $J_1$ and $J_2$ are $\mathcal F$-filtered isomorphic.   
On the other hand, it is still possible to show that $\cC^\infty(J_1)$ is not locally equivalent to $\cC^\infty(J_2)$, which means that the filtered isomorphism (or its inverse) is not induced by a chain map that preserves the filtration $\mathcal F$.

In order to find an obstruction for the existence of such a map, we need to use another family of invariants, which was introduced by Kim and Livingston in \cite{Livingston} and by Alfieri in \cite{Antonio} for knots.
We define the \emph{secondary $\Upsilon$-invariants} $\Upsilon^{(2)}_{S^+,S^-,S}(L)$ of an $n$-component link $L$ as $-\Upsilon_S(L)$ plus the supremum of $k\in\mathbb Z$ such that \[\mathcal F^{S_k\cup\left(S^+_{\gamma^+}\right)\cup\left(S^-_{\gamma^-}\right)}\cC^\infty_1(L)\] contains a $1$-chain $a$ with $\partial^-a=x_1+x_2$; the cycles \[x_1\in\mathcal F^{\left(S^+_{\gamma^+}\right)}\cC^\infty_0(L)\quad\text{ and }\quad x_2\in\mathcal F^{\left(S^-_{\gamma^-}\right)}\cC^\infty_0(L)\] have the property that their homology classes are generators of $\mathcal F^{\{j\leq0\}}\cH^\infty_0(L)$, where $\gamma^\pm=\Upsilon_{S^\pm}(L)$ and $S^+,S^-$ and $S$ are three centered south-west regions of $\mathbb R^2$. Note that $\Upsilon^{(2)}_{S^+,S^-,S}(L)$ can be $+\infty$, as it happens for the unknot.

We can define a \emph{secondary $\Upsilon^*$-invariant} exactly in the same way, only this time we consider elements in Maslov gradings $1-n$ and $2-n$. For the sake of simplicity, in this subsection we only write proofs for $\Upsilon^{(2)}_{S^+,S^-,S}(L)$, but all the results also hold for this version of the invariant. 
\begin{prop}
 \label{prop:secondary}
 Let us consider a link $L$. Then the invariant $\Upsilon^{(2)}_{S^+,S^-,S}(L)$ is a concordance invariant for every triple of centered south-west regions $S^+,S^-$ and $S$ of $\mathbb R^2$.
\end{prop}
\begin{proof}
 Suppose that $L_1$ is concordant to $L_2$ and $\Upsilon^{(2)}_{S^+,S^-,S}(L_1)<\Upsilon^{(2)}_{S^+,S^-,S}(L_2)$. Then there is an integer $k>\Upsilon^{(2)}_{S^+,S^-,S}(L_1)+\Upsilon_S(L_1)$ such that \[z^\pm\in\mathcal F^{S^\pm_{\gamma^\pm}}\cC^\infty_0(L_2)\:,\] the homology class $[z^+]=[z^-]$ is the generator of $\mathcal F^{\{j\leq0\}}\cH^\infty_0(L_2)$ and there exists \[\beta\in\mathcal F^{S_k\cup\left(S^+_{\gamma^+}\right)\cup\left(S^-_{\gamma^-}\right)}\cC^\infty_1(L_2)\] with $\partial^-\beta=z^++z^-$. We recall that $\gamma^\pm=\Upsilon_{S^\pm}(L_1)=\Upsilon_{S^\pm}(L_2)$, since $\Upsilon$ is a concordance invariant.
 
 From Theorem \ref{teo:concordance} we know that the corresponding chain complexes of two links are locally equivalent. Then we find a chain map $g:\cC^\infty_0(L_2)\rightarrow\cC^\infty_0(L_1)$, which preserves $\mathcal F$ and induces an $\mathcal F$-filtered isomorphism between $\cH^\infty_0(L_2)$ and $\cH^\infty_0(L_1)$. Therefore, we can take \[g(z^+)+g(z^-)=g(\partial^-\beta)=\partial^-g(\beta)\] and we have \[g(z^\pm)\in\mathcal F^{S^\pm_{\gamma^\pm}}\cC^\infty_0(L_1)\:,\] the homology class $[g(z^+)]=[g(z^-)]$ is the generator of $\mathcal F^{\{j\leq0\}}\cH^\infty_0(L_1)$ and \[g(\beta)\in\mathcal F^{S_k\cup\left(S^+_{\gamma^+}\right)\cup\left(S^-_{\gamma^-}\right)}\cC^\infty_1(L_1)\:.\] This is a contradiction, because it implies $k\leq\Upsilon^{(2)}_{S^+,S^-,S}(L_1)+\Upsilon_S(L_1)$.
\end{proof}
Now we can show that the links $J_1$ and $J_2$ are not concordant.  
We have that \[\cC^\infty_{1-n}(J_1)=\cC^\infty_0(T_{5,7})\quad\text{ and }\quad\cC^\infty_{1-n}(J_2)=\cC^\infty_0(T_{2,5}\#T_{5,6})\] up to acyclics; hence, if $J_1$ and $J_2$ were concordant then Proposition \ref{prop:secondary} should imply \[\Upsilon^{(2)}_{S^+,S^-,S}(T_{5,7})=(\Upsilon^*)^{(2)}_{S^+,S^-,S}(J_1)=(\Upsilon^*)^{(2)}_{S^+,S^-,S}(J_2)=\Upsilon^{(2)}_{S^+,S^-,S}(T_{2,5}\#T_{5,6})\] for every south-west regions $S^\pm$ and $S$. This is not true, as shown by Allen in \cite{Allen}.

We showed that the secondary $\Upsilon$-invariants can give more information than the $\mathcal F$-filtered isomorphism type of $\cH^\infty(L)$; nonetheless, the following proposition holds. Here we recall that the invariants $V_L(0)$ and $W_L(0)$, corresponding to the south-west regions $V_0$ and $W_0$, are defined before in Subsection \ref{subsection:nu}.
\begin{prop}
 \label{prop:zero}
 If $V_L(0)=W_L(0)=0$ then all of the $\Upsilon$'s of $L$ are zero and all of the $\Upsilon^{(2)}$'s of $L$ are $+\infty$. In the same way, if $V_{L^*}(0)=W_{L^*}(0)=0$ then all of the $\Upsilon^*$'s of $L$ are zero and all of the $(\Upsilon^*)^{(2)}$'s of $L$ are $+\infty$.  
\end{prop}
\begin{proof}
 Suppose that $S$ is a centered south-west region of $\mathbb R^2$. Then we have that $V_0\subset S$ and $0=\Upsilon_{V_0}(L)\leq\Upsilon_S(L)$. In the same way, we have that $S\subset W_0$ and $\Upsilon_S(L)\leq\Upsilon_{W_0}(L)=0$. This implies $\Upsilon_S(L)=0$.
 
 Consider two centered south-west regions $S^\pm$ of $\mathbb R^2$. We have that $V_0\subset S^+\cap S^-$ and then there is a cycle, which represents the generator of the algebraic level zero of $\cH^\infty_0(L)$, in \[\mathcal F^{V_0}\cC^\infty_0(L)\subset\mathcal F^{S^+}\cC^\infty_0(L)\cap\mathcal F^{S^-}\cC^\infty_0(L)\:.\] Since $\Upsilon_S(L)=0$ for every $S$ for what we said before, we obtain $\Upsilon^{(2)}_{S^+,S^-,S}(L)=+\infty$.
 The proof for $\Upsilon^*$ is exactly the same because of Proposition \ref{prop:mirror}.
\end{proof}
In particular, for knots we have the following corollary.
\begin{cor}
 For a knot $K$ if $V_K(0)=V_{K^*}(0)=0$ then $\Upsilon_S(K)=0$ and $\Upsilon^{(2)}_{S^+,S^-,S}(L)=+\infty$ for every $S^\pm$ and $S$ south-west regions of $\mathbb R^2$.
\end{cor}
\begin{proof}
 It follows immediately from Propositions \ref{prop:mirror} and \ref{prop:zero}.
\end{proof}
In fact, it is possible to prove that $V_K(0)=V_{K^*}(0)=0$ forces $CFK^\infty(K)$ to be stably equivalent to $\F[U,U^{-1}]_{(0)}$, the filtered chain homotopy type of the unknot, see \cite{Hom}.

\section{Unoriented Heegaard Floer homology}
\label{section:five}
\subsection{The homology group \texorpdfstring{$HFL'(L)$}{HFL'(L)}}
Let us take a Heegaard diagram $\mathcal D$ for a link $L$ in $S^3$.
The chain complex $CFL'(L)$ is the filtered chain homotopy type of $CFL'(\mathcal D)$, the free $\F[U,U^{-1}]$-module over $\T=\T_{\alpha}\cap\T_{\beta}$ with differential given by
\[\partial'x=\sum_{y\in\T}\sum_{\substack{\phi\in\pi_2(x,y) \\ \mu(\phi)=1}}m(\phi)\cdot U^{n_{\textbf w}(\phi)+n_{\textbf z}(\phi)}y\:,\] where $\phi,n_*(\phi)$ and $m(\phi)$ are as in Subsection \ref{subsection:complex}, and
\[\partial'(U^{\pm 1}p)=U^{\pm 1}\cdot\partial'p\] for any $x\in\T$ and $p\in CFL'(\mathcal D)$.

Foe every $x\in\T$ we define the \emph{$\delta$-grading} as \[\delta(x)=M(x)-A(x)\:.\] It is easy to check that, with this definition, the variable $U^{\pm1}$ drops the $\delta$-grading by $\pm1$. Moreover, we have that there is a map \[\partial'_d:CFL'_d(\mathcal D)\longrightarrow CFL'_{d-1}(\mathcal D)\] for any $d\in\mathbb Z$.

The chain complex $CFL'(L)$ also has the algebraic filtration $j$, defined as in Subsection \ref{subsection:complex}:
\[j^tCFL'(L)=U^{-t}\cdot CFL''(L),\]  where $CFL''(L)$ is the free $\F[U]$-module over $\T$ and $t\in\mathbb Z$. Note that the latter group was the original unoriented chain complex defined by Ozsv\'ath, Stipsicz and Szab\'o in \cite{Unoriented}. It is easy to check that the differential $\partial'$ preserves $j$.  

We define the homology group as usual: \[HFL'(L)=\bigoplus_{d\in\mathbb Z}HFL'_d(L)\] and \[\mathcal F^tHFL'_d(L)=\pi_d(\Ker\partial'_{d,t}):=\pi_d(\Ker\partial'_d\cap\mathcal F^tCFL'(L))\:,\] where $\pi_d:\Ker\partial'_d\rightarrow HFL'(L)$ is the quotient map.
\begin{prop}
 \label{prop:prime}
 For every $n$-component link $L$ one has \[HFL'(L)\cong_{\F[U,U^{-1}]}\F[U,U^{-1}]^{2^{n-1}}\:,\] with $\delta$-homogeneous generators, and \[\dfrac{\mathcal F^0HFL'(L)}{\mathcal F^{-1}HFL'(L)}\cong_{\F}\F^{2^{n-1}}\:.\]
\end{prop}
\begin{proof}
 The first claim follows from \cite{Unoriented}, while the second one from the fact that the $U$-action drops the $\delta$-grading by one: each homology class in $\mathcal F^0HFL'(L)\setminus\mathcal F^{-1}HFL'(L)$ corresponds exactly to an $\F[U,U^{-1}]$-summand of $HFL'(L)$.
\end{proof}
In \cite{Unoriented} is proved that $HFL'(L)$ is an isotopy link invariant. This is also implied by the following theorem. 
\begin{teo}
 \label{teo:prime}
 There exists a chain map \[i:\cC^\infty(L)\oplus\cC^\infty(L)\lkhov-1\rkhov\longrightarrow CFL'(L)\:,\] which is an isomorphism of $\F$-vector spaces that identifies the Maslov grading with the $\delta$-grading. 
\end{teo}
\begin{proof}
 Let us consider all the intersection points $x_1,...,x_l$ whose Maslov grading has the same parity of $d$. We define  \[\begin{aligned}i^0_d:\cC^\infty_d(L)&\longrightarrow CFL'_d(L)\\
 U^{k_1}x_1+...+U^{k_l}x_l&\longrightarrow U^{2k_1-A(x_1)}x_1+...+U^{2k_l-A(x_l)}x_l\end{aligned}\] and \[\begin{aligned}i^1_d:\cC^\infty_d(L)&\longrightarrow CFL'_{d-1}(L)\\
 U^{k_1}x_1+...+U^{k_l}x_l&\longrightarrow U^{1+2k_1-A(x_1)}x_1+...+U^{1+2k_l-A(x_l)}x_l\end{aligned}\:.\]
 These maps are linear by definition; let us prove that they are also injective. We observe that $i_d^\epsilon(U^{k_1}x_1+...+U^{k_l}x_l)\neq0$, where $\epsilon$ is 0 or 1, because the monomials $U^{\epsilon+2k_i-A(x_i)}x_i$ for $i=1,...,l$ are linearly independent in $CFL'_{d-\epsilon}(L)$; hence, the kernel of $i_d^\epsilon$ is trivial.
 
 We now show that $i_d=i_d^0+i^1_{d+1}$ is surjective. Suppose that $q=U^{h_1}x_1+...+U^{h_l}x_l\in CFL'_d(L)$. If $h_i\equiv A(x_i)$ mod 2 then there exists a $k_i$ such that $2k_i-A(x_i)=h_i$; otherwise, if $h_j\equiv A(x_j)+1$ mod 2 then there exists a $k_j$ such that $1+2k_j-A(x_j)=h_j$. Therefore, say $q=q_1+q_2$ and $q_i$ consists of monomials of these two kinds respectively, we find $p_1$ and $p_2$ such that \[i_d(p_1,p_2)=i_d^0(p_1)+i_{d+1}^1(p_2)=q_1+q_2=q\] and the claim follows.
 
 Since $i^0_d$ and $i^1_{d+1}$ are both injective and their images have trivial intersection, and then give a direct sum of $CFL'_d(L)$, we obtain that each $i_d$ is a linear isomorphism between $\cC^\infty(L)\oplus\cC^\infty(L)\lkhov-1\rkhov$ in $\delta$-grading $d$ and $CFL'_d(L)$. 
 
 In order to complete the proof we now have to show that $i$ is a chain map, which means $i\circ(\partial^-,\partial^-)=\partial'\circ i$. Since $i$ is linear we can just check monomials. We have \[\begin{aligned}(i_{d-1}\circ(\partial^-,\partial^-))&(U^kx,0)=i_{d-1}(\partial^-(U^kx),0))=i^0_{d-1}(U^k\partial^-x)=\\
 &=U^{2k}\cdot i^0_{d-1}\left(\sum_{y\in\T}\sum_{\substack{\phi\in\pi_2(x,y) \\ \mu(\phi)=1}}m(\phi)\cdot U^{n_{\textbf w}(\phi)}y\right)=\sum_{y\in\T}\sum_{\substack{\phi\in\pi_2(x,y) \\ \mu(\phi)=1}}m(\phi)\cdot U^{2k+2n_{\textbf w}(\phi)-A(y)}y\end{aligned}\] and
 \[(\partial'\circ i_d)(U^kx,0)=\partial'(i^0_d(U^kx))=\partial'(U^{2k-A(x)}x)=\sum_{y\in\T}\sum_{\substack{\phi\in\pi_2(x,y) \\ \mu(\phi)=1}}m(\phi)\cdot U^{2k-A(x)+n_{\textbf w}(\phi)+n_{\textbf z}(\phi)}y\:.\] To conclude we need to see that $n_{\textbf w}(\phi)-A(y)=n_{\textbf z}(\phi)-A(x)$ and this holds for every $\phi\in\pi_2(x,y)$, see \cite{OSlinks}. The proof for the monomials $(0,U^hy)$ is the same.
\end{proof} 
The graded object associated to $CFL'(L)$ is $\widehat{CFL}\:'(L)$, which is the version of $\widehat{CFL}$ obtained by collapsing the bigrading accordingly. Hence, if $L_1$ and $L_2$ are isotopic links then \[\widehat{HFL}_d\:'(L_1)\cong_{\F}\widehat{HFL}_d\:'(L_2)\] for every $d\in\mathbb Z$. This means that both $HFL'(L)$ and $\widehat{HFL}\:'$ are link invariants.

\subsection{The \texorpdfstring{$\upsilon$}{upsilon}-set and unoriented concordance}
We start this subsection with some properties of $HFL'(L)$.
\begin{lemma}
 \label{lemma:prime}
 For every link $L$ we have that
 \begin{enumerate}
  \item if there is a chain map $F:\cC^\infty(L_1)\rightarrow\cC^\infty(L_2)$ which preserves the $\mathcal F$-filtration then the map $F':CFL'(L_1)\rightarrow CFL'(L_2)$, defined as $i_2\circ\left(F\oplus F\lkhov-1\rkhov\right)\circ i_1^{-1}$, preserves $j$;
  \item  if $\cC^\infty(L_1)$ is locally equivalent to $\cC^\infty(L_2)$ then there is a $j$-filtered and $\delta$-graded isomorphism between $HFL'(L_1)$ and $HFL'(L_2)$.
 \end{enumerate}
\end{lemma}
\begin{proof}
 Let us prove Point 1). 
 We have to show that $F'$ is $j$-filtered of degree zero. We do this by proving that if $U^kx\in\mathcal F^tCFL'(L_1)$ then $F'(U^kx)\in\mathcal F^tCFL'(L_2)$ for every monomial.
 
 We assume that $k\geq-t$. Then one has \[i_1^{-1}(U^kx)=\left\{\begin{aligned}&\left(U^{\frac{k+A(x)}{2}},\:0\right)\hspace{0.5cm}\quad\text{ if }\quad k+A(x)\quad\text{ is even}\\ &\left(0,\:U^{\frac{-1+k+A(x)}{2}}\right)\quad\text{ if }\quad k+A(x)\quad\text{ is odd}\end{aligned}\right.\:.\] Now, when $k+A(x)$ is even, we can write
 \[(F\oplus F\lkhov-1\rkhov)(i_1^{-1}(U^kx))=\left(\sum_{y\in\T}a(x,y)\cdot U^{\frac{k+A(x)}{2}+\Delta(x,y)}y,\:0\right)\:,\] with $a(x,y)\in\F$. This yields \[F'(U^kx)=\sum_{y\in\T}a(x,y)\cdot U^{k+A(x)-A(y)+2\Delta(x,y)}y\] and it is easy to check that we get the same result when $k+A(x)$ is odd. To conclude we need to argue that $A(y)\leq A(x)+2\Delta(x,y)$.
 Since $F$ preserves $\mathcal F$, we have that it is both $j$ and $\mathcal A$-filtered of degree zero. Therefore, it is $\Delta(x,y)\geq0$ and $A(y)\leq A(x)+\Delta(x,y)$ whenever $a(x,y)=1$ and the claim follows.
 
 Finally, to prove Point 2) take the maps $f:\cC^\infty(L_1)\rightarrow\cC^\infty(L_2)$ and $g:\cC^\infty(L_2)\rightarrow\cC^\infty(L_1)$, which both preserve the $\mathcal F$-filtration. Now Theorem \ref{teo:prime} implies that $f'$, defined as $i_2\circ\left(f\oplus f\lkhov-1\rkhov\right)\circ i_1^{-1}$, and $g'$, defined in the same way from $g$, induce $\delta$-graded isomorphisms in homology; moreover, Lemma \ref{lemma:prime} Point 1) also gives that they preserve $j$. Hence, we proved that $HFL'(L_1)$ is $j$-filtered isomorphic to $HFL'(L_2)$.
\end{proof}
The first consequence of this lemma is that the group $HFL'$ is also a concordance invariant.
\begin{cor}
 \label{cor:concordance_prime}
 If the link $L_1$ is concordant to the link $L_2$ then the unoriented link Floer homology group $HFL'(L_1)$ is $j$-filtered isomorphic to $HFL'(L_2)$, which means that \[\mathcal F^tHFL'_d(L_1)\cong_{\F}\mathcal F^tHFL'_d(L_2)\] for every $t,d\in\mathbb Z$.
\end{cor}
\begin{proof}
 From Theorem \ref{teo:concordance} we know that $\cC^\infty(L_1)$ is locally equivalent to $\cC^\infty(L_2)$. Then the claim follows from Lemma \ref{lemma:prime} Point 2).
\end{proof}
From Theorem \ref{teo:dim} we know that for an $n$-component link $L$ one has \[\dfrac{\mathcal F^{\{j\leq0\}}\cH^\infty_d(L)}{\mathcal F^{\{j\leq-1\}}\cH^\infty_d(L)}\cong_{\F}\F^{\binom{n-1}{-d}}\] for $d=0,...,1-n$. Let us denote with $\{h_1,...,h_{2^{n-1}}\}$ a basis for the direct sum of such groups, where the homology classes $h_i$'s are taken so that they satisfy the following property: for each $i$, there are an integer $k$ and a Maslov grading $d\in[0,1-n]$ such that \\ $h_i\in\faktor{\mathcal F^{(A_1)_k}\cH^\infty_d(L)}{\mathcal F^{(A_1)_{k+1}}\cH^\infty_d(L)}$, where $A_1$ is the centered south-west region \[\{(j,A)\in\mathbb R^2\:|\:j+A\leq0\}\] that we used in Subsection \ref{subsection:Alexander} to define $\Upsilon_L(1)$, and, for any fixed $k$ and $d$, the number of $h_i$’s with those $k$ and $d$ is exactly \[\dim_{\F}\dfrac{\mathcal F^{(A_1)_k}\cH^\infty_d(L)}{\mathcal F^{(A_1)_{k+1}}\cH^\infty_d(L)}\:.\] We also take $h_1$ to be the only homology class as above in Maslov grading $0$ and $h_{2^{n-1}}$ the same, but in Maslov grading $1-n$.

We define $u_i(L)$ for $i=1,...,2^{n-1}$ as the maximum $k\in\mathbb R$ such that $\mathcal F^{(A_1)_k}\cH^\infty_d(L)$ contains the homology class $h_i$. Note that the unordered set $\{u_1(L),...,u_{2^{n-1}}(L)\}$ does not depend on the choice of the $h_i$'s, but only on the $\mathcal F$-filtered isomorphism type of $\cH^\infty(L)$. Moreover, we have that $u_1(L)=\Upsilon_L(1)$ and $u_{2^{n-1}}(L)=\Upsilon^*_L(1)$.

Now let $\upsilon(L)=\{\upsilon_1(L),...,\upsilon_{2^{n-1}}(L)\}$ be the set of $\delta$-gradings of a homogeneous $\F$-basis of \[\dfrac{\mathcal F^0HFL'(L)}{\mathcal F^{-1}HFL'(L)}\:.\] Such set exists because of Proposition \ref{prop:prime} and it does not depend on the choice of the basis, but only on the $j$-filtered isomorphism type of $HFL'(L)$. We have the following lemma.
\begin{lemma}
 \label{lemma:pain}
 A homogeneous $\F$-basis as before is obtained by taking the homology classes of elements $\{U^{k_1}q_1,...,U^{k_{2^{n-1}}}q_{2^{n-1}}\}$, where $q_i=i_{d_i}^0(p_i)$ and $p_i$ represents the homology class $h_i$ in Maslov grading $d_i$ for every $i=1,...,2^{n-1}$.
\end{lemma}
\begin{proof}
 Since $i$ is an isomorphism for Theorem \ref{teo:prime}, we have that there is an injective map $\cH^\infty(L)\rightarrow HFL'(L)$ identifying the Maslov grading with the $\delta$-grading. This means that if $p$ is a representative for $h$, with Maslov grading $d$, then $i^0_d(p)$ represents a non-zero homology class in $HFL'(L)$; moreover, representatives of distinct homology classes are sent into representatives of distinct homology classes, because of Theorem \ref{teo:prime}.
 
 The element $q=i^0_d(p)$ is in $\delta$-grading $d$, but the minimal $j$-level of $[q]$ is not necessarily zero; although, since the $\delta$-grading is an absolute $\mathbb Z$-grading and the $U$-action drops it by one, we have that there is an integer $k$ such that $U^k[q]$ has indeed minimal $j$-level equal to $0$. 
 
 The fact that the set of all the $U^kq$'s obtained in this way gives a basis as wanted is assured by the condition we put on the choice of the $h_i$'s.
\end{proof}
We use this lemma to show that the $\upsilon$-set of $L$ is closely related to the set $\{u_1(L),...,u_{2^{n-1}}(L)\}$.
\begin{prop}
 \label{prop:upsilon}
 Let $\upsilon(L)$ and $u_i(L)$ for $i=1,...,2^{n-1}$ be as before. Then we have that $\upsilon_i(L)=u_i(L)+d_i$, where $u_i(L)$ is associated to the homology class $h_i$ with Maslov grading $d_i$. In particular, one has $\upsilon_1(L)=\Upsilon_L(1)$ and $\upsilon_{2^{n-1}}(L)=\Upsilon^*_L(1)+1-n$.
\end{prop}
\begin{proof}
 Suppose that $p_i=U^{k_1}x_1+...+U^{k_\ell}x_\ell\in\cC^\infty_{d_i}(L)$ represents the homology class $h_i$; moreover, we assume that \[k_j-A(U^{k_j}x_j)=2k_j-A(x_j)\geq u_i(L)\hspace{1cm}\text{ for any }\hspace{1cm}j=1,...,\ell\] and $2k_1-A(x_1)=u_i(L)$.
 
 Using Lemma \ref{lemma:pain} we obtain that $q_i=i(p_i)=U^{2k_1-A(x_1)}x_1+...+U^{2k_\ell-A(x_\ell)}x_\ell\in CFL'_{d_i}(L)$ represents a non-zero homology class in $HFL'_{d_i}(L)$ and $U^{-2k_1+A(x_1)}\cdot q_i$ is in minimal algebraic level zero. Moreover, we saw that we get a homogeneous basis by considering all the $h_i$'s and then, by definition of $\upsilon(L)$, we have
 \[\upsilon_i(L)=\delta\left(U^{-2k_1+A(x_1)}\cdot q_i\right)=\delta(q_i)+2k_1-A(x_1)=d_i+u_i(L)\] for every $i=1,...,2^{n-1}$.
\end{proof}
We can shift $HFL'(L)$ in order to turn it into an unoriented link invariant. 
\begin{teo}
 \label{teo:signature}
 The complex $CFL'(L_1)\left\llbracket\frac{\sigma(L_1)}{2}\right\rrbracket$ is $j$-filtered chain homotopy equivalent to \\ $CFL'(L_2)\left\lkhov\frac{\sigma(L_2)}{2}\right\rkhov$ whenever $L_1$ is isotopic to $L_2$ as unoriented links, where $\sigma$ is the signature of a link as in \emph{\cite{GL}}. 
 In particular, the set \[\upsilon(L)-\dfrac{\sigma(L)}{2}=\left\{\Upsilon_L(1)-\frac{\sigma(L)}{2},...,\Upsilon_L^*(1)+1-n-\frac{\sigma(L)}{2}\right\}\] is an unoriented link invariant for every link $L$.
\end{teo}
\begin{proof}
 Changing the orientation of a link $L$ from $\vec L_1$ to $\vec L_2$, by reversing the orientation on the $i$-th component, results in a grid diagram $G$ where the $\OO_i$-markings and the $\X_i$-markings are swapped. Then everything stays the same except for the $\delta$-grading, which is renormalized. Using \cite[Proposition 7.1]{Unoriented} we conclude that \[\delta_1(x)-\delta_2(x)=\dfrac{\sigma(\vec L_1)}{2}-\dfrac{\sigma(\vec L_2)}{2}\] for every grid state $x$ of $G$.
\end{proof}
It is important to note that, if we only compute the group $HFL'(L)$, we do not know how to identify $\Upsilon_L(1)$ and $\Upsilon_L^*(1)+1-n$ in the $\upsilon$-set of $L$. This means that the latter is an unoriented link invariant only if considered as an unordered set of $2^{n-1}$ integers, up to an overall shift that can be determined from a diagram representing $L$.
Furthermore, an analogue of the last result holds for unoriented concordant links.  
\begin{proof}[Proof of Theorem \ref{teo:unoriented_concordance}]
 It follows in the same way as the last theorem, using Corollary \ref{cor:concordance_prime}.
\end{proof}

\subsection{Unoriented cobordisms} 
\label{subsection:unoriented}
\subsubsection{Normal form and Euler number}
Let us denote with $\upsilon_{\max}$ (resp. $\upsilon_{\min}$) the maximal (resp. minimal) value in the $\upsilon$-set of a link. From \cite[Theorem 5.2]{Unoriented} if there is an oriented saddle between $L$ and $L'$, where $L'$ has one more component with respect to $L$, then \begin{equation}\upsilon_{\max}(L')\leq\upsilon_{\max}(L)\leq\upsilon_{\max}(L')+1
\label{equation:oriented_max}\end{equation} and \begin{equation}\upsilon_{\min}(L')\leq\upsilon_{\min}(L)\leq\upsilon_{\min}(L')+1\:.
\label{equation:oriented_min}\end{equation}
The following inequalities agree with Proposition \ref{prop:slice_genus}.
\begin{prop}
 \label{prop:new}
 Suppose that a link $L$ bounds a compact oriented surface $\Sigma$, properly embedded in $D^4$, with genus $g(\Sigma)$ and $k$ connected components. Then we have that
 \[
-g(\Sigma)+k-n\leq\upsilon_{\max}(L)\leq g(\Sigma)\hspace{1cm}\text{ and }\hspace{1cm}
-g(\Sigma)+1-n\leq\upsilon_{\min}(L)\leq g(\Sigma)+1-k\:.\]
\end{prop}
\begin{proof}
 From Corollary \ref{cor:concordance_prime} we know that $\upsilon_{\max}$ and $\upsilon_{\min}$ are concordance invariants. Hence, since every oriented cobordism $\Sigma$ between $\bigcirc_k$ and an $n$-component link $L$ can be decomposed as explained at the beginning of Subsection \ref{subsection:slice}, and the values of $\upsilon_{\max}(\bigcirc_k)$ and $\upsilon_{\min}(\bigcirc_k)$ are $0$ and $1-k$ respectively, the claim follows from Equations \eqref{equation:oriented_max} and \eqref{equation:oriented_min}.
\end{proof}
We now want to study how these invariants behave when we consider unoriented cobordisms. First, we note that there still exists a normal form; in fact, comparing the oriented case with the results of Kamada in \cite{Kamada} applied to cobordisms, we obtain that every unoriented cobordism $F$ between $L_1$ and $L_2$ can be written as in Figure \ref{Unoriented}.
\begin{figure}[t]
 \centering
 \def\svgwidth{15cm}
\begingroup%
  \makeatletter%
  \providecommand\color[2][]{%
    \errmessage{(Inkscape) Color is used for the text in Inkscape, but the package 'color.sty' is not loaded}%
    \renewcommand\color[2][]{}%
  }%
  \providecommand\transparent[1]{%
    \errmessage{(Inkscape) Transparency is used (non-zero) for the text in Inkscape, but the package 'transparent.sty' is not loaded}%
    \renewcommand\transparent[1]{}%
  }%
  \providecommand\rotatebox[2]{#2}%
  \newcommand*\fsize{\dimexpr\f@size pt\relax}%
  \newcommand*\lineheight[1]{\fontsize{\fsize}{#1\fsize}\selectfont}%
  \ifx\svgwidth\undefined%
    \setlength{\unitlength}{3412.88602418bp}%
    \ifx\svgscale\undefined%
      \relax%
    \else%
      \setlength{\unitlength}{\unitlength * \real{\svgscale}}%
    \fi%
  \else%
    \setlength{\unitlength}{\svgwidth}%
  \fi%
  \global\let\svgwidth\undefined%
  \global\let\svgscale\undefined%
  \makeatother%
  \begin{picture}(1,0.29903445)%
    \lineheight{1}%
    \setlength\tabcolsep{0pt}%
    \put(0,0){\includegraphics[width=\unitlength,page=1]{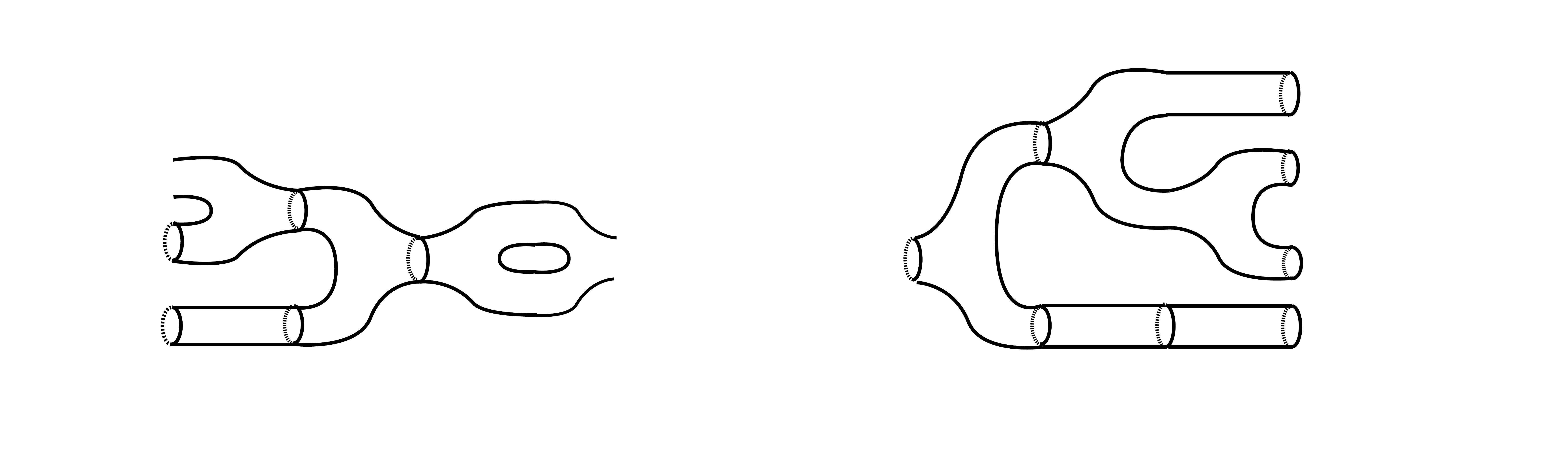}}%
    \put(-0.00243219,0.25529187){\color[rgb]{0,0,0}\makebox(0,0)[lt]{\lineheight{0}\smash{\begin{tabular}[t]{l}$L_1$\end{tabular}}}}%
    \put(0.91551723,0.26408209){\color[rgb]{0,0,0}\makebox(0,0)[lt]{\lineheight{0}\smash{\begin{tabular}[t]{l}$L_2$\end{tabular}}}}%
    \put(0,0){\includegraphics[width=\unitlength,page=2]{Unoriented.pdf}}%
  \end{picture}%
\endgroup%
   
 \caption{Canonical form of unoriented cobordisms between two links: only one connected component of $F$ is shown. The non-orientable saddles are M\"obius strips with a small open disk removed.}
 \label{Unoriented}
\end{figure}
\begin{figure}[t]
 \centering
 \def\svgwidth{4cm}
\begingroup%
  \makeatletter%
  \providecommand\color[2][]{%
    \errmessage{(Inkscape) Color is used for the text in Inkscape, but the package 'color.sty' is not loaded}%
    \renewcommand\color[2][]{}%
  }%
  \providecommand\transparent[1]{%
    \errmessage{(Inkscape) Transparency is used (non-zero) for the text in Inkscape, but the package 'transparent.sty' is not loaded}%
    \renewcommand\transparent[1]{}%
  }%
  \providecommand\rotatebox[2]{#2}%
  \newcommand*\fsize{\dimexpr\f@size pt\relax}%
  \newcommand*\lineheight[1]{\fontsize{\fsize}{#1\fsize}\selectfont}%
  \ifx\svgwidth\undefined%
    \setlength{\unitlength}{1282.01258573bp}%
    \ifx\svgscale\undefined%
      \relax%
    \else%
      \setlength{\unitlength}{\unitlength * \real{\svgscale}}%
    \fi%
  \else%
    \setlength{\unitlength}{\svgwidth}%
  \fi%
  \global\let\svgwidth\undefined%
  \global\let\svgscale\undefined%
  \makeatother%
  \begin{picture}(1,1.1538049)%
    \lineheight{1}%
    \setlength\tabcolsep{0pt}%
    \put(0,0){\includegraphics[width=\unitlength,page=1]{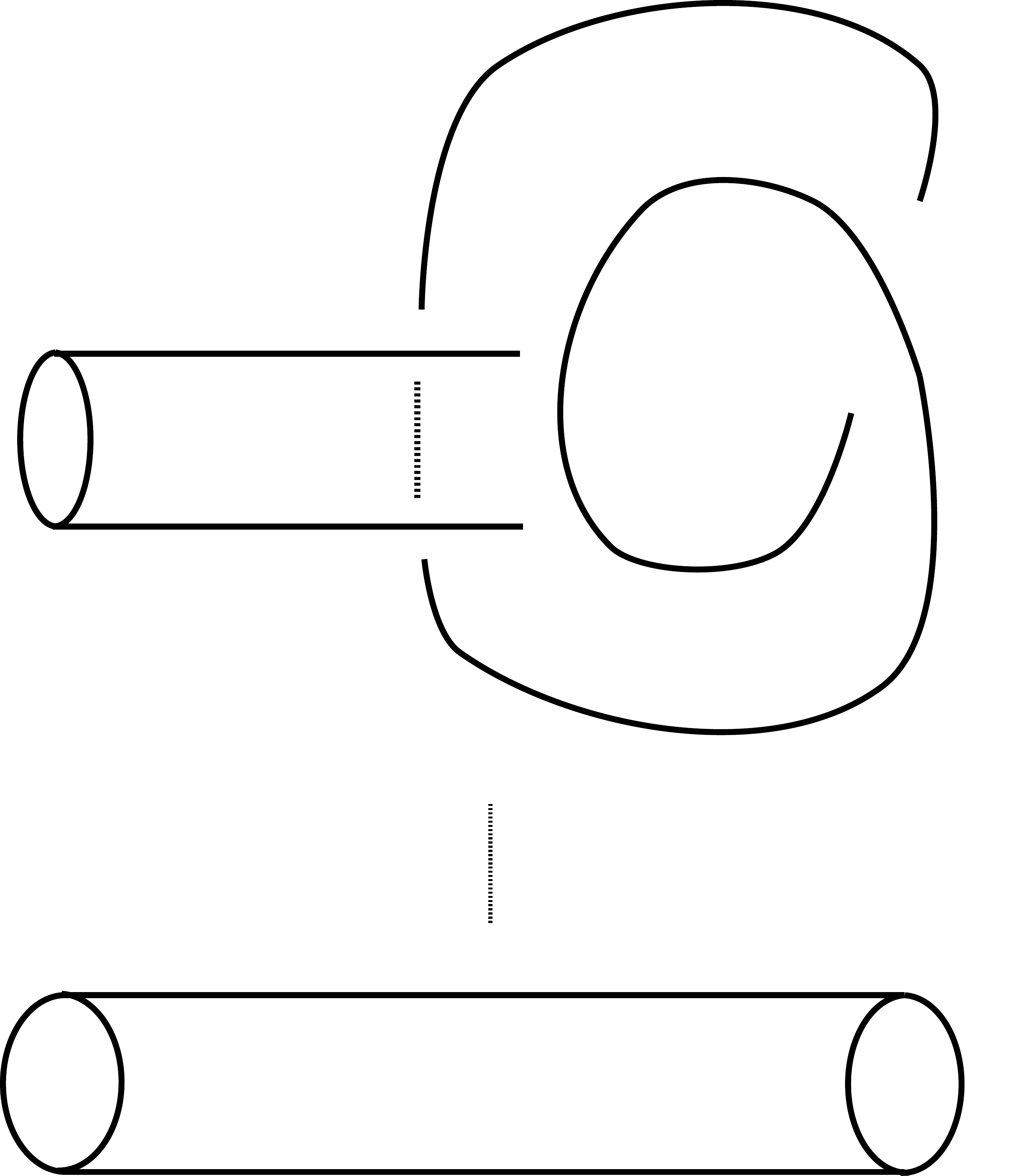}}%
    \put(0.00835806,0.53375622){\color[rgb]{0,0,0}\makebox(0,0)[lt]{\lineheight{1.25}\smash{\begin{tabular}[t]{l}$L_1$\end{tabular}}}}%
    \put(0.82007661,0.22824332){\color[rgb]{0,0,0}\makebox(0,0)[lt]{\lineheight{1.25}\smash{\begin{tabular}[t]{l}$L_2$\end{tabular}}}}%
  \end{picture}%
\endgroup%
   
 \caption{A non-orientable saddle corresponds to a non-oriented band move on a single component.}
 \label{Specific_nonorientable_saddle}
\end{figure}
Hence, we just need to check what happens to the $\upsilon$-set when two links are related by many non-orientable saddles. Of course, we can just study the case where there is only one of such moves, since the general case is obtained by composing the cobordism in Figure \ref{Specific_nonorientable_saddle}.

We recall that, if $F$ is an unoriented cobordism, there is a well-defined integer $e(F)$, called the \emph{Euler number}, defined as \[e(F):=\sum_{p\in F\cap F'}\epsilon_p\] where $\epsilon_p$ is the sign of a oriented basis of $T_pF\oplus T_pF'$, induced by a local orientation system of $F$, compared with the one given by the orientation of $T_p(S^3\times I)$; and where $F'$ denotes a push-off of $F$ along the trivialization of $\nu(L_1)$ (resp. $\nu(L_2)$) in $S^3\times\{0\}$ (resp. $S^3\times\{1\}$) given by the Seifert framing, see \cite{GL,Unoriented}. Clearly, we have that $e(F)=0$ if $F$ is an orientable knot cobordism.

The integer $e(F)$ can also be interpreted in the following way. Suppose that $L_1$ has $n$ components, while $L_2$ has $m$; since $F$ is homotopy equivalent to a 1-dimensional CW-complex, its normal $1$-sphere bundle admits a section $F'$. The boundary of $F'$ consists of the links $L_1'$ and $L_2'$, which can be oriented accordingly to $L_1$ and $L_2$.
Then one has \[e(F)=\sum_{i=1}^n\lk(L_1^i,(L_1^i)')-\sum_{j=1}^m\lk(L_2^j,(L_2^j)')\:.\] The reader can check that this definition is independent of the choice of the section, see \cite{GL}.

From the previous statement we obtain that if $F$ is the union of disjoint surfaces $F_1,...,F_k$ then $e(F)=e(F_1)+...+e(F_k)$. In particular, a non-orientable saddle as in Figure \ref{Specific_nonorientable_saddle} has Euler number equal to that of the unique non-orientable component. 
\begin{lemma}
 \label{lemma:writhe}
 Suppose that $L_1$ and $L_2$ are related by a non-orientable saddle $F$. 
 Say $D_1$ and $D_2$ are planar diagrams for them such that the saddle is represented as in Figure \ref{Local_saddle}. 
 Denote with $D_i'$ the corresponding diagram obtained from $D_i$ by deleting all the components that do not appear in the saddle. 
 
 We have that \[e(F)=\writhe(D_1')-\writhe(D_2')+\epsilon\:,\] where $\epsilon$ is equal to $1$ if the crossing is positive and $-1$ if is negative.
\end{lemma}
\begin{proof}
 From what we said before $e(F)=e(F')$, 
 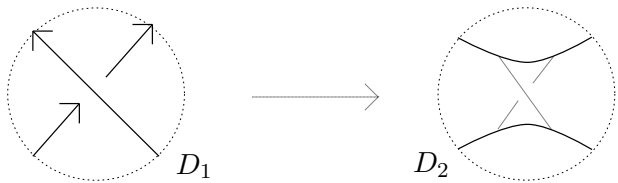
\begin{figure}[t]
 \centering
 \def\svgwidth{8cm}
\begingroup%
  \makeatletter%
  \providecommand\color[2][]{%
    \errmessage{(Inkscape) Color is used for the text in Inkscape, but the package 'color.sty' is not loaded}%
    \renewcommand\color[2][]{}%
  }%
  \providecommand\transparent[1]{%
    \errmessage{(Inkscape) Transparency is used (non-zero) for the text in Inkscape, but the package 'transparent.sty' is not loaded}%
    \renewcommand\transparent[1]{}%
  }%
  \providecommand\rotatebox[2]{#2}%
  \newcommand*\fsize{\dimexpr\f@size pt\relax}%
  \newcommand*\lineheight[1]{\fontsize{\fsize}{#1\fsize}\selectfont}%
  \ifx\svgwidth\undefined%
    \setlength{\unitlength}{5493.10240173bp}%
    \ifx\svgscale\undefined%
      \relax%
    \else%
      \setlength{\unitlength}{\unitlength * \real{\svgscale}}%
    \fi%
  \else%
    \setlength{\unitlength}{\svgwidth}%
  \fi%
  \global\let\svgwidth\undefined%
  \global\let\svgscale\undefined%
  \makeatother%
  \begin{picture}(1,0.2872822)%
    \lineheight{1}%
    \setlength\tabcolsep{0pt}%
    \put(0,0){\includegraphics[width=\unitlength,page=1]{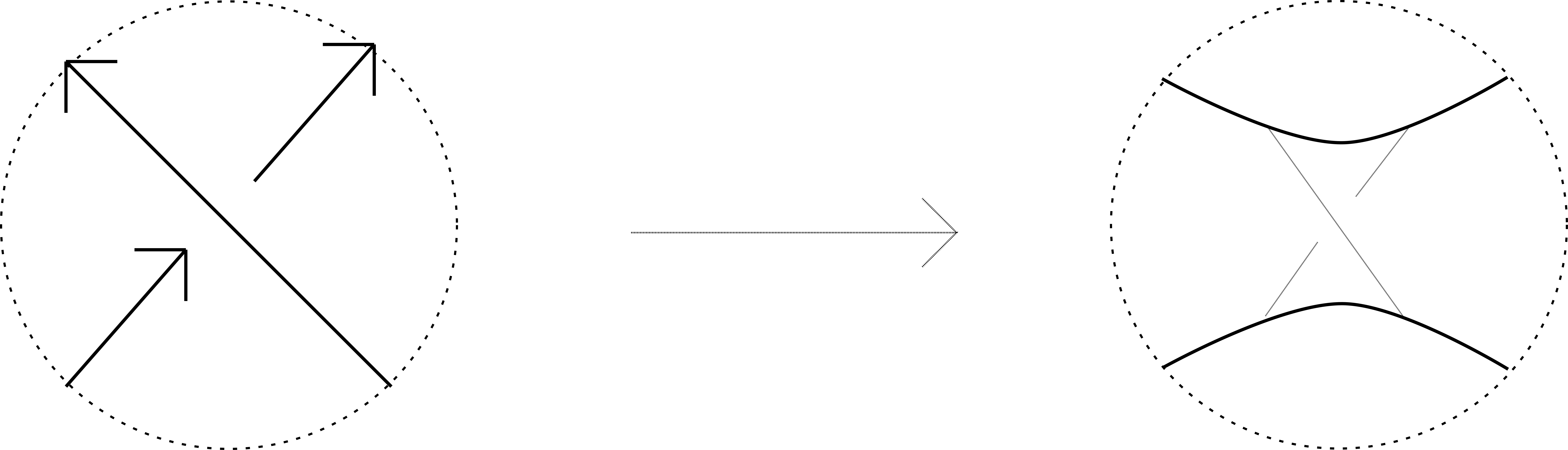}}%
    \put(0.42318835,0.05208053){\color[rgb]{0,0,0}\makebox(0,0)[lt]{\begin{minipage}{0.07157365\unitlength}\raggedright \end{minipage}}}%
    \put(0.26238005,0.01582895){\color[rgb]{0,0,0}\makebox(0,0)[lt]{\begin{minipage}{0.09481184\unitlength}\raggedright \end{minipage}}}%
    \put(-0.02298496,-0.0761943){\color[rgb]{0,0,0}\makebox(0,0)[lt]{\begin{minipage}{0.19334177\unitlength}\raggedright \end{minipage}}}%
    \put(0.27725251,0.01118131){\color[rgb]{0,0,0}\makebox(0,0)[lt]{\lineheight{1.25}\smash{\begin{tabular}[t]{l}$D_1$\end{tabular}}}}%
    \put(0.66740866,0.01372728){\color[rgb]{0,0,0}\makebox(0,0)[lt]{\lineheight{1.25}\smash{\begin{tabular}[t]{l}$D_2$\end{tabular}}}}%
  \end{picture}%
\endgroup%
   
 \caption{The non-orientable saddle is represented in the diagrams as an unoriented resolution of a crossing, where both arcs belong to the same component of $L_1$.}
 \label{Local_saddle}
\end{figure}
 where $F'$ is a non-orientable saddle between $K_1$ and $K_2$, the components of the links represented by $D_1'$ and $D_2'$. 
 \begin{figure}[t]
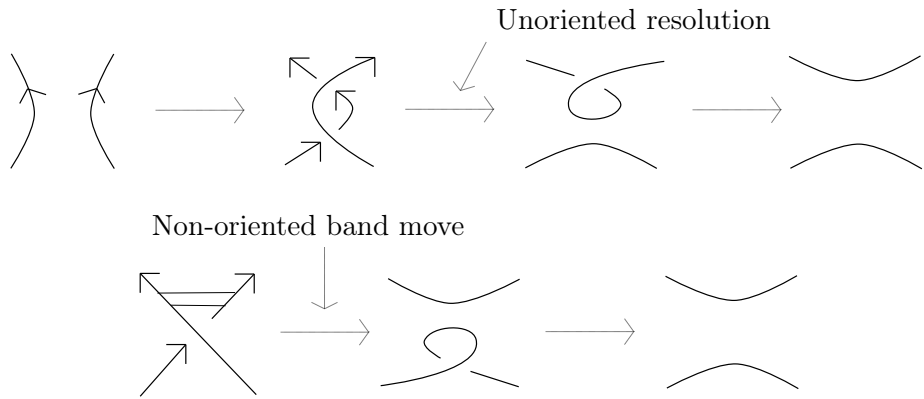

 \centering
 \def\svgwidth{12cm}
\begingroup%
  \makeatletter%
  \providecommand\color[2][]{%
    \errmessage{(Inkscape) Color is used for the text in Inkscape, but the package 'color.sty' is not loaded}%
    \renewcommand\color[2][]{}%
  }%
  \providecommand\transparent[1]{%
    \errmessage{(Inkscape) Transparency is used (non-zero) for the text in Inkscape, but the package 'transparent.sty' is not loaded}%
    \renewcommand\transparent[1]{}%
  }%
  \providecommand\rotatebox[2]{#2}%
  \newcommand*\fsize{\dimexpr\f@size pt\relax}%
  \newcommand*\lineheight[1]{\fontsize{\fsize}{#1\fsize}\selectfont}%
  \ifx\svgwidth\undefined%
    \setlength{\unitlength}{8396.03044148bp}%
    \ifx\svgscale\undefined%
      \relax%
    \else%
      \setlength{\unitlength}{\unitlength * \real{\svgscale}}%
    \fi%
  \else%
    \setlength{\unitlength}{\svgwidth}%
  \fi%
  \global\let\svgwidth\undefined%
  \global\let\svgscale\undefined%
  \makeatother%
  \begin{picture}(1,0.42447843)%
    \lineheight{1}%
    \setlength\tabcolsep{0pt}%
    \put(0,0){\includegraphics[width=\unitlength,page=1]{Unoriented_resolution.pdf}}%
    \put(0.15524584,0.17893457){\color[rgb]{0,0,0}\makebox(0,0)[lt]{\lineheight{1.25}\smash{\begin{tabular}[t]{l}$\text{Non-oriented band move}$\end{tabular}}}}%
    \put(0.53290292,0.40638022){\color[rgb]{0,0,0}\makebox(0,0)[lt]{\lineheight{1.25}\smash{\begin{tabular}[t]{l}$\text{Unoriented resolution}$\end{tabular}}}}%
    \put(0.28295596,0.09379448){\color[rgb]{0,0,0}\makebox(0,0)[lt]{\begin{minipage}{0.14413\unitlength}\raggedright \end{minipage}}}%
    \put(0,0){\includegraphics[width=\unitlength,page=2]{Unoriented_resolution.pdf}}%
  \end{picture}%
\endgroup%
   
 \caption{Each of two rows shows a direction of the equivalence of the two representations of a non-orientable saddle.}
 \label{Unoriented_resolution}
\end{figure}
 Since $e(F')$ is computed from a tubular neighborhood of $F'$ and $F'$ is disjoint from the other annuli of $F$, we have that $e(F')$ can be computed using \cite[Lemma 4.3]{Unoriented}: \[e(F')=\writhe(D_1')-\writhe(D_2')+\epsilon\:.\] The fact that every non-orientable saddle can be seen as an unoriented resolution of a crossing (and vice versa) follows easily from Figure \ref{Unoriented_resolution}. 
\end{proof}

\subsubsection{Unorientable saddle move}
We use the grid diagrams and maps defined in \cite[Section 5]{Unoriented}. 
Say $G_1$ and $G_2$ are grid diagrams for $L_1$ and $L_2$, which are related by a non-orientable saddle as in Figure \ref{Move}. 
\begin{figure}[t]
 \centering
 \def\svgwidth{10cm}
 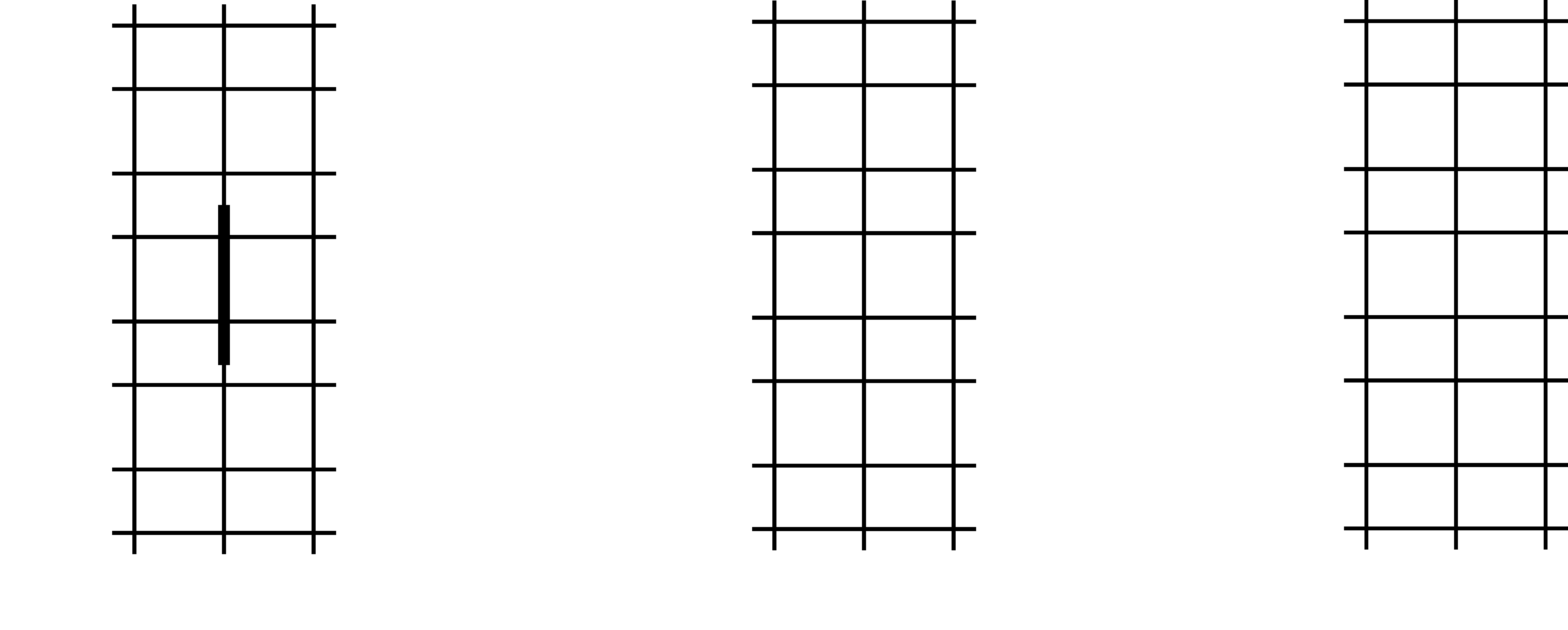   
 \caption{Non-orientable saddle in a grid diagram. Assume the markings in the first two columns of $G_1$ belong to the same component of $L_1$; we switch the $\X$-marking in the first column with the $\OO$-marking in the second one to get $G'$. Then starting from the $\X$ at the bottom, we reverse all the markings on this component of the link until we obtain the diagram $G_2$.}
 \label{Move}
\end{figure}
Then we have chain maps $\nu:CFL'(G_1)\rightarrow CFL'(G_2)$ and $\nu':CFL'(G_2)\rightarrow CFL'(G_1)$, such that $\nu'\circ\nu=\nu\circ\nu'=U$, defined as 
\[\nu(x)=\left\{\begin{aligned}&Ux\quad\text{ if }\hspace{0.4cm}x\cap A\neq\emptyset\\
&x\hspace{0.3cm}\quad\text{ if }\quad x\cap A=\emptyset\end{aligned}\right.\quad\text{ and }\quad\nu'(x)=\left\{\begin{aligned}&x\hspace{0.3cm}\quad\text{ if }\hspace{0.4cm}x\cap A\neq\emptyset\\
&Ux\quad\text{ if }\quad x\cap A=\emptyset\end{aligned}\right.\] for every grid state $x$.
\begin{lemma}
 \label{lemma:ill}
 The maps $\nu$ and $\nu'$ as before drop the $\delta$-grading by \[\dfrac{2-e(F)}{4}-\dfrac{\lk(K_1,L_1\setminus K_1)-\lk(K_2,L_2\setminus K_2)}{2}\quad\text{ and }\quad \dfrac{2+e(F)}{4}+\dfrac{\lk(K_1,L_1\setminus K_1)-\lk(K_2,L_2\setminus K_2)}{2}\] respectively. Here by $K_i$ we denote the component of $L_i$ where we perform the non-orientable saddle move.
\end{lemma}
\begin{proof}
 Say $G_1,G'$ and $G_2$ are as in Figure \ref{Move}, with orientations on $G_1$ and $G_2$ given as in Subsection \ref{subsection:overview}. We prove the claim for the map $\nu$: from \cite[Proposition 5.7]{Unoriented} and its proof we have that $\delta_{G_1}(x)=\delta_{G'}(\nu(x))$ and \[\begin{aligned}\delta_{G'}(\nu(x))-\delta_{G_2}(\nu(x)&)=-\dfrac{1}{4}\left[\writhe(G_1)-\writhe(G_2)+1-2\right]=\\&=-\dfrac{1}{4}\left[\writhe(G_1^1)-\writhe(G^1_2)+1-2\right]-\dfrac{1}{4}\left[2\cdot\lk(K_1,L_1\setminus K_1)-2\cdot\lk(K_2,L_2\setminus K_2)\right]\:,\end{aligned}\] where $G^1_i$ is the subdiagram representing $K_i$.
 
 Then we have that \[\delta_{G_2}(\nu(x))=\delta_{G_1}(x)-\dfrac{2-e(F)}{4}+\dfrac{\lk(K_1,L_1\setminus K_1)-\lk(K_2,L_2\setminus K_2)}{2}\] because of Lemma \ref{lemma:writhe}. The case of $\nu'$ is done in the same way.
\end{proof}
This lemma implies the following result.
\begin{prop}
 \label{prop:unoriented_saddle}
 Suppose that $L_i,K_i$ are as before and $F$ is the corresponding non-orientable saddle. Then the following inequality holds: \[\begin{aligned}\upsilon_{\max}&(L_1)-\dfrac{2-e(F)}{4}+\dfrac{1}{2}\left[\lk(K_1,L_1\setminus K_1)-\lk(K_2,L_2\setminus K_2)\right]\leq\\ &\leq\upsilon_{\max}(L_2)\leq\upsilon_{\max}(L_1)+\dfrac{2+e(F)}{4}+\dfrac{1}{2}\left[\lk(K_1,L_1\setminus K_1)-\lk(K_2,L_2\setminus K_2)\right]\:,\end{aligned}\] where $L_1\setminus K_1$ and $L_2\setminus K_2$ are oriented in the same way. The same is true for $\upsilon_{\min}$. 
\end{prop}
\begin{proof}
 Since $\nu'\circ\nu=\nu\circ\nu'=U$ we have that $\nu$ and $\nu'$ induce isomorphisms in homology. Therefore, the claim follows from Lemma \ref{lemma:ill} and the definition of $\upsilon_{\max}$ and $\upsilon_{\min}$.
\end{proof}
These inequalities do not depend on the orientation of the components of $L_1$ and $L_2$ where the saddle appears. The proof of this statement is given in Lemma \ref{lemma:new}.

\subsection{Bounds for the unoriented slice genus of a link}
Suppose that the $n$-component (unoriented) link $L$ bounds a compact, unoriented surface $F$, with $k$ connected components and Euler number $e(F)$, properly embedded in $D^4$. 
Define the number $\textbf v=\textbf v_1+...+\textbf v_k$ 
as in Figure \ref{Unoriented_surface}; moreover, using the notation in \cite{GL} we write \[\lambda(\vec L):=\sum_{1\leq i<j\leq n}\lk(\vec L_i,\vec L_j)\] for the \emph{total linking number} of $\vec L$ and we take $\overline e_{\vec L}(F):=e(F)-2\lambda(\vec L)$, where $\vec L$ means that we pick an orientation of $L$. We have that $\overline e_{\vec L}(F)=0$ when $F$ is oriented and $\vec L$ inherits its orientation from $F$, see \cite[Section 5]{GL}.
\begin{lemma}
\label{lemma:new}
Suppose that a link $L=\widehat L=\partial F$ as in Figure \ref{Unoriented_surface} is such that $n=k$,  
\begin{figure}[t]
 \centering
 \def\svgwidth{11cm}
\begingroup%
  \makeatletter%
  \providecommand\color[2][]{%
    \errmessage{(Inkscape) Color is used for the text in Inkscape, but the package 'color.sty' is not loaded}%
    \renewcommand\color[2][]{}%
  }%
  \providecommand\transparent[1]{%
    \errmessage{(Inkscape) Transparency is used (non-zero) for the text in Inkscape, but the package 'transparent.sty' is not loaded}%
    \renewcommand\transparent[1]{}%
  }%
  \providecommand\rotatebox[2]{#2}%
  \newcommand*\fsize{\dimexpr\f@size pt\relax}%
  \newcommand*\lineheight[1]{\fontsize{\fsize}{#1\fsize}\selectfont}%
  \ifx\svgwidth\undefined%
    \setlength{\unitlength}{2308.1531507bp}%
    \ifx\svgscale\undefined%
      \relax%
    \else%
      \setlength{\unitlength}{\unitlength * \real{\svgscale}}%
    \fi%
  \else%
    \setlength{\unitlength}{\svgwidth}%
  \fi%
  \global\let\svgwidth\undefined%
  \global\let\svgscale\undefined%
  \makeatother%
  \begin{picture}(1,0.62648069)%
    \lineheight{1}%
    \setlength\tabcolsep{0pt}%
    \put(0,0){\includegraphics[width=\unitlength,page=1]{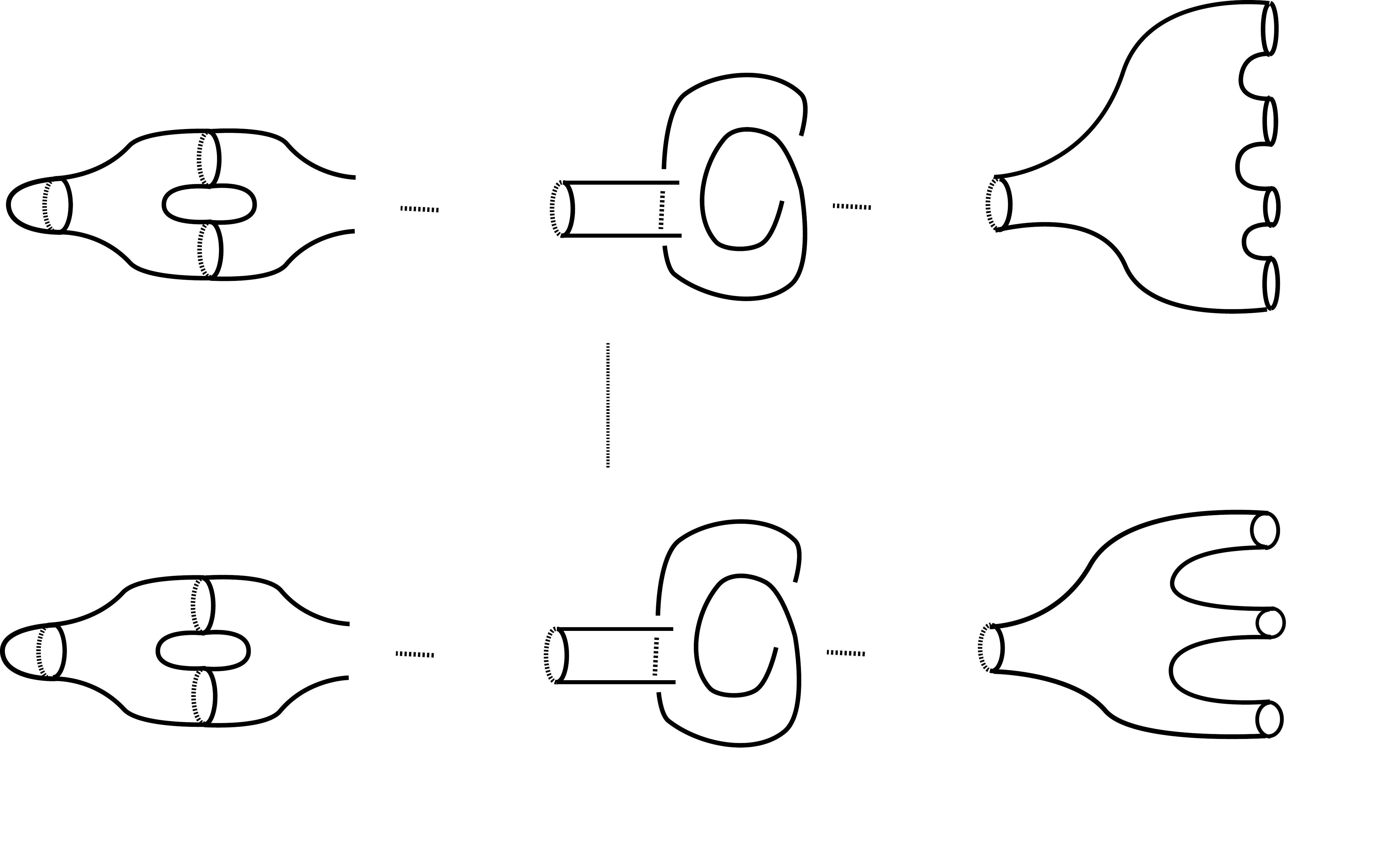}}%
    \put(0.89178743,0.28602507){\color[rgb]{0,0,0}\makebox(0,0)[lt]{\lineheight{1.25}\smash{\begin{tabular}[t]{l}$L$\end{tabular}}}}%
    \put(0.02931682,0.09663432){\color[rgb]{0,0,0}\makebox(0,0)[lt]{\lineheight{1.25}\smash{\begin{tabular}[t]{l}$\bigcirc_k$\end{tabular}}}}%
    \put(0.097089,0.05671373){\color[rgb]{0,0,0}\makebox(0,0)[lt]{\lineheight{1.25}\smash{\begin{tabular}[t]{l}$g$ tori\end{tabular}}}}%
    \put(0.42295249,0.0520718){\color[rgb]{0,0,0}\makebox(0,0)[lt]{\lineheight{1.25}\smash{\begin{tabular}[t]{l}$\textbf v\text{ non-orientable}$\end{tabular}}}}%
    \put(0.49443821,0.00565252){\color[rgb]{0,0,0}\makebox(0,0)[lt]{\lineheight{1.25}\smash{\begin{tabular}[t]{l}$\text{saddles}$\end{tabular}}}}%
    \put(0,0){\includegraphics[width=\unitlength,page=2]{Unoriented_surface.pdf}}%
    \put(0.6949696,0.52554867){\color[rgb]{0,0,0}\makebox(0,0)[lt]{\lineheight{1.25}\smash{\begin{tabular}[t]{l}$\widehat L$\end{tabular}}}}%
  \end{picture}%
\endgroup%
   
 \caption{The number $\textbf v_i$ denotes how many non-orientable saddles there are on each of the $k$ components of $F$. In the picture we omitted the attachment of the extended birth and death moves.}
 \label{Unoriented_surface}
\end{figure}
which means that $F$ is the union of $n$ disjoint unoriented surfaces $F_i$ each one bounding a knot.
Then we have that 
\begin{equation}
\label{eq:unoriented_max}
-g-\dfrac{\emph{\textbf v}}{2}+\dfrac{\overline e_{\vec L}(F)}{4}\leq\upsilon_{\max}(\vec L)\leq g+\dfrac{\emph{\textbf v}}{2}+\dfrac{\overline e_{\vec L}(F)}{4}
\end{equation}
and 
\begin{equation}
\label{eq:unoriented_min}
-g-\dfrac{\emph{\textbf v}}{2}+1-n+\dfrac{\overline e_{\vec L}(F)}{4}\leq\upsilon_{\min}(\vec L)\leq g+\dfrac{\emph{\textbf v}}{2}+1-n+\dfrac{\overline e_{\vec L}(F)}{4}
\end{equation}
for every possible orientation we put on $L$.
\end{lemma}
\begin{proof}
 If $\textbf v=0$ then the claims are true because in this case $\overline e_{\vec L}(F)=e(F)=\lambda(\vec L)=0$ (every orientation on $L$ is compatible with one on $F$) and Proposition \ref{prop:new}.
 Suppose that $\textbf v\geq1$, we prove the last statement first: we assume Equations \eqref{eq:unoriented_max} and \eqref{eq:unoriented_min} are satisfied for one orientation $\vec L$ and we prove them for another one, that we call $\vec L'$. Obviously, we can also suppose that $\vec L'$ is obtained from $\vec L$ by just reversing the orientation on one component of $L$, that we denote by $K$.
 
 From \cite[Corollary 2.7.10]{Book} and Theorem \ref{teo:signature} we have that \[\upsilon(\vec L')=\upsilon(\vec L)+\lk(\vec K,\vec L\setminus\vec K)\] where here $\upsilon$ denotes either $\upsilon_{\max}$ or $\upsilon_{\min}$. Hence, since \[\lambda(\vec L')=\lambda(\vec L\setminus\vec K)-\lk(\vec K,\vec L\setminus\vec K)=\lambda(\vec L)-2\lk(\vec K,\vec L\setminus \vec K)\:,\] we obtain
 \[\dfrac{\overline e_{\vec L'}(F)}{4}=\dfrac{\overline e_{\vec L}(F)}{4}+\lk(\vec K,\vec L\setminus\vec K)\:.\] This means that if we add $\lk(\vec K,\vec L\setminus\vec K)$ to each term in the inequalities in Equations \eqref{eq:unoriented_max} and \eqref{eq:unoriented_min}
 then we obtain precisely the corresponding equations for $\vec L'$; and this part of the proof is complete.
 
 We now prove that the inequalities hold for at least one orientation of $L$. We proceed by induction on $\textbf v$, where the initial step has been done at the beginning of the proof. Therefore, we assume that Equations \eqref{eq:unoriented_max} and \eqref{eq:unoriented_min} hold for $\vec L$ and we prove them for $\vec L'$, where this time $L'$ is
 related to $L$ by a non-orientable saddle move as in Figure \ref{Local_saddle}; denote with $K$ and $K'$ the components of $L$ and $L'$ where the move is performed; and we orient them as in the proof of Lemma \ref{lemma:ill} and Proposition \ref{prop:unoriented_saddle}. 
 
 We show the case of Equation \eqref{eq:unoriented_max}: the argument for Equation \eqref{eq:unoriented_min} is exactly the same. We start by writing \[-g-\dfrac{\textbf v}{2}+\dfrac{\overline e_{\vec L}(F)}{4}\leq\upsilon_{\max}(\vec L)\hspace{1cm}\text{ and }\hspace{1cm}\upsilon_{\max}(\vec L)\leq g+\dfrac{\textbf v}{2}+\dfrac{\overline e_{\vec L}(F)}{4}\]
 from the inductive step; we call $S$ the saddle move and $F'$ the surface obtained by gluing $S$ to $F$, which means $\partial F'=L'$. Then the first inequality becomes
 \[\begin{aligned}-g-\dfrac{\textbf v+1}{2}&+\dfrac{\overline e_{\vec L'}(F')}{4}=-g-\dfrac{\textbf v}{2}+\dfrac{\overline e_{\vec L}(F)}{4}+\left(-\dfrac{1}{2}+\dfrac{e(S)}{4}+\dfrac{1}{2}\lk(\vec K,\vec L\setminus\vec K)-\dfrac{1}{2}\lk(\vec K',\vec L'\setminus\vec K')\right)\leq \\ &\leq\upsilon_{\max}(\vec L)+\left(-\dfrac{1}{2}+\dfrac{e(S)}{4}+\dfrac{1}{2}\lk(\vec K,\vec L\setminus\vec K)-\dfrac{1}{2}\lk(\vec K',\vec L'\setminus\vec K')\right)\leq\upsilon_{\max}(\vec L')\:,\end{aligned}\] where the first equality can be easily computed from the definition of $\overline e$ and the last inequality follows from Proposition \ref{prop:unoriented_saddle}. In the same way, we have 
 \[\begin{aligned}\upsilon_{\max}(\vec L')&\leq\upsilon_{\max}(\vec L)+\left(\dfrac{1}{2}+\dfrac{e(S)}{4}+\dfrac{1}{2}\lk(\vec K,\vec L\setminus\vec K)-\dfrac{1}{2}\lk(\vec K',\vec L'\setminus\vec K')\right)\leq g+\dfrac{\textbf v}{2}+\dfrac{\overline e_{\vec L}(F)}{4}+\\ &+\left(\dfrac{1}{2}+\dfrac{e(S)}{4}+\dfrac{1}{2}\lk(\vec K,\vec L\setminus\vec K)-\dfrac{1}{2}\lk(\vec K',\vec L'\setminus\vec K')\right)\leq g+\dfrac{\textbf v+1}{2}+\dfrac{\overline e_{\vec L'}(F')}{4}\:.\end{aligned}\]
 This concludes the proof because all the terms in Equations \eqref{eq:unoriented_max} and \eqref{eq:unoriented_min} are preserved under concordance; hence, we can ignore extended births and deaths in $F$.
\end{proof}
This lemma allows us to prove Proposition \ref{prop:unoriented_bound}. Suppose that $L$ is a link which bounds an unoriented surface $F$ in $D^4$, with $F_1,...,F_k$ as connected components, as in Figure \ref{Unoriented_surface}. Fix an orientation on $L$, we need to define the integer $\lambda(\vec L,F):=\lambda(\textbf L_1)+...+\lambda(\textbf L_k)$, where $\textbf L_i$ is the oriented sublink of $\vec L$ such that $\textbf L_i=\partial F_i$; note that the orientation on $L_i$ has nothing to do with $F_i$ which may be non-orientable as well. We say that $\lambda(\textbf L_i)=0$ when $\textbf L_i$ is a knot. 

We also write $\widehat L$ for the $k$-component link which appears before the split moves in the decomposition of $F$ in Figure \ref{Unoriented_surface}. Hence, if we denote by $\widehat F\subset F$ the sub-surface such that $\widehat L=\partial \widehat F$ then $\widehat L$ and $\widehat F$ satisfy the hypothesis of Lemma \ref{lemma:new}.
\begin{prop}
 \label{prop:unoriented_bound}
 With the notation established above, the following inequalities are satisfied for the $2^{k}$ orientations of $L$ which are determined by the ones on $\widehat L$:
 \[-g-\dfrac{\vv}{2}+k-n+\dfrac{\overline e_{\vec L}(F)}{4}
 \leq\upsilon_{\max}(L)\leq g+\dfrac{\vv}{2}+\dfrac{\overline e_{\vec L}(F)}{4}
 \] and \[-g-\dfrac{\vv}{2}+1-n+\dfrac{\overline e_{\vec L}(F)}{4}
 \leq\upsilon_{\min}(L)\leq g+\dfrac{\vv}{2}+1-k+\dfrac{\overline e_{\vec L}(F)}{4}
 \:.\]
\end{prop}
\begin{proof}
 We have that \[\lk(\widehat L_i,\widehat L_j)=\sum_{t\in I_i,\:l\in I_j}\lk(\vec L_t,\vec L_l)\] for every $i,j=1,...,k$, where $I_a$ is the set of the components of $L$ in $\textbf L_a$ for $a=1,...,k$. Therefore, one has $\lambda(\widehat L)+\lambda(\vec L,F)=\lambda(\vec L)$.
 We name $F'\subset F$ the cobordism between $\widehat L$ and $L$ and we obtain
 \[\overline e_{\vec L}(F)=e(F)-2\lambda(\vec L)=e(\widehat F)+e(F')-2(\lambda(\widehat L)+\lambda(\vec L,F))=\overline e_{\widehat L}(\widehat F)+(e(F')-2\lambda(\vec L,F))\] and from this, say $F_i'=F'\cap F_i$ is a connected component of $F'$, we argue that \[e(F')-2\lambda(\vec L,F)=\sum_{i=1}^k(e(F_i')-2\lambda(\textbf L_i))\] by definition of Euler number.
 Since each $F_i'$ is oriented and it is a cobordism from $\widehat L_i=\widehat L\cap F_i$ to $\textbf L_i$, we can cap $F_i'$ off in $D^4$ by gluing a compact oriented surface with boundary $\widehat L_i$. In this way, we obtain an oriented surface $G_i$ such that $\partial G_i=\textbf L_i$ and $e(G_i)=e(F'_i)$ for every $i=1,...,k$ and then \[\sum_{i=1}^k(e(F_i')-2\lambda(\textbf L_i))=\sum_{i=1}^k(e(G_i)-2\lambda(\textbf L_i))=\sum_{i=1}^k\overline e_{\textbf L_i}(G_i)=0\] because the orientation on $\textbf L_i$ is induced by the one on $G_i$ (which is the same induced by $F_i'$). 
 
 We have proved that $\overline e_{\vec L}(F)=\overline e_{\widehat L}(\widehat F)$ and now we can apply Lemma \ref{lemma:new} to show that \[-g-\dfrac{\emph{\textbf v}}{2}+\dfrac{\overline e_{\vec L}(F)}{4}\leq\upsilon_{\max}(\widehat L)\leq g+\dfrac{\emph{\textbf v}}{2}+\dfrac{\overline e_{\vec L}(F)}{4}\]
and 
\[-g-\dfrac{\emph{\textbf v}}{2}+1-k+\dfrac{\overline e_{\vec L}(F)}{4}\leq\upsilon_{\min}(\widehat L)\leq g+\dfrac{\emph{\textbf v}}{2}+1-k+\dfrac{\overline e_{\vec L}(F)}{4}\:.\] In order to conclude the proof, we apply Equations \eqref{equation:oriented_max} and \eqref{equation:oriented_min} which tell us that
 \[\upsilon_{\max}(\vec L)\leq\upsilon_{\max}(\widehat L)\leq\upsilon_{\max}(\vec L)+n-k\hspace{1cm}\text{ and }\hspace{1cm} \upsilon_{\min}(\vec L)\leq\upsilon_{\min}(\widehat L)\leq\upsilon_{\min}(\vec L)+n-k\:,\] provided that the orientation on $L$ belongs to the $2^k$ ones induced by an orientation of $F'$.  
\end{proof}
We can use this result to prove that the wideness of the $\upsilon$-set of $L$ gives a lower bound for the \textbf{unoriented slice genus} $\gamma^{(k)}_4(L)$, which is defined as the smallest first Betti number of a surface $F$ as in Figure \ref{Unoriented_surface} and $k$ connected components.
\begin{proof}[Proof of Theorem \ref{teo:wideness}]
 It follows from Proposition \ref{prop:unoriented_bound} because $2g+\textbf v+n-k$ is exactly the first Betti number of $F$.
\end{proof}
Note that Theorem \ref{teo:unoriented_concordance} tells us that $\upsilon_{\max}(L)-\upsilon_{\min}(L)$ is an unoriented concordance invariant of $L$. As a consequence of Theorem \ref{teo:wideness} we obtain Corollary \ref{cor:DO}; see also \cite[Section 5]{DO} for another proof of this result.
\begin{proof}[Proof of Corollary \ref{cor:DO}]
 Suppose that $F$ is the unoriented surface with maximal value of $\chi(F)$ and say it appears like in Figure \ref{Unoriented_surface}. As we saw in the proof of Theorem \ref{teo:wideness}, the first Betti number of $F$ is $2g+\textbf v+n-k$ and then the same theorem implies
 \[k-1\leq2g+\textbf v+n-k\:,\] because for a quasi-alternating link $L$ it is $\upsilon_{\max}(L)=\upsilon_{\min}(L)$ from Theorem \ref{teo:alternating}.
 
 The latter inequality can be rewritten as \[2k-n-2g-\textbf v\leq1\] and it is easy to check that the left-most side is precisely $\chi(F)$.
\end{proof}
In particular, suppose that the quasi-alternating link $L$ has $n$ components and $F$ is the disjoint union of $a$ disks and $n-a$ M\"obius strips. Then $a$ can be at most equal to one.
 
We saw in Theorem \ref{teo:unoriented_concordance} that we can shift $HFL'(\vec L)$ to obtain an unoriented concordance invariant of links. This suggests that we can modify the bounds in Proposition \ref{prop:unoriented_bound} in a way that only unoriented invariants appear. The main tool to achieve this goal is the Gordon-Litherland formula from \cite[Corollary $5^{\prime\prime}$]{GL}:
\begin{equation}
 \label{eq:GL}
 \left|\sigma(\vec L)-\dfrac{\overline e_{\vec L}(F)}{2}\right|\leq\gamma_4^{(k)}(L)
\end{equation}
where $L=\partial F$ and $F=F_1\sqcup...\sqcup F_k$.
\begin{proof}[Proof of Theorem \ref{teo:unoriented_bound}]
 We just need to apply Equation \eqref{eq:GL} to Proposition \ref{prop:unoriented_bound}.
\end{proof} 
Note that the quantities that appear in the left-most side of all the inequalities in Theorem \ref{teo:unoriented_bound}  are unoriented concordance invariants; in particular, they are independent of the choice of the orientation on $L$.

We conclude the paper with a couple of applications, which imply Corollary \ref{cor:app}. First, we compute $\gamma_4^{(2)}(L_n)$ when $L_n$ is the $2$-component link $T_{2,4}^*\#T_{3,4}^{\#n}$.
\begin{figure}
    \centering
    \includegraphics[width=6cm]{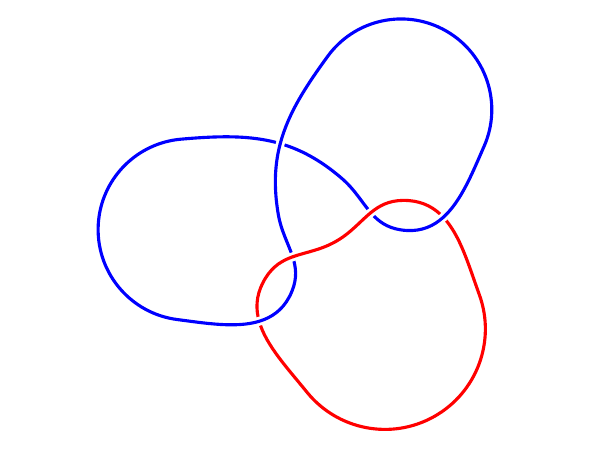}
    \caption{The link $T_{2,4}$: this link becomes the 2-component unlink after the unoriented resolution of the crossing on the blue component.}
    \label{T_2,4}
\end{figure} 
\begin{cor}
 We have that $\gamma_4^{(2)}(L_n)=n+1$ for every $n\geq0$.
\end{cor}
\begin{proof}
Since $T_{2,4}^*$ is non-split alternating we can easily compute $\upsilon_{\min}(T_{2,4}^*)=1$ using Theorem \ref{teo:alternating}, while the fact that $\upsilon(T_{3,4}^{\#n})=-2n$ is known from \cite[Corollary 1.4]{Unoriented}. Moreover, applying Corollary \ref{cor:additive} we obtain that \[\upsilon_{\min}(L_n)=\upsilon_{\min}(T_{2,4}^*)+\upsilon(T_{3,4}^{\#n})=1-2n\:.\] Now we just use Theorem \ref{teo:unoriented_bound} and remember that $\sigma(T_{2,4}^*)=3$ and $\sigma(T_{3,4})=-6$:
\[\left|1-2n-\dfrac{3-6n}{2}+1\right|=\left|n+\dfrac{1}{2}\right|\leq n+1\leq\gamma_4^{(2)}(L_n)\:.\] In order to complete the proof we observe that there is a sequence of $n+1$ non-orientable saddles that change $L_n$ into the unlink $\bigcirc_2$: there is one from $T_{3,4}$ to the unknot and we perform one on each summand, while we can go from $T_{2,4}$ to $\bigcirc_2$ by an unoriented resolution of a crossing, see Figure \ref{T_2,4}.   
\end{proof}
Finally, we show that $\gamma_4^{(1)}(L_n)$ can be arbitrarily large.
\begin{cor}
We have that $\gamma_4^{(1)}(L_n)\geq n$ for every $n\geq0$.
\end{cor}
\begin{proof}
 We use the last inequality in Theorem \ref{teo:unoriented_bound} with $\upsilon_{\min}(L_n)$ and we immediately obtain
 \[\left|1-2n-\dfrac{3-6n-1}{2}\right|=n\leq\gamma_4^{(1)}(L_n)\:.\]
\end{proof}
We point out that these two results were unobtainable if we only used Theorem \ref{teo:wideness}.

\end{document}